\documentclass[10pt]{amsart}
\usepackage{amsmath,amsthm,amssymb,amsfonts,amscd,color}

\newtheorem{pro}{Proposition}[section]
\newtheorem{proposition}[pro]{Proposition}
\newtheorem{theorem}[pro]{Theorem}
\newtheorem{lemma}[pro]{Lemma}

\newtheorem{corollary}[pro]{Corollary}
\newtheorem{remark}[pro]{Remark}

\theoremstyle{remark}

\newtheorem*{Thm*}{Theorem}
\newtheorem*{ThmA}{Theorem A}
\newtheorem*{ThmB}{Theorem B}
\newtheorem*{ProC}{Proposition C}

\newtheorem*{ThmD}{Theorem D}

\newcommand{\tz}{\tiny}

\newcommand{\nz}{\normalsize}
\newcommand{\ti}{\tilde}
\newcommand{\al}{\alpha}
\newcommand{\be}{\beta}

\newcommand{\ie}{{i.e$.$\,}}
\newcommand{\cf}{{cf$.$\,}}

\newcommand{\PP}{{\mathbb P}}
\newcommand{\RR}{{\mathbb R}}

\newcommand{\ZZ}{{\mathbb Z}}

\newcommand{\GL}{\mathop{\rm GL}\nolimits}

\numberwithin{equation}{section}

\begin{document}
\title[Determination of Real Bott Manifolds]
{Determination of Real Bott Manifolds}

\author[A. Nazra]{Admi Nazra}

\address{Department of Mathematics, Andalas University, Kampus Unand Limau Manis
Padang 25163, INDONESIA} \email{admi@fmipa.unand.ac.id}

\curraddr{Department of Mathematics, Tokyo Metropolitan
University, Minami-Ohsawa  1-1, Hachioji, Tokyo 192-0397, JAPAN} \email{nazra-admi@ed.tmu.ac.jp}

\date{\today}

\begin{abstract}
A real Bott manifold is obtained as the orbit space of
the $n$-torus $T^n$ by the free action of an
elementary abelian $2$-group $(\mathbb{Z}_2)^n$. This paper deals
with the classification of $5$-dimensional real Bott manifolds and
study certain type of $n$-dimensional real Bott manifolds ($n\geq 6$).

\end{abstract}
\maketitle

\section*{Introduction}
A real Bott tower is described as a sequence of $\RR\PP^1$-bundles of height $n$
which is the real restriction to a Bott tower introduced in 
\cite{Gross-Kar}. The total space of such a sequence is called a \textit {real Bott manifold}.
From the viewpoint of group actions, an n-dimensional
real Bott manifold
 is the quotient of the n-dimensional torus
$T^n=S^1\times \dots \times S^1$ by the product $(\mathbb{Z}_2)^n$ of
cyclic group of order 2. A {\em Bott matrix} $A$ of
size $n$ is an upper triangular matrix whose diagonal entries are
1 and the other entries are either
1 or 0. By the definition,
the number of distinct Bott matrices of size $n$ is
$\displaystyle 2^{\frac 12{(n^2-n)}}$.
 The free action of $(\mathbb{Z}_2)^n$ on $T^n$ can be expressed by each
row of the Bott matrix $A$ whose orbit space
$M(A)=T^n/(\mathbb{Z}_2)^n$ is the real Bott manifold. It is easy to
see that $M(A)$ is a compact euclidean space form (Riemannian flat
manifold). Then we can apply the Bieberbach theorem \cite{Wolf} to
classify real Bott manifolds. Using this theorem, the classification of
real Bott manifolds up to dimension 4 has been obtained in
\cite{N08}.

 In \cite{K08} we have proved that every $n$-dimensional real Bott manifold
$M(A)$ admits an injective Seifert fibred structure
which has the form $M(A)=T^k\times_{(\ZZ_2)^s}M(B)$, that is there
is a $k$-torus action on $M(A)$ whose quotient space is an
$(n-k)$-dimensional real Bott orbifold $M(B)/(\ZZ_2)^s$ by
some $(\ZZ_2)^s$-action $(1\leq s\leq k)$. Moreover we have
proved the smooth rigidity that
two real Bott manifolds $M(A_i)$ $i=1,2$ are diffeomorphic if and only
if the corresponding actions $((\mathbb Z_2)^{s_i},M(B_i))$ are
equivariantly diffeomorphic. When the low dimensional real Bott
manifolds with $(\mathbb Z_2)^{s}$-actions are
classified, we can determine the
diffeomorphism classes of higher dimensional ones by the
above rigidity. We have classified
real Bott manifolds until dimension 4 (see \cite{N08a}).

The main purpose of this paper is to determine: $(a)$ Diffeomorphism
classes of $5$-dimensional real Bott manifolds from the classifications
of $2,3,4$-dimensional ones with $({\mathbb
Z_2})^s$-actions $(s=1,2)$, $(b)$ Classification of certain type of
$n$-dimensional real Bott manifolds $M(A)$.   

We have obtained the following to $(a)$, (Compare Theorem \ref{dim-5}.)
\begin{ThmA}
There are $54$ diffeomorphism classes of
$5$-dimensional real Bott manifolds.
 \end{ThmA}
Since each $n$-dimensional real Bott manifold has
the form $\displaystyle M(A)=T^k\times_{(\ZZ_2)^s}M(B)$,
it is shown that the diffeomorphism class of $M(A)$ is determined by
the equivariant diffeomorphism class of the action
$((\ZZ_2)^s,M(B))$ $(1\leq s\leq k)$ as above.
We prove that if $k=1$, there are $12$-nonequivariant
diffeomorphism classes 
of $4$-dimensional real Bott manifolds with $\ZZ_2$-actions.
Then they create $29$-diffeomorphi-sm classes of such $5$-dimensional real Bott manifolds.
When $k=2$, there are $4$-nonequivariant
diffeomorphism classes 
of $3$-dimensional real Bott manifolds with $(\ZZ_2)^s$-actions $(s=1,2)$. 
Then from these, there are $19$-diffeomorphism classes of 
the $5$-dimensional real Bott manifolds.
When $k=3$, there are $2$-nonequivariant
diffeomorphism classes 
of $2$-dimensional real Bott manifolds with $(\ZZ_2)^s$-actions $(s=1,2)$. 
Then there are $4$-diffeomorphism classes of 
the $5$-dimensional real Bott manifolds.
When $k=4$, the $1$-dimensional real Bott manifold
is $S^1$ with conjugate action of $\ZZ_2$, there
exists only one  such a $5$-dimensional
real Bott manifold. Finally if $k=5$, the $5$-dimensional real Bott
manifold is $T^5$. As a consequence, the total number of $5$-dimensional
diffeomorphism classes is $54$. The details of the proof is in section \ref{S2}.

It is far to determine the number of diffeomorphism 
classes of $n$-dimensional real Bott manifolds for $n\geq 6$. However,
we shall solve the special types of higher dimensional real Bott manifolds.
\begin{ThmB}
Let $T^k$ be the maximal torus action on an
$n$-dimensional real Bott manifold $(n\geq 4)$. If $k\geq n-2$, then the
diffeomorphism classes of such real Bott manifolds consists of $6$.
\end{ThmB}
See Corollary \ref{cor1} for the proof.

\begin{ProC}
The following hold.
\begin{itemize}
\item[(i)]
 Let $M(A)$ be a real Bott manifold which fibers $S^1$
over the real Bott manifold $M(B)$ for which $M(B)$ is either
\[
T^k\times_{(\ZZ_2)^s}T^2\ \mbox{or}\ T^k\times_{(\ZZ_2)^s}{\rm K} \
(k\geq 2).
\]Here ${\rm K}$ is a Klein bottle.
Then the diffeomorphism classes of such $M(A)$ is $3$.

\item[(ii)]
 Let $M(A)$ be a 
 real Bott manifold which fibers $S^1$
over the real Bott manifold $M(B)$ where
\[
M(B)=S^1\times_{\ZZ_2}T^{t} \ (t\geq 2), \] then the diffeomorphism
classes of such $M(A)$ is $[\frac{t}{2}]+1$. Here $[x]$ is the Gauss
integer.

\item[(iii)]
The diffeomorphism class is unique for the real
Bott manifold of the form $\displaystyle
M(A)=T^k\mathop{\times}_{\ZZ_2}^{}T^{n-k}$ for any $n\geq 2$ and $k\geq 1$.
\end{itemize}
\end{ProC}
We prove these results in Section 4 (see Corollary \ref{Torus},
Corollary \ref{KBottle}, Theorem \ref{greatest}, Proposition
\ref{prop1} respectively).

A special kind of  Bott matrices is introduced in
Section 1. We consider such a
class of  Bott matrices in \eqref{bottB}. 

\begin{ThmD}
Let $M(A)=S^1\times_{\ZZ_2}M(B)$ be an $n$-dimensional real Bott
manifold. Suppose that $B$ is either one of the list in
\eqref{bottB}. Then $M(B)$ are diffeomorphic each other and the
number of diffeomorphism classes of such real Bott manifolds $M(A)$
above is $(k+1)2^{n-k-3}$ ($k\geq 2$, $n-k\geq 3$).
\end{ThmD}
\noindent See Theorem \ref{type1..1k} for the proof.

\section{Preliminaries}\label{S1}

\subsection{Seifert fiber space}\label{sec:Seifert}
Each  $i$-th row of a Bott matrix $A$ defines a $\ZZ_2$-action
on $T^n$ by
\[
g_i(z_1,\dots,z_n)=(z_1,\dots, z_{i-1}, -z_i, \tilde
z_{i+1},\dots, \tilde z_n), \,\, (i=1,\dots,n)
\]
where $(i,i)$-(diagonal) entry  $1$ acts as $z_i \to -z_i$ while
${\tilde z_j}$ is either $z_j$ or $\bar z_j$ depending on whether
$(i,j)$-entry $(i<j)$ is $0$ or $1$ respectively. Note that $\bar z$
is the conjugate of the complex number $z\in S^1$. It is always
trivial; $z_j\to z_j$ whenever $j<i$. Here $(z_1,\dots,z_n)$ are the
standard coordinates of the $n$-dimensional torus $T^n$
whose universal covering is the $n$-dimensional euclidean
space $\RR^n$. The projection $p\colon\RR^n\to T^n$ is denoted by
\[p(x_1,\dots,x_n)=(exp(2\pi{\bf i}x_1),\dots,exp(2\pi{\bf i}x_n))=(z_1,\dots,z_n). \]
Those
$\langle g_1,\dots,g_n\rangle$ constitute the generators of
$(\ZZ_2)^n$. It is easy to see that $(\ZZ_2)^n$ acts freely on $T^n$
such that the orbit space $M(A)=T^n/(\ZZ_2)^n$ is a smooth compact
manifold. In this way, given a Bott matrix $A$ of size $n$, we
obtain a free action of $(\ZZ_2)^n$ on $T^n$.

Let $\pi(A)=\langle \tilde g_1,\dots,\tilde g_n\rangle$ be the lift
of $(\ZZ_2)^n=\langle g_1,\dots,g_n\rangle$ to $\RR^n$. Then we get
$$\tilde g_i(x_1,\dots,x_n)=(x_1,\dots,x_{i-1},\frac{1}{2}+x_i,\tilde x_{i+1},\dots,\tilde x_n)$$
where $\tilde x_j$ is either $x_j$ or $-x_j$.
It is easy to see that $\pi(A)$ acts properly discontinuously and freely on $\RR^n$ as euclidean motions.
Note that $\pi(A)$ is a Bieberbach group which is a discrete uniform subgroup of the
euclidean group E$(n)=\RR^n\rtimes$O$(n)$ (cf.\cite{Wolf}).
It follows that
$$M(A)=T^n/(\ZZ_2)^n=\RR^n/\pi(A).$$

Now let us recall moves {\bf I}, {\bf II} and {\bf III} \cite{K08}
to a Bott matrix $A$ of size n under which the diffeomorphism class of $M(A)$
does not change.\\
{\bf I}. If the $j$-th column has all $0$-entries except for the
$(j,j)$-entry $1$ for some $j>1$,
 then interchange the $j$-th column and the $(j-1)$-th column. Next, interchange
 the $j$-th row and  the $(j-1)$-th row.\\
This move {\bf I} is interpreted in terms of the coordinates $z_j$'s
of $T^n$ and the generators $g_j$'s of $(\ZZ_2)^n$ as follows:
\begin{equation*}\begin{split}
z_j\to z_{j-1}',\, z_{j-1}\to z_j',\ \ g_j\to g_{j-1}',\, g_{j-1}\to
g_j'. \end{split}
\end{equation*} It
is easy to see that the resulting matrix $A'$ under move {\bf I} is
again a Bott matrix such that $M(A)$ is diffeomorphic to $M(A')$.\\
We perform move {\bf I} iteratively to get a Bott matrix $A'$\\
\begin{equation}\label{formA1}
A'= \left (\begin{array}{c|c}
 I_k  & C \\
 \hline
0  & B
\end{array}\right)\,\,\,
B= \left (\begin{array}{ccc}
1&   &  \mbox{\Large$\ast $}  \\
 & \ddots &   \\
 &   & 1
\end{array}\right)
\end{equation}
where $I_{k}$ is a maximal block of identity matrix of size $k$,
 the entries of the \mbox{\Large$\ast $} are either 1 or 0,
$B$ is  a Bott matrix of size $(n-k)$ which
represents
a real Bott manifold $M(B)=T^{n-k}/(\ZZ_2)^{n-k}$.
Since $I_k$ is a maximal block of identity matrix, each $k+j$ \,$(j=1,\dots,n-k)$-th column
of $A'$ has at least
two non zero elements.

Associated with  $A'$, the $(\ZZ_2)^{n}$-action splits
into $(\ZZ_2)^{k}\times (\ZZ_2)^{n-k}$
and $T^n$ splits into $T^k\times T^{n-k}$. Hence
\begin{align}\label{MMM}
M(A)=T^n/(\ZZ_2)^n\cong \frac{T^k\times T^{n-k}}{(\ZZ_2)^{k}\times (\ZZ_2)^{n-k}}
=T^k\mathop{\times}_{(\ZZ_2)^k} M(B)=M(A').
\end{align}
Note that the above $(\ZZ_2)^k$-action of (\ref{MMM}) is not necessarily
effective on $M(B)$ but we can reduce it to the effective
$(\ZZ_2)^s$-action on $M(B)$ for some $s$ $(1\leq s\leq k)$.
In order to do so, we have two more moves.\\
{\bf II}. If there is an $m$-th row $(1\leq m\leq k)$  whose
entries in C are all zero,
then divide $T^k \times M(B)$ by the corresponding $\ZZ_2$-action. 
 For example, suppose $M(A_1)=T^2\times_{(\ZZ_2)^2}M(B)$ with
$$A_1=
\left(\begin{array}{cc|ccc}
1& 0 & 0 & 0 & 0 \\
0& 1 & 1 & 0 & 0  \\
\hline
0& 0 & 1 & 1 & 0\\
0& 0 & 0 & 1 & 1\\
0& 0 & 0 & 0 & 1
\end{array}\right).$$
By move {\bf II},  $M(A_1)\cong T^2\times_{\ZZ_2}M(B)$.
\\
{\bf III}. If the $p$-th row and
$\ell$-th row $(1\leq p<\ell\leq k)$ have the common entries in $C$,
then compose the $\ZZ_2$-action of $p$-th row with $l$-th row and
divide $T^k \times M(B)$ by this $\ZZ_2$-action.
 For example, suppose $M(A_2)=T^2\times_{(\ZZ_2)^2}M(B)$ with
$$A_2=
\left(\begin{array}{cc|ccc}
1& 0 & 1 & 0 & 0 \\
0& 1 & 1 & 0 & 0  \\
\hline
0& 0 & 1 & 1 & 0\\
0& 0 & 0 & 1 & 1\\
0& 0 & 0 & 0 & 1
\end{array}\right).$$
By move {\bf III}, $M(A_2)\cong T^2\times_{\ZZ_2}M(B)\cong M(A_1)$.

By an iteration of {\bf II, III}, the quotient is again
diffeomorphic to $T^k \times M(B)$ but eventually  the
$(\ZZ_2)^k$-action is reduced to the effective $(\ZZ_2)^s$-action
on $T^k \times M(B)$. Therefore $A'$ reduces to
\begin{equation}\label{newformA1}
A''=  \left(\begin{array}{c|c|c}
\mbox{\large$I_{{\small k-s}}$}   & \mbox{\Large$0$}  & \mbox{\Large$0$}      \\
\hline
  &    &    \\
\mbox{\Large$0$}   &\mbox{\large$I_s$}       &\mbox{\Large$*$}     \\
   &   &          \\
\hline
   &   &         \\
\mbox{\Large$0$}    &  \mbox{\Large$0$}  &   \mbox{\large$B$}
\end{array}\right)
\end{equation} in which
\begin{align*}
M(A')&=T^k\mathop{\times}_{(\ZZ_2)^k}^{} M(B)\\
     &=\frac{T^{k-s}\times T^s\times
     M(B)}{(\ZZ_2)^{k-s}\times(\ZZ_2)^s}
     =M(A'').
\end{align*}
Since $(\ZZ_2)^{k-s}$ acts trivially on $T^s\times M(B)$ then we have
\[M(A'')\cong T^k\mathop{\times}_{(\ZZ_2)^s}^{} M(B).\]

From now on, we write $M(A)$ instead of $M(A'')$.

\begin{remark}
Since $(\ZZ_2)^s$ acts trivially on $T^{k-s}$,
\begin{align*}
M(A)&\cong T^k\mathop{\times}_{(\ZZ_2)^s} M(B)
    =T^{k-s} \times T^s \mathop{\times}_{(\ZZ_2)^s} M(B)\\
    &\cong  (S^1)^{k-s} \times T^s \mathop{\times}_{(\ZZ_2)^s} M(B)\\
    &=(S^1)^{k-s} \times M(B')\\
\end{align*}
where $M(B')=T^s \mathop{\times}_{(\ZZ_2)^s} M(B)$.
That is, for $s<k$, a real Bott manifold $M(A)$ is the product of $(S^1)^{k-s}$
 and an $(n-k+s)$-dimensional real Bott manifold $M(B')$. 
In particular,
if  $M(A)=T^{n-1}\mathop{\times}_{\ZZ_2} S^1$
then it is diffeomorphic to $(S^1)^{n-2} \times$ Klein bottle.
\end{remark}

\begin{remark}
From the submatrix \, $\mbox{\Large$\ast $}$ of $(\ref{newformA1})$, the group
$(\ZZ_2)^s=\langle g_{k-s+1},\dots,$ $g_k\rangle$ acts on
$T^k\times M(B)$ by
\begin{align}\label{Zs}
\begin{split}
&g_i(z_1,\dots,z_{k-s+1},\dots,z_k,[z_{k+1},\dots,z_n])\\&=(z_1,\dots,z_{k-s+1},\dots,-z_i,\dots,z_k,[\ti
z_{k+1},\dots,\ti z_n])
\end{split}
\end{align}
where $\ti z=\bar z$ or $z$. So there induces an action of
$(\ZZ_2)^s$ on $M(B)$ by
\begin{align}\label{inZs}
g_i([z_{k+1},\dots,z_n])=[\ti z_{k+1},\dots,\ti z_n].
\end{align}
\end{remark}

Moreover in \cite{K08},

\begin{theorem}[Structure]\label{T1}
Given a real Bott manifold $M(A)$, there exists a maximal $T^k$-action
$(k\geq 1)$ such that \[M(A)=T^k\mathop{\times}_{(\ZZ_2)^s}^{}
M(B)\] is an injective Seifert fiber space over the
$(n-k)$-dimensional real Bott orbifold $M(B)/(\ZZ_2)^s$;
\begin{equation}\label{injectiveSeifert}
T^k \to M(A)\to M(B)/(\ZZ_2)^s. \end{equation}
There is a central extension of the fundamental group $\pi(A)$ of $M(A)$:
\begin{equation}\label{Seifertgroup}
1\to \ZZ^k \to \pi(A)\to Q_B\to 1 \end{equation}such that
\begin{itemize}
\item[(i)] $\ZZ^k$ is the maximal central free abelian subgroup
\item[(ii)] The induced group $Q_B$ is  the semidirect product $\pi(B)\rtimes (\ZZ_2)^s$
        for which $\RR^{n-k}/\pi(B)=M(B)$.
\end{itemize}
\end{theorem}
\noindent See \cite{K08} for the proof.\\

By this theorem, a real Bott manifold $M(A)$ which admits a maximal
$T^k$-action $(k\geq 1)$ can be created from an $(n-k)$-dimensional
real Bott manifold $M(B)$ by a $(\ZZ_2)^s$-action, and the
corresponding Bott matrix $A$ has the form as in (\ref{newformA1})
above.

\subsection{Affine maps between real Bott manifolds}\label{sec:Fixed-charac}

Next, we can apply the following theorem to check whether two
real Bott manifolds are diffeomorphic.

\begin{theorem}[Rigidity]\label{T2}
Let $M(A_1)$, $M(A_2)$ be $n$-dimensional real Bott manifolds and
$\displaystyle 1\to \ZZ^{k_i} \to \pi(A_i)\to Q_{B_i}\to 1$ be
the associated group extensions $(i=1,2)$. Then the following are
equivalent:
\begin{itemize}
\item[(i)] $\pi(A_1)$ is isomorphic to $\pi(A_2)$.
\item[(ii)] There exists an isomorphism of $Q_{B_1}=\pi(B_1)\rtimes (\ZZ_2)^{s_1}$ onto
            $Q_{B_2}=\pi(B_2)\rtimes (\ZZ_2)^{s_2}$ preserving $\pi(B_1)$ and $\pi(B_2)$.
\item[(iii)] The action $((\ZZ_2)^{s_1}, M(B_1))$ is equivariantly diffeomorphic to the action 
$((\ZZ_2)^{s_2}, M(B_2))$.
\end{itemize}
\end{theorem}
\noindent See \cite{K08} for the proof.
Here Bott matrices $A_1$ and $A_2$ are created from $B_1$ and $B_2$ respectively.\\

Note that two real Bott manifolds $M(A_1)$ and
$M(A_2)$
 are diffeomorphic if and only if $\pi(A_1)$ is isomorphic to $\pi(A_2)$
 by the Bieberbach theorem \cite{Wolf}. 
Moreover  by  Theorem \ref{T1} and  \ref{T2} we have,
\begin{remark}\label{rem15}
Let real Bott manifolds $M(A_i)=T^{k_i}\times_{(\ZZ_2)^{s_i}}M(B_i)$ $(i=1,2)$.
If $M(A_1)$ and $M(A_2)$ are diffeomorphic
 then the following hold.
\begin{itemize}
\item[(i)] $k_1=k_2. $
\item[(ii)]  $  M(B_1)$ and $M(B_2)$ are diffeomorphic.
\item[(iii)] $s_1=s_2$. 
\end{itemize}
\end{remark}
Therefore two real Bott manifolds which admit different maximal
$T^k$-action are not diffeomorphic.  
If they have the same maximal
$T^k$-action, then the quotients $((\ZZ_2)^{s_i},M(B_i))$ are
compared. 
If $M(B_1)$ is not diffeomorphic to $M(B_2)$ or $s_1\neq
s_2$, then $M(A_1)$ and $M(A_2)$ are not diffeomorphic.  
So our
task is to distinguish the $(\ZZ_2)^{s_i}$-action on $M(B_i)$ when
it is the case that $s_1=s_2=s$ and $M(B_1)$ is diffeomorphic to
$M(B_2)$.

\subsection{Type of fixed point set}\label{sec:Fixed type}

Note that from (\ref{inZs}), the action of $(\ZZ_2)^s$ on $M(B)$ is
defined by
\[\al[(z_1,\dots,z_{n-k})]=[\al(z_1,\dots,z_{n-k})]=[(\ti z_1,\dots,\ti
z_{n-k})]\]
for $\al\in(\ZZ_2)^s$ and $\ti z=z$ or $\bar z$. Since
$M(B)=T^{n-k}/(\ZZ_2)^{n-k}$, the action $\langle \al\rangle$ lifts
to a linear (affine) action on $T^{n-k}$ naturally:
\[\al(z_1,\dots,z_{n-k})=(\ti z_1,\dots,\ti z_{n-k}).\]
Then the fixed point set is characterized by the equation:
\[(\ti z_1,\dots,\ti z_{n-k})=g(z_1,\dots,z_{n-k})\]
for some $g\in(\ZZ_2)^{n-k}$. 
It is also an affine subspace of $T^{n-k}$. 
So the fixed point sets of $(\ZZ_2)^s$
are affine subspaces in $M(B)$.

Let $B$ be the Bott matrix as in \eqref{formA1}. By a repetition of
move {\bf I},  $B$ has the form
\begin{align}\label{matrixB}
B= \left(\begin{array}{ccccc}
I_2 & C_{23} & \dots & \dots & C_{2\ell} \\
    & I_3    & C_{34}& \dots & C_{3\ell} \\
    &        & \ddots&       & \dots     \\
    &  $\mbox{\Large$0 $}$      &       &  I_{\ell-1} & C_{(\ell-1)\ell} \\
    &        &       &             & I_{\ell}
\end{array}\right)
\end{align}
where rank$B$=n-k=rank$I_2$+\dots+rank$I_\ell$
and $I_i$ ($i=2,\dots,\ell$) is the identity matrix,
$C_{jt}$ ($j=2,\dots,\ell-1$, $t=3,\dots,\ell$) is a $p_j\times q_t$ matrix ($p_j$=rank$I_j$, $q_t$=rank$I_t)$.

Note that by the Bieberbach theorem (cf. \cite{Wolf}),
if $f$ is an isomorphism of $\pi(A_1)$ onto $\pi(A_2)$, then
 there exists an affine
element $g=(h,H)\in$A$(n)=\RR^n\rtimes \GL(n,\RR)$ such that
\begin{align}\label{conj}
f(r)=g r g^{-1} \,\,(\forall r\in\pi(A_1)).
\end{align}
Recall that if $M(A_1)$ is diffeomorphic to $M(A_2)$ then $M(B_1)$ is
diffeomorphic to $M(B_2)$. This implies that  $B_1$ and $B_2$
have the form as in \eqref{matrixB}.

Using  \eqref{conj} and according to the form of  $B$ in \eqref{matrixB}
we obtain that
\begin{align}\label{matrixAE}
g= \left(
\left( \begin{array}{c}
{\bf h}_1\\
{\bf h}_2\\
\vdots\\
{\bf h}_\ell
\end{array}\right),
\left(\begin{array}{cccc}
H_1 &     &  &  \\
    & H_2 &  & $\mbox{\Large$0 $}$ \\
 $\mbox{\Large$0 $}$    &     & \ddots &     \\
    &     &        &  H_{\ell}
\end{array}\right)\right)
\end{align}
where ${\bf h}_i$ is an $s_i\times1$ ($s_i$=rank $I_i$) column matrix
(${\bf h}_1$ is a $k\times1$ column matrix), $H_i\in$GL($s_i,\RR$) ($i=2,\dots,\ell$),
$H_1\in$GL$(k,\RR$)
(see Remark 3.2 \cite{K08}).

Let $\bar f\colon Q_{B_1}\to Q_{B_2}$ be the induced isomorphism from $f$ (\cf
Theorem \ref{T2}).
Now the affine equivalence
$\bar g\colon \RR^{n-k}\to \RR^{n-k}$ has the form
\begin{align}\label{finalform}
\bar  g= \left ( \left ( \begin{array}{c}
{\bf h}_2\\
\vdots\\
{\bf h}_\ell
\end{array}\right),
\left ( \begin{array}{ccc}
  H_2&         &$\mbox{\Large$0 $}$\\
     & \ddots &\\
  $\mbox{\Large$0 $}$   &         &H_\ell
\end{array}\right )\right)
\end{align}which is equivariant with respect to $\bar f$.
The pair $(\bar f,\bar g)$ induces an equivariant affine diffeomorphism
$(\hat f, \hat g)\colon ((\ZZ_2)^s,M(B_1))\to ((\ZZ_2)^s,$ $M(B_2)).$

Let ${\rm rank} H_i=b_i$ ($i=2,\dots,\ell$).
(Note that $b_2+\dots+b_\ell=n-k$.) 
Since
$\displaystyle M(B_1)=T^{n-k}/(\ZZ_2)^{n-k}$, $\bar g$ induces an affine map $\tilde g$ of
$T^{n-k}$. Put 
$$
X_{b_2}=
\begin{pmatrix}
x_1\\ \vdots \\ x_{b_2}
\end{pmatrix}, \dots,
X_{b_\ell}=
\begin{pmatrix}
x_{b_{\ell'}+1}\\ \vdots \\ x_{b_{\ell'}+b_\ell}
\end{pmatrix},
w_{b_i}=p(X_{b_i})\in T^{b_i}\, (i=2,\dots,\ell), \,
$$ $$
b_{\ell'}=b_2+\dots+b_{\ell-1}.
$$
Since $\tilde gp=p\bar g$, 
\begin{align}\label{preserve}
\tilde g(^tw_{b_2},\dots,{^tw_{b_\ell}})=(^tw'_{b_2},\dots,{^tw'_{b_\ell}})
\end{align}
where
$w'_{b_i}=p({\bf h}_i+H_iX_{b_i})\in T^{b_i}$.
That is, $\tilde g$ preserves each $T^{b_i}$ of $T^{n-k}= T^{b_2}\times \cdots\times T^{b_\ell}$,
so does 
$\hat g$ on
$$ M(B_1)=\lbrace{[z_1,\dots,z_{b_2};z_{b_2+1},\dots,z_{b_2+b_3};\dots \dots;
z_{b_{\ell'}+1},\dots,z_{b_{\ell'}+b_{\ell}}]\rbrace}.$$
We say that 
$\hat g$ \emph{preserves the type} $(b_2,\dots,b_{\ell})$ of
$M(B_1)$. As $\hat g$ is $\hat f$-equivariant, it also preserves the 
type corresponding to the fixed point sets between $((\ZZ_2)^s, M(B_1))$
and $((\ZZ_2)^s, M(B_2))$.

\begin{pro}\label{T5}
The $(\ZZ_2)^{s}$-action on $M(B)$ is distinguished by the number of
components and types of each positive dimensional fixed point
subsets.
\end{pro}
\noindent See \cite{K08} for the proof.\\

\noindent {\bf Definition.} We say that two Bott matrices $A$ and
$A'$ are \textit {equivalent} (denoted by $A \sim A'$) if   $M(A)$
and $M(A')$ are diffeomorphic.

\section{Examples}\label{EXAMPEL}

We shall give
some real Bott manifolds in order to determine
 diffeomorphism classes of $5$-dimensional ones.
We introduce the following Bott matrices created from
$B=
\left(\begin{array}{ccc}
1 & 1 & 0\\
0 & 1 & 1\\
0 & 0 & 1
\end{array}\right)$. 
\begin{align*}
&A_1=
\left(\begin{array}{cc|ccc}
1& 0 & 1 & 1 & 0 \\
0& 1 & 0 & 1 & 0  \\
\hline
$\mbox{\Large$0 $}$& $\mbox{\Large$0 $}$ &  & $\mbox{\large$B $}$ & 
\end{array}\right),\,
A_2=
\left(\begin{array}{cc|ccc}
1& 0 & 0 & 1 & 0 \\
0& 1 & 1 & 0 & 0  \\
\hline
$\mbox{\Large$0 $}$& $\mbox{\Large$0 $}$ &  & $\mbox{\large$B $}$ & 
\end{array}\right),\\
&A_3=
\left(\begin{array}{cc|ccc}
1& 0 & 1 & 0 & 0 \\
0& 1 & 0 & 0 & 1  \\
\hline
$\mbox{\Large$0 $}$& $\mbox{\Large$0 $}$ &  & $\mbox{\large$B $}$ & 
\end{array}\right),\,
A_4=
\left(\begin{array}{cc|ccc}
1& 0 & 1 & 0 & 1 \\
0& 1 & 1 & 0 & 0  \\
\hline
$\mbox{\Large$0 $}$& $\mbox{\Large$0 $}$ &  & $\mbox{\large$B $}$ & 
\end{array}\right).
\end{align*}
Then we obtain the $5$-dimensional real Bott manifolds $M(A_i)$ for which
the $(\ZZ_2)^2$-action on $M(B)$
is given by the first two rows of $A_i$
$(i=1,2,3,4)$.
We prove that there are two distinct diffeomorhism classes 
among $M(A_i)$ $(i=1,2,3,4)$.
\begin{itemize}
\item[\bf a)] $M(A_1)$ is diffeomorphic to $M(A_2)$. 
        For this, the $(\ZZ_2)^2$-actions on $M(B)$ corresponding to 
          $A_1$ and $A_2$ are given as follows: 
  \begin{itemize}
  \item[(i)]    $ g_1([z_3,z_4,z_5])=[\bar z_3,\bar z_4,z_5]=[g_3(\bar z_3,\bar z_4,z_5)]=[-\bar z_3,z_4,z_5],\\
            g_2([z_3,z_4,z_5])=[z_3,\bar z_4,z_5],$
  \item[(ii)]   $ h_1([z_3,z_4,z_5])=[z_3,\bar z_4,z_5],\,h_2([z_3,z_4,z_5])=[\bar z_3,z_4,z_5]$.
  \end{itemize}
  There is an equivariant diffeomorphism
       $\varphi \colon((\ZZ_2)^2, M(B))\to ((\ZZ_2)^2,$ $ M(B))$  defined by
            $\varphi([z_3,z_4,z_5])=([{\bf i}z_3,z_4,z_5]) $
            such that $\varphi g_1=h_2\varphi $ and $\varphi g_2=h_1\varphi $.
        Hence the result follows from Theorem \ref{T2}.
\item[\bf b)] $M(A_2)$ is not diffeomorphic to $M(A_3)$.
If $M(A_2)$ and $M(A_3)$ are diffeomorphic,
by Theorem $\ref{T2}$ there is an equivariant diffeomorphism
$\varphi\colon((\ZZ_2)^{2},$ $M(B))\to ((\ZZ_2)^{2},M(B))$.
Let $\bar \varphi:\RR^{3}\rightarrow \RR^{3}$ be the lift of $\varphi$.
According to the form of $B$, the affine element $\bar \varphi$
has the form
\begin{align}\label{aff}
\bar \varphi=\left ( \left ( \begin{array}{c}
a_2\\
a_3\\
a_4
\end{array}\right ),
\left ( \begin{array}{ccc}
        1  & 0 & 0  \\
        0  & 1 & 0  \\
        0  & 0 & 1
\end{array}\right )\right)
\end{align}
for some $a_i\in\RR$ $(i=2,3,4)$ $($see \eqref{finalform}$)$.
Since $M(B)=T^{3}/(\ZZ_2)^{3}$, $\bar \varphi$ induces
an affine map of $T^{3}$. By the formula of $(\ref{aff})$, it preserves
each $S^1$ of $T^{3}= S^1\times S^1\times S^1$, so does $\varphi$
on $M(B)$. 
Since $\varphi$ is equivariant, it also preserves the type $(1,1,1)$
of the fixed point sets of $((\ZZ_2)^{2},M(B))$.
That is, if $[z_3,z_4,z_5]$ is a fixed point set of $((\ZZ_2)^2,M(B))$, then
$\varphi$ preserves each coordinate $z_i$ $(i=3,4,5)$ 
$($i.e.,  
$\varphi [z_3,z_4,z_5]=[exp(2\pi {\bf i}a_2)z_3,$ $exp(2\pi{\bf i} a_3)z_4,exp(2\pi{\bf i} a_4)z_5])$. 

The fixed point sets of $((\ZZ_2)^{2},M(B))$ corresponding
            to $A_2$ and $A_3$ are as follows:
 \begin{itemize}
   \item[(i)] $3$ components $T^2=\{[z_3,1,z_5],[1,z_4,z_5],
                                 [{\bf i},z_4,z_5]\}$, \\
            $4$ components   $S^1=\{[z_3,{\bf i},1],[{\bf i},1,z_5],
                                 [z_3,{\bf i},{\bf i}],[1,1,z_5]\}$,\\
             $4$ points $\{[{\bf i},{\bf i},1],[{\bf i},{\bf i},{\bf i}, [1,{\bf i},1],[1,{\bf i},{\bf i}]\}$,
   \item[(ii)] $3$ components $T^2=\{[z_3,z_4,{\bf i}],[1,z_4,z_5],
                                 [z_3,z_4,1]\}$, \\
            $4$ components $S^1=\{[1,z_4,1],[{\bf i},1,z_5],
                                 [1,z_4,{\bf i}],[{\bf i},{\bf i},z_5]\}$, \\ 
                                 and $4$ points 
                                 $\{[{\bf i},{\bf i},1],[{\bf i},{\bf i},{\bf i}, [{\bf i},1,1],[{\bf i},1,{\bf i}]\}$.
 \end{itemize}
 We see that the number of components of fixed point sets of $((\ZZ_2)^{2},$ $M(B))$
 corresponding to $A_2$ and $A_3$ is the same.   
Since the type of fixed point set is preserved, $\varphi$
maps $T^2=\{[z_3,1,z_5]\}$,  $z_3,z_5\in S^1$  ($($i$)$ in $A_2$) onto the fixed point set 
$T^2=\{[w_3,exp(2\pi{\bf i} a_3),w_5]\}$ $(w_3,w_5\in S^1)$ of $A_3$. 
However there is no type of such fixed point set in $($ii$)$ of $A_3$.  
Therefore by Proposition \ref{T5},  $M(A_2)$ and $M(A_3)$ are not diffeomorphic.

\item[\bf c)] $M(A_3)$ is diffeomorphic to $M(A_4)$.  In this case,
the $(\ZZ_2)^2$-actions on $M(B)$ corresponding to 
          $A_3$ and $A_4$ are given as follows:
  \begin{itemize}
  \item[(i)]    $ g_1([z_3,z_4,z_5])=[\bar z_3,z_4,z_5],\,
            g_2([z_3,z_4,z_5])=[z_3, z_4,\bar z_5],$
  \item[(ii)]   $ h_1([z_3,z_4,z_5])=[\bar z_3, z_4, \bar z_5],\,h_2([z_3,z_4,z_5])=[\bar z_3,z_4,z_5]$.
  \end{itemize}
  We change the generator $h_1$ by $h'_1$:
\[h'_1([z_3,z_4,z_5])=h_1h_2[z_3,z_4,z_5]=[z_3,z_4, \bar z_5].
\]
Define an equivariant diffeomorphism
       $\varphi \colon((\ZZ_2)^2, $ $M(B))\to ((\ZZ_2)^2,$ $ M(B))$ to be
            $\varphi([z_3,z_4,z_5])=([z_3,z_4,z_5]) $
            such that $\varphi g_1=h_2\varphi $ and $\varphi g_2=h'_1\varphi $.
        Hence $M(A_3)$ is diffeomorphic to $M(A_4)$ by Theorem \ref{T2}.
\end{itemize}

\section{Five-Dimensional Real Bott manifolds}\label{S2}

Before giving the classification of 5-dimensional real Bott manifolds, we recall
the classification of 2, 3, 4-dimensional ones
as stated in \cite{N08}, \cite{N08a}.

\begin{theorem}\label{dim-2}
The diffeomorphism classes of $2$-dimensional real Bott manifolds
consist of two. The corresponding Bott matrices  are as follows.
\begin{align*}
A_1=
\left(\begin{array}{cc}
1& 0 \\
0& 1
\end{array}\right),\,\, 
A_2= \left(\begin{array}{cc}
1& 1 \\
0& 1
\end{array}\right).
\end{align*}
\end{theorem}

\begin{theorem}\label{dim-3}
The diffeomorphism classes of $3$-dimensional real Bott manifolds
consist of four. The corresponding Bott matrices are classified 
into four equivalence classes as follows:
\begin{itemize}
\item[a)]
${
\left(\begin{array}{lcr}
1& 1 & 0\\
0& 1 & 1\\
0& 0 & 1
\end{array}\right),  \,
\left(\begin{array}{lcr}
1& 1 & 1\\
0& 1 & 1\\
0& 0 & 1
\end{array}\right)}.$
\item[b)]
${
\left(\begin{array}{lcr}
1& 1 & 1\\
0& 1 & 0\\
0& 0 & 1
\end{array}\right)}.$
\item[c)]
${
\left(\begin{array}{lcr}
1& 0 & 0\\
0& 1 & 1\\
0& 0 & 1
\end{array}\right),\,
\left(\begin{array}{lcr}
1& 0 & 1\\
0& 1 & 1\\
0& 0 & 1
\end{array}\right)},\,
\left(\begin{array}{lcr}
1& 0 & 1\\
0& 1 & 0\\
0& 0 & 1
\end{array}\right),\,
\left(\begin{array}{lcr}
1& 1 & 0\\
0& 1 & 0\\
0& 0 & 1
\end{array}\right).$
\item[d)] $I_3$.
\end{itemize}
\end{theorem}

\begin{theorem}\label{dim-4}
The diffeomorphism classes of $4$-dimensional real Bott manifolds
consist of twelve. The corresponding Bott matrices are classified 
into twelve equivalence classes as follows:
\begin{itemize}
\item[i)]
${
\left(\begin{array}{cccc}
1& 1 & 0 & 0 \\
0& 1 & 1 & 1\\
0& 0 & 1 & 0\\
0& 0 & 0 & 1
\end{array}\right),
\left(\begin{array}{cccc}
1& 1 & 1 & 1 \\
0& 1 & 1 & 1\\
0& 0 & 1 & 0\\
0& 0 & 0 & 1
\end{array}\right)}.$
\item[ii)]
${
\left(\begin{array}{cccc}
1& 1 & 0 & 1 \\
0& 1 & 1 & 1\\
0& 0 & 1 & 0\\
0& 0 & 0 & 1
\end{array}\right),
\left(\begin{array}{cccc}
1& 1 & 1 & 0 \\
0& 1 & 1 & 1\\
0& 0 & 1 & 0\\
0& 0 & 0 & 1
\end{array}\right)
}.$
\item[iii)]
${
\left(\begin{array}{cccc}
1& 1 & 0 & 0 \\
0& 1 & 1 & 0\\
0& 0 & 1 & 1\\
0& 0 & 0 & 1
\end{array}\right),
\left(\begin{array}{cccc}
1& 1 & 1 & 0 \\
0& 1 & 1 & 0\\
0& 0 & 1 & 1\\
0& 0 & 0 & 1
\end{array}\right),
\left(\begin{array}{cccc}
1& 1 & 0 & 0 \\
0& 1 & 1 & 1\\
0& 0 & 1 & 1\\
0& 0 & 0 & 1
\end{array}\right),
\left(\begin{array}{cccc}
1& 1 & 1 & 1 \\
0& 1 & 1 & 1\\
0& 0 & 1 & 1\\
0& 0 & 0 & 1
\end{array}\right)}.$
\item[iv)]
${ 
\left(\begin{array}{cccc}
1& 1 & 0 & 1 \\
0& 1 & 1 & 0\\
0& 0 & 1 & 1\\
0& 0 & 0 & 1
\end{array}\right),
\left(\begin{array}{cccc}
1& 1 & 1 & 1 \\
0& 1 & 1 & 0\\
0& 0 & 1 & 1\\
0& 0 & 0 & 1
\end{array}\right),
\left(\begin{array}{cccc}
1& 1 & 0 & 1 \\
0& 1 & 1 & 1\\
0& 0 & 1 & 1\\
0& 0 & 0 & 1
\end{array}\right),
\left(\begin{array}{cccc}
1& 1 & 1 & 0 \\
0& 1 & 1 & 1\\
0& 0 & 1 & 1\\
0& 0 & 0 & 1
\end{array}\right)}.$
\item[v)]
$
\left(\begin{array}{cccc}
1& 1 & 1 & 0 \\
0& 1 & 0 & 0\\
0& 0 & 1 & 1\\
0& 0 & 0 & 1
\end{array}\right),
\left(\begin{array}{cccc}
1& 1 & 1 & 0 \\
0& 1 & 0 & 1\\
0& 0 & 1 & 1\\
0& 0 & 0 & 1
\end{array}\right),
\left(\begin{array}{cccc}
1& 1 & 1 & 1 \\
0& 1 & 0 & 1\\
0& 0 & 1 & 1\\
0& 0 & 0 & 1
\end{array}\right),
\left(\begin{array}{cccc}
1& 1 & 1 & 0 \\
0& 1 & 0 & 1\\
0& 0 & 1 & 0\\
0& 0 & 0 & 1
\end{array}\right),
\\
\left(\begin{array}{cccc}
1& 1 & 1 & 1 \\
0& 1 & 0 & 1\\
0& 0 & 1 & 0\\
0& 0 & 0 & 1
\end{array}\right),
\left(\begin{array}{cccc}
1& 1 & 1 & 1 \\
0& 1 & 0 & 0\\
0& 0 & 1 & 1\\
0& 0 & 0 & 1
\end{array}\right),
{\left(\begin{array}{cccc}
1& 1 & 0 & 1 \\
0& 1 & 1 & 0\\
0& 0 & 1 & 0\\
0& 0 & 0 & 1
\end{array}\right),
\left(\begin{array}{cccc}
1& 1 & 1 & 1 \\
0& 1 & 1 & 0\\
0& 0 & 1 & 0\\
0& 0 & 0 & 1
\end{array}\right)}.$
\item[vi)]
${
\left(\begin{array}{cccc}
1& 1 & 1 & 1 \\
0& 1 & 0 & 0\\
0& 0 & 1 & 0\\
0& 0 & 0 & 1
\end{array}\right)}.$
\item[vii)]
${
\left(\begin{array}{cccc}
1& 0 & 0 & 1 \\
0& 1 & 1 & 0\\
0& 0 & 1 & 0\\
0& 0 & 0 & 1
\end{array}\right),
\left(\begin{array}{cccc}
1& 0 & 0 & 1 \\
0& 1 & 1 & 1\\
0& 0 & 1 & 0\\
0& 0 & 0 & 1
\end{array}\right),
\left(\begin{array}{cccc}
1& 0 & 1 & 0 \\
0& 1 & 0 & 1\\
0& 0 & 1 & 0\\
0& 0 & 0 & 1
\end{array}\right)},
\left(\begin{array}{cccc}
1& 0 & 1 & 1 \\
0& 1 & 0 & 1\\
0& 0 & 1 & 0\\
0& 0 & 0 & 1
\end{array}\right),
\\
\left(\begin{array}{cccc}
1& 0 & 1 & 0 \\
0& 1 & 1 & 1\\
0& 0 & 1 & 0\\
0& 0 & 0 & 1
\end{array}\right),
\left(\begin{array}{cccc}
1& 0 & 1 & 1 \\
0& 1 & 1 & 0\\
0& 0 & 1 & 0\\
0& 0 & 0 & 1
\end{array}\right),
{\left(\begin{array}{cccc}
1& 1 & 0 & 0 \\
0& 1 & 0 & 0\\
0& 0 & 1 & 1\\
0& 0 & 0 & 1
\end{array}\right),
\left(\begin{array}{cccc}
1& 1 & 0 & 1 \\
0& 1 & 0 & 0\\
0& 0 & 1 & 1\\
0& 0 & 0 & 1
\end{array}\right)}.$
\item[viii)]
${
\left(\begin{array}{cccc}
1& 0 & 0 & 0 \\
0& 1 & 1 & 1\\
0& 0 & 1 & 0\\
0& 0 & 0 & 1
\end{array}\right),
\left(\begin{array}{cccc}
1& 0 & 1 & 1 \\
0& 1 & 0 & 0\\
0& 0 & 1 & 0\\
0& 0 & 0 & 1
\end{array}\right),
\left(\begin{array}{cccc}
1& 0 & 1 & 1 \\
0& 1 & 1 & 1\\
0& 0 & 1 & 0\\
0& 0 & 0 & 1
\end{array}\right)},
{\left(\begin{array}{cccc}
1& 1 & 0 & 1 \\
0& 1 & 0 & 0\\
0& 0 & 1 & 0\\
0& 0 & 0 & 1
\end{array}\right)},
\\{
\left(\begin{array}{cccc}
1& 1 & 1 & 0 \\
0& 1 & 0 & 0\\
0& 0 & 1 & 0\\
0& 0 & 0 & 1
\end{array}\right)}.$
\item[ix)]
${
\left(\begin{array}{cccc}
1& 0 & 0 & 0 \\
0& 1 & 1 & 0\\
0& 0 & 1 & 1\\
0& 0 & 0 & 1
\end{array}\right),
\left(\begin{array}{cccc}
1& 0 & 0 & 0 \\
0& 1 & 1 & 1\\
0& 0 & 1 & 1\\
0& 0 & 0 & 1
\end{array}\right),
\left(\begin{array}{cccc}
1& 0 & 1 & 0 \\
0& 1 & 0 & 0\\
0& 0 & 1 & 1\\
0& 0 & 0 & 1
\end{array}\right)},
\left(\begin{array}{cccc}
1& 0 & 1 & 1 \\
0& 1 & 0 & 0\\
0& 0 & 1 & 1\\
0& 0 & 0 & 1
\end{array}\right),
\\
\left(\begin{array}{cccc}
1& 0 & 1 & 0 \\
0& 1 & 1 & 0\\
0& 0 & 1 & 1\\
0& 0 & 0 & 1
\end{array}\right),
\left(\begin{array}{cccc}
1& 0 & 1 & 1 \\
0& 1 & 1 & 1\\
0& 0 & 1 & 1\\
0& 0 & 0 & 1
\end{array}\right),
{\left(\begin{array}{cccc}
1& 1 & 0 & 0 \\
0& 1 & 0 & 1\\
0& 0 & 1 & 0\\
0& 0 & 0 & 1
\end{array}\right),
\left(\begin{array}{cccc}
1& 1 & 0 & 0 \\
0& 1 & 1 & 0\\
0& 0 & 1 & 0\\
0& 0 & 0 & 1
\end{array}\right)},
\\
{\left(\begin{array}{cccc}
1& 1 & 0 & 1 \\
0& 1 & 0 & 1\\
0& 0 & 1 & 0\\
0& 0 & 0 & 1
\end{array}\right),
\left(\begin{array}{cccc}
1& 1 & 1 & 0 \\
0& 1 & 1 & 0\\
0& 0 & 1 & 0\\
0& 0 & 0 & 1
\end{array}\right)}.$
\item[x)]
${
\left(\begin{array}{cccc}
1& 0 & 0 & 1 \\
0& 1 & 1 & 0\\
0& 0 & 1 & 1\\
0& 0 & 0 & 1
\end{array}\right),
\left(\begin{array}{cccc}
1& 0 & 0 & 1 \\
0& 1 & 1 & 1\\
0& 0 & 1 & 1\\
0& 0 & 0 & 1
\end{array}\right),
\left(\begin{array}{cccc}
1& 0 & 1 & 0 \\
0& 1 & 0 & 1\\
0& 0 & 1 & 1\\
0& 0 & 0 & 1
\end{array}\right),}
\left(\begin{array}{cccc}
1& 0 & 1 & 1 \\
0& 1 & 0 & 1\\
0& 0 & 1 & 1\\
0& 0 & 0 & 1
\end{array}\right),
\\
\left(\begin{array}{cccc}
1& 0 & 1 & 0 \\
0& 1 & 1 & 1\\
0& 0 & 1 & 1\\
0& 0 & 0 & 1
\end{array}\right),
\left(\begin{array}{cccc}
1& 0 & 1 & 1 \\
0& 1 & 1 & 0\\
0& 0 & 1 & 1\\
0& 0 & 0 & 1
\end{array}\right),
{\left(\begin{array}{cccc}
0& 1 & 0 & 0 \\
0& 1 & 0 & 1\\
0& 0 & 1 & 1\\
0& 0 & 0 & 1
\end{array}\right),
\left(\begin{array}{cccc}
1& 1 & 0 & 1 \\
0& 1 & 0 & 1\\
0& 0 & 1 & 1\\
0& 0 & 0 & 1
\end{array}\right)}.$
\item[xi)]
${
\left(\begin{array}{cccc}
1& 0 & 0 & 0 \\
0& 1 & 0 & 0\\
0& 0 & 1 & 1\\
0& 0 & 0 & 1
\end{array}\right),
\left(\begin{array}{cccc}
1& 0 & 0 & 0 \\
0& 1 & 0 & 1\\
0& 0 & 1 & 0\\
0& 0 & 0 & 1
\end{array}\right),
\left(\begin{array}{cccc}
1& 0 & 0 & 0 \\
0& 1 & 0 & 1\\
0& 0 & 1 & 1\\
0& 0 & 0 & 1
\end{array}\right)},
\left(\begin{array}{cccc}
1& 0 & 0 & 1 \\
0& 1 & 0 & 0\\
0& 0 & 1 & 0\\
0& 0 & 0 & 1
\end{array}\right),
\\
{\left(\begin{array}{cccc}
1& 0 & 0 & 1 \\
0& 1 & 0 & 0\\
0& 0 & 1 & 1\\
0& 0 & 0 & 1
\end{array}\right),
\left(\begin{array}{cccc}
1& 0 & 0 & 1 \\
0& 1 & 0 & 1\\
0& 0 & 1 & 0\\
0& 0 & 0 & 1
\end{array}\right),
\left(\begin{array}{cccc}
1& 0 & 0 & 1 \\
0& 1 & 0 & 1\\
0& 0 & 1 & 1\\
0& 0 & 0 & 1
\end{array}\right),
\left(\begin{array}{cccc}
1& 0 & 0 & 0 \\
0& 1 & 1 & 0\\
0& 0 & 1 & 0\\
0& 0 & 0 & 1
\end{array}\right)},
\\
\left(\begin{array}{cccc}
1& 1 & 0 & 0 \\
0& 1 & 0 & 0\\
0& 0 & 1 & 0\\
0& 0 & 0 & 1
\end{array}\right),
\left(\begin{array}{cccc}
1& 0 & 1 & 0 \\
0& 1 & 0 & 0\\
0& 0 & 1 & 0\\
0& 0 & 0 & 1
\end{array}\right),
\left(\begin{array}{cccc}
1& 0 & 1 & 0 \\
0& 1 & 1 & 0\\
0& 0 & 1 & 0\\
0& 0 & 0 & 1
\end{array}\right).$
\item[xii)] $I_4$.
\end{itemize}
\end{theorem}


Using the classification results of Theorem \ref{dim-2}, \ref{dim-3},
\ref{dim-4}, we shall classify $5$-dimensional real Bott manifolds. 
\subsection{$S^1$-actions with $4$-dimensional quotients}\label{subs-S1} \par\

The Bott matrices of $M(A)$ admitting $S^1$-actions have the following form
\begin{align*}
{
\left(\begin{array}{c|cccc}
 1& 1    & a_{13}  & a_{14} & a_{15} \\
\hline
 \mbox{\Large{$0$}}&     & \mbox{\large{$B$}} &  &  
\end{array}\right)}
\end{align*}
where $a_{13}, a_{14},a_{15}\in \{0,1\}$. In this case $M(B)$ corresponds to the
Bott matrices $B$ in Theorem \ref{dim-4}. 
Taking the first Bott matrix from i) as $B$,
we consider the following Bott matrices.

\begin{align*}
{A_1=\left(\begin{array}{c|cccc}
 1& 1    &0  & 0 & 0 \\
\hline
 0& 1    &1  & 0 & 0 \\
 0& 0    &1  & 1 & 1 \\
 0& 0    &0  & 1 & 0 \\
 0& 0    &0  & 0 & 1
\end{array}\right)}\,
{
A_2=\left(\begin{array}{c|cccc}
 1& 1    &0  & 1 & 0 \\
\hline
0& 1    &1  & 0 & 0 \\
 0& 0    &1  & 1 & 1 \\
 0& 0    &0  & 1 & 0 \\
 0& 0    &0  & 0 & 1
\end{array}\right)}\,
{
A_3=\left(\begin{array}{c|cccc}
 1& 1    &0  & 1 & 1 \\
\hline
0& 1    &1  & 0 & 0 \\
 0& 0    &1  & 1 & 1 \\
 0& 0    &0  & 1 & 0 \\
 0& 0    &0  & 0 & 1
\end{array}\right)}.
\end{align*}
Then the fixed point sets of the $\ZZ_2$-actions on $M(B)$ corresponding to
$A_1$, $A_2$ and $A_3$ respectively are as follows:
(1) $T^3$, $T^2$, $4$ points,
(2) $2$ components $T^2$, $4$ components $S^1$,
(3) $T^2$, $4$ components $S^1$, $4$ points. 
(1), (2) and (3) have the different fixed point sets each other
so each $A_i$ ($i=1,2,3$) is not equivalent by Proposition \ref{T5}.


The following Bott matrices are created from the first Bott matrix in ii).
\begin{align*}
{
A_4=\left(\begin{array}{c|cccc}
 1& 1    &1  & 0 & 0 \\
\hline
 0& 1    &1  & 0 & 1 \\
 0& 0    &1  & 1 & 1 \\
 0& 0    &0  & 1 & 0 \\
 0& 0    &0  & 0 & 1
\end{array}\right)}\,
{
A_5=\left(\begin{array}{c|cccc}
 1& 1    &0  & 0 & 0 \\
\hline
 0& 1    &1  & 0 & 1 \\
 0& 0    &1  & 1 & 1 \\
 0& 0    &0  & 1 & 0 \\
 0& 0    &0  & 0 & 1
\end{array}\right)}\,
{
A_6=\left(\begin{array}{c|cccc}
 1& 1    &1  & 1 & 0 \\
\hline
 0& 1    &1  & 0 & 1 \\
 0& 0    &1  & 1 & 1 \\
 0& 0    &0  & 1 & 0 \\
 0& 0    &0  & 0 & 1
\end{array}\right)}.
\end{align*}
The fixed point sets of the $\ZZ_2$-actions on $M(B)$ corresponding to
 $A_4$, $A_5$ and $A_6$ are obtained as:
(1) $3$ components $T^2$, $4$ points,
(2) $T^3$, $4$ components $S^1$,
(3) $8$ components $S^1$. 
In view of the fixed points, similarly $A_i$ ($i=4,5,6$) 
are not equivalent to each other.

The following Bott matrices are created from the first Bott matrix in iii). 
\begin{align*}
\begin{split}
&{
A_7=\left(\begin{array}{c|cccc}
 1& 1    &0  & 0 & 0 \\
\hline
 0& 1    &1  & 0 & 0 \\
 0& 0    &1  & 1 & 0 \\
 0& 0    &0  & 1 & 1 \\
 0& 0    &0  & 0 & 1
\end{array}\right)} \,
A_8=\left(\begin{array}{c|cccc}
 1& 1    &0  & 1 & 0 \\
\hline
 0& 1    &1  & 0 & 0 \\
 0& 0    &1  & 1 & 0 \\
 0& 0    &0  & 1 & 1 \\
 0& 0    &0  & 0 & 1
\end{array}\right) \\
&{
A_9=\left(\begin{array}{c|cccc}
 1& 1    &0  & 0 & 1 \\
\hline
 0& 1    &1  & 0 & 0 \\
 0& 0    &1  & 1 & 0 \\
 0& 0    &0  & 1 & 1 \\
 0& 0    &0  & 0 & 1
\end{array}\right)} \, 
A_{10}=\left(\begin{array}{c|cccc}
 1& 1    &0  & 1 & 1 \\
\hline
 0& 1    &1  & 0 & 0 \\
 0& 0    &1  & 1 & 0 \\
 0& 0    &0  & 1 & 1 \\
 0& 0    &0  & 0 & 1
\end{array}\right).
\end{split}
\end{align*}
The fixed point sets of the $\ZZ_2$-actions on $M(B)$ corresponding 
to  $A_7$, $A_8$, $A_9$ and $A_{10}$ are as follows:
\begin{itemize}
\item[$(1)$] $T^3$, $T^2$, $S^1$, $2$ points,
\item[$(2)$] $2$ components $T^2=\{[1,z_3, 1,z_5],[{\bf i}, {\bf i},z_4,z_5]\}$, $3$ components\\
              $S^1= \{[ {\bf i}, 1, 1,$ $z_5],[ 1,z_3, {\bf i}, 1],$ $[ 1,z_3, {\bf i}, {\bf i}]\}$, 
            $2$ points=$\{[{\bf i}, 1, {\bf i}, 1],[ {\bf i}, 1, {\bf i},{\bf i}]\}$,
\item[$(3)$] $2$ components $T^2=\{[ 1,z_3,z_4, 1],[ 1,z_3,z_4, {\bf i}]\}$, $3$ components\\
              $S^1=\{[ {\bf i}, {\bf i},{\bf i},$ $z_5], [ {\bf i}, 1,z_4, 1],$ $[{\bf i}, 1,z_4, {\bf i}]\}$, 
            $2$ points=$\{[{\bf i}, {\bf i}, 1, 1],[{\bf i}, {\bf i}, 1,{\bf i}]\}$,
\item[$(4)$] $T^2$, $5$ components $S^1$, $2$ points.
\end{itemize}  
Note that the fixed point sets of  $(2)$ and $(3)$ coincide, but
the type of them are different. (Compare {\bf b)} in Section \ref{EXAMPEL} for the type
(1,1,1,1).) Hence $A_8$ and $A_9$ are not equivalent. 
As the fixed point sets (1), (4) and (2) (or (3)) are all different,
each $A_i$ $(i=7,8,9,10)$ is not equivalent.  


The following Bott matrices are created from the first Bott matrix in iv).  

\begin{align*}
\begin{split}
&{
A_{11}=\left(\begin{array}{c|cccc}
 1& 1    &0  & 0 & 0 \\
\hline
 0& 1    &1  & 0 & 1 \\
 0& 0    &1  & 1 & 0 \\
 0& 0    &0  & 1 & 1 \\
 0& 0    &0  & 0 & 1
\end{array}\right)} \, 
{
A_{12}=\left(\begin{array}{c|cccc}
 1& 1    &0  & 1 & 0 \\
\hline
 0& 1    &1  & 0 & 1 \\
 0& 0    &1  & 1 & 0 \\
 0& 0    &0  & 1 & 1 \\
 0& 0    &0  & 0 & 1
\end{array}\right)} \\
&{
A_{13}=\left(\begin{array}{c|cccc}
 1& 1    &0  & 0 & 1 \\
\hline
 0& 1    &1  & 0 & 1 \\
 0& 0    &1  & 1 & 0 \\
 0& 0    &0  & 1 & 1 \\
 0& 0    &0  & 0 & 1
\end{array}\right)} \, 
{
A_{14}=\left(\begin{array}{c|cccc}
 1& 1    &0  & 1 & 1 \\
\hline
 0& 1    &1  & 0 & 1 \\
 0& 0    &1  & 1 & 0 \\
 0& 0    &0  & 1 & 1 \\
 0& 0    &0  & 0 & 1
\end{array}\right)}.
\end{split}
\end{align*}
The fixed point sets of the $\ZZ_2$-actions on $M(B)$ corresponding
to 
$A_{11}$, $A_{12}$, $A_{13}$ and $A_{14}$ are as follows:
(1) $T^3$, $3$ components $S^1$, 2 points, 
(2) $T^2$, $5$ components $S^1$, 2 points, 
(3) $3$ components $T^2$, $S^1$, 2 points, 
(4) $2$ components $T^2$, $3$ components $S^1$, $2$ points.  
By Proposition \ref{T5}, $A_{i}$ ($i=11,12,13,14$) are not equivalent to each other.


The Bott matrices $A_{i}$ ($i=15,16,17,18$)  below are created 
from the first  Bott matrix in v) while  $A_{19}$ is created from the second
Bott matrix in v).  
\begin{align*}
\begin{split}
&{
A_{15}=\left(\begin{array}{c|cccc}
 1& 1    &0  & 0 & 0 \\
\hline
 0& 1    &1  & 1 & 0 \\
 0& 0    &1  & 0 & 0 \\
 0& 0    &0  & 1 & 1 \\
 0& 0    &0  & 0 & 1
\end{array}\right)} \,
{
A_{16}=\left(\begin{array}{c|cccc}
 1& 1    &0  & 1 & 0 \\
\hline
 0& 1    &1  & 1 & 0 \\
 0& 0    &1  & 0 & 0 \\
 0& 0    &0  & 1 & 1 \\
 0& 0    &0  & 0 & 1
\end{array}\right)} \\
&{
A_{17}=\left(\begin{array}{c|cccc}
 1& 1    &0  & 0 & 1 \\
\hline
 0& 1    &1  & 1 & 0 \\
 0& 0    &1  & 0 & 0 \\
 0& 0    &0  & 1 & 1 \\
 0& 0    &0  & 0 & 1
\end{array}\right)} \,
{
A_{18}=\left(\begin{array}{c|cccc}
 1& 1    &0  & 1 & 1 \\
\hline
 0& 1    &1  & 1 & 0 \\
 0& 0    &1  & 0 & 0 \\
 0& 0    &0  & 1 & 1 \\
 0& 0    &0  & 0 & 1
\end{array}\right)} \\
&{
A_{19}=\left(\begin{array}{c|cccc}
 1& 1    &1  & 0 & 0 \\
\hline
 0& 1    &1  & 1 & 0 \\
 0& 0    &1  & 0 & 1 \\
 0& 0    &0  & 1 & 1 \\
 0& 0    &0  & 0 & 1
\end{array}\right)}.
\end{split}
\end{align*}
The fixed point sets of the $\ZZ_2$-actions on $M(B)$  corresponding  to 
$A_{15}$, $A_{16}$, $A_{17}$, $A_{18}$ and $A_{19}$ are as follows:
(1) $T^3$, $2$ components $S^1$, $4$ points,
(2) $3$ components $T^2$, $2$ components $S^1$,
(3) $2$ components $T^2$, $2$ components $S^1$, $4$ point,
(4) $T^2$, $6$ components $S^1$,
(5) $2$ components $T^2$,
    $4$ components $S^1$.  
By Proposition \ref{T5}, 
$A_{i}$ ($i=15,16,17,18,19$) are not equivalent to each other.


The following Bott matrices are created from the Bott matrix  vi). 
\begin{align*}
{
A_{20}=\left(\begin{array}{c|cccc}
 1& 1    &0  & 0 & 0 \\
\hline
 0& 1    &1  & 1 & 1 \\
 0& 0    &1  & 0 & 0 \\
 0& 0    &0  & 1 & 0 \\
 0& 0    &0  & 0 & 1
\end{array}\right)} \, 
{
A_{21}=\left(\begin{array}{c|cccc}
 1& 1    &1  & 0 & 0 \\
\hline
 0& 1    &1  & 1 & 1 \\
 0& 0    &1  & 0 & 0 \\
 0& 0    &0  & 1 & 0 \\
 0& 0    &0  & 0 & 1
\end{array}\right)}.
\end{align*}
The fixed point sets of the $\ZZ_2$-actions on $M(B)$  corresponding to
$A_{20}$ and $A_{21}$ are as follows:
(1) $T^3$, $8$ points,
(2) $2$ components $T^2$, $4$ components $S^1$. 
By Proposition \ref{T5}, $A_{20}$ and  $A_{21}$ are not equivalent.


The Bott matrix 
\begin{align*}
{
A_{22}=\left(\begin{array}{c|cccc}
 1& 1    &1  & 0 & 0 \\
\hline
 0& 1    &0  & 0 & 1 \\
 0& 0    &1  & 1 & 0 \\
 0& 0    &0  & 1 & 0 \\
 0& 0    &0  & 0 & 1
\end{array}\right)}
\end{align*}
is created from the first  Bott matrix  in vii). 


The following Bott matrices are created from the first Bott matrix  in viii). 
\begin{align*}
{
A_{23}=\left(\begin{array}{c|cccc}
 1& 1    &1  & 0 & 0 \\
\hline
 0& 1    &0  & 0 & 0 \\
 0& 0    &1  & 1 & 1 \\
 0& 0    &0  & 1 & 0 \\
 0& 0    &0  & 0 & 1
\end{array}\right)} \, 
{
A_{24}=\left(\begin{array}{c|cccc}
 1& 1    &1  & 1 & 0 \\
\hline
 0& 1    &0  & 0 & 0 \\
 0& 0    &1  & 1 & 1 \\
 0& 0    &0  & 1 & 0 \\
 0& 0    &0  & 0 & 1
\end{array}\right)}.
\end{align*}
The fixed point sets of the $\ZZ_2$-actions on $M(B)$  corresponding to 
 $A_{23}$ and $A_{24}$ are as follows:
(1) $2$ components $T^2$, $8$ points,
(2) $8$ components $S^1$. 
By Proposition \ref{T5}, $A_{23}$ is not equivalent to  $A_{24}$.


The following Bott matrices are created from the first Bott matrix  in ix).  
\begin{align*}
{
A_{25}=\left(\begin{array}{c|cccc}
 1& 1    &1  & 0 & 0 \\
\hline
 0& 1    &0  & 0 & 0 \\
 0& 0    &1  & 1 & 0 \\
 0& 0    &0  & 1 & 1 \\
 0& 0    &0  & 0 & 1
\end{array}\right)} \, 
{
A_{26}=\left(\begin{array}{c|cccc}
 1& 1    &1  & 0 & 1 \\
\hline
 0& 1    &0  & 0 & 0 \\
 0& 0    &1  & 1 & 0 \\
 0& 0    &0  & 1 & 1 \\
 0& 0    &0  & 0 & 1
\end{array}\right)}.
\end{align*}
The fixed point sets of the $\ZZ_2$-actions on $M(B)$  corresponding to
 $A_{25}$ and  $A_{26}$ are as follows:
(1) $2$ components $T^2$, $2$ components $S^1$, $4$ points,
(2) $6$ components $S^1$, $4$ points. 
By Proposition \ref{T5}, $A_{25}$ and $A_{26}$ are not equivalent.


The Bott matrix $A_{27}$ (resp. $A_{28}$) below is created from the first  Bott matrix  in 
x) (resp.  xi)).  
\begin{align*}
{
A_{27}=\left(\begin{array}{c|cccc}
 1& 1    &1  & 0 & 0 \\
\hline
 0& 1    &0  & 0 & 1 \\
 0& 0    &1  & 1 & 0 \\
 0& 0    &0  & 1 & 1 \\
 0& 0    &0  & 0 & 1
\end{array}\right)} \,
{
A_{28}=\left(\begin{array}{c|cccc}
 1& 1    &1  & 1 & 0 \\
\hline
 0& 1    &0  & 0 & 0 \\
 0& 0    &1  & 0 & 0 \\
 0& 0    &0  & 1 & 1 \\
 0& 0    &0  & 0 & 1
\end{array}\right)}.
\end{align*}

Finally from $I_4$, we get
$A_{29}=\left(\begin{array}{c|cccc}
 1& 1    &1  & 1 & 1 \\
\hline
 $\mbox{\Large$0 $}$&     &$\mbox{\large$I_4 $}$  &  &  
\end{array}\right)$.\\

Since each Bott matrix $B$ of i) to xii) is not equivalent to
each other, the resulting Bott matrix $A$ is not equivalent. 
Totally, 29 Bott matrices  $A_{i}$ ($i=1,\dots,29$)
 are not equivalent to each other. 
 When we take the second Bott matrix $B'$ from i),
the resulting Bott matrix $A'$  gives an action $(\ZZ_2,M(B'))$.
We can check that $(\ZZ_2,M(B'))$ is equivariantly diffeomorphic to 
one of the actions $(\ZZ_2,M(B))$ corresponding to $A_1$, $A_2$, $A_3$ by the ad hoc argument. (Compare Section \ref{EXAMPEL}
for the argument to find an equivariant diffeomorphism.)
Once there exists such an equivariant diffeomorphism, $A'$ is equivalent to one of 
$A_1$, $A_2$, $A_3$ by Theorem \ref{T2}. 
Similarly, if $A'$ is another Bott matrix created from the first Bott matrix in i), 
we can check that the corresponding 
$(\ZZ_2,M(B))$ is equivariantly diffeomorphic to one of the actions $(\ZZ_2,M(B))$ 
corresponding to $A_1,A_2,A_3$.  
(Note that the total number of Bott matrices created from the first Bott matrix in i) is 8.)
This argument works not only the case i) but also the 
cases from ii) to xii). As a consequence the Bott matrix $A'$ created from
Bott matrices from ii) to xii) is equivalent to one of $A_i$'s $(i=4,\dots,29)$. 
In summary, we obtain the following but the proof is omitted because of a
 tedious argument. 
\begin{lemma}\label{lemS1}
A Bott matrix created from any one of Bott matrices of Theorem \ref{dim-4} 
is equivalent to 
one of the Bott matrices $A_{i}$ $(i=1,\dots,29)$ above.
\end{lemma}
\begin{pro}\label{proS1}
There are $29$ diffeomorphism classes of the case $S^1$-actions with $4$-dimensional
quotients.
\end{pro}


\subsection{$T^2$-actions with $3$-dimensional quotients}\par\

The Bott matrices of $M(A)$ admitting $T^2$-actions have the following
form
\begin{align*}
{
\left(\begin{array}{c|c}
 \mbox {\large{$I_2$}} & \mbox {\large{$\ast $}} \\
\hline
\mbox {\Large{0}} & \mbox {\large{B}}\\
\end{array}\right)}.
\end{align*}


The following Bott matrices are created from the first Bott matrix
$B$ of a) in Theorem \ref{dim-3}. 
\begin{align*}
&{
A_{30}=\left(\begin{array}{cc|ccc}
 1& 0    &0  & 0 & 0 \\
 0& 1    &1  & 0 & 0 \\
\hline
 0& 0    &1  & 1 & 0 \\
 0& 0    &0  & 1 & 1 \\
 0& 0    &0  & 0 & 1
\end{array}\right)}\,
{
A_{31}=\left(\begin{array}{cc|ccc}
 1& 0    &0  & 0 & 0 \\
 0& 1    &1  & 0 & 1 \\
\hline
 0& 0    &1  & 1 & 0 \\
 0& 0    &0  & 1 & 1 \\
 0& 0    &0  & 0 & 1
\end{array}\right)} \\
&{
A_{32}=\left(\begin{array}{cc|ccc}
 1& 0    &0  & 1 & 0 \\
 0& 1    &1  & 0 & 0 \\
\hline
 0& 0    &1  & 1 & 0 \\
 0& 0    &0  & 1 & 1 \\
 0& 0    &0  & 0 & 1
\end{array}\right)}\,
{
A_{33}=\left(\begin{array}{cc|ccc}
 1& 0    &1  & 0 & 1 \\
 0& 1    &0  & 1 & 0 \\
\hline
 0& 0    &1  & 1 & 0 \\
 0& 0    &0  & 1 & 1 \\
 0& 0    &0  & 0 & 1
\end{array}\right)} \\
&{
A_{34}=\left(\begin{array}{cc|ccc}
 1& 0    &1  & 0 & 0 \\
 0& 1    &0  & 0 & 1 \\
\hline
 0& 0    &1  & 1 & 0 \\
 0& 0    &0  & 1 & 1 \\
 0& 0    &0  & 0 & 1
\end{array}\right)}\,
{
A_{35}=\left(\begin{array}{cc|ccc}
 1& 0    &0  & 1 & 1 \\
 0& 1    &1  & 0 & 0 \\
\hline
 0& 0    &1  & 1 & 0 \\
 0& 0    &0  & 1 & 1 \\
 0& 0    &0  & 0 & 1
\end{array}\right)}.
\end{align*}
The fixed point sets of the $(\ZZ_2)^s$-actions ($s=1,2$) on $M(B)$  corresponding to
$A_{30}$, $A_{31}$, $A_{32}$, $A_{33}$, $A_{34}$ and $A_{35}$ are as follows:
(1) $T^2$, $S^1$, $2$ points, 
(2) $3$ components $S^1$, $2$ points, 
(3) $3$ components $T^2$,  
        $4$ components   $S^1$,  
                                 $4$ points,    
(4) $T^2$, $8$ components $S^1$, $4$ points, 
(5) $3$ components $T^2$, 
        $4$ components $S^1$, 
        $4$ points,   
(6) $2$ components $T^2$, $6$ components $S^1$, $4$ points. 
Compared (3) with (5), 
we see from {\bf b)} in Section  \ref{EXAMPEL}
that $A_{32}$ is not equivalent to $A_{34}$ .  
By Proposition \ref{T5}, Bott matrices 
$A_{i}$ ($i=30,31$) (resp. $A_{j}$ ($j=32,33,34,35$)) are not equivalent to each other. 
Moreover, by Remark \ref{rem15}, Bott matrices $A_{i}$ ($i=30,31$)  are not equivalent to $A_{j}$ ($j=32,33,34,35$) 
because the $(\ZZ_2)^2$-action corresponding to  $A_{j}$ ($j=32,33,34,35$) 
cannot be reduced to a $\ZZ_2$-action.

The following Bott matrices are created from the Bott matrix  b) in
Theorem \ref{dim-3}.
\begin{align*}
\begin{split}
&{
A_{36}=\left(\begin{array}{cc|ccc}
 1& 0    &0  & 0 & 0 \\
 0& 1    &1  & 0 & 0 \\
\hline
 0& 0    &1  & 1 & 1 \\
 0& 0    &0  & 1 & 0 \\
 0& 0    &0  & 0 & 1
\end{array}\right)} \,  
{
A_{37}=\left(\begin{array}{cc|ccc}
 1& 0    &0  & 0 & 0 \\
 0& 1    &1  & 1 & 0 \\
\hline
 0& 0    &1  & 1 & 1 \\
 0& 0    &0  & 1 & 0 \\
 0& 0    &0  & 0 & 1
\end{array}\right)} \\
&{
A_{38}=\left(\begin{array}{cc|ccc}
 1& 0    &0  & 1 & 1 \\
 0& 1    &1  & 0 & 0 \\
\hline
 0& 0    &1  & 1 & 1 \\
 0& 0    &0  & 1 & 0 \\
 0& 0    &0  & 0 & 1
\end{array}\right)} \,
{
A_{39}=\left(\begin{array}{cc|ccc}
 1& 0    &0  & 1 & 0 \\
 0& 1    &1  & 0 & 0 \\
\hline
 0& 0    &1  & 1 & 1 \\
 0& 0    &0  & 1 & 0 \\
 0& 0    &0  & 0 & 1
\end{array}\right)} \\
&{
A_{40}=\left(\begin{array}{cc|ccc}
 1& 0    &1  & 0 & 1 \\
 0& 1    &1  & 1 & 0 \\
\hline
 0& 0    &1  & 1 & 1 \\
 0& 0    &0  & 1 & 0 \\
 0& 0    &0  & 0 & 1
\end{array}\right)}.
\end{split}
\end{align*}
The fixed point sets of the $(\ZZ_2)^s$-actions ($s=1,2$) on $M(B)$  corresponding to
$A_{36}$, $A_{37}$, $A_{38}$, $A_{39}$ and $A_{40}$ are as follows:
(1) $T^2$, $4$ points,
(2) $4$ components $S^1$,
(3) $2$ components $T^2$, $4$ components $S^1$, $8$ points,
(4) $3$ components $T^2$, $4$ components $S^1$, $4$ points,
(5) $12$ components $S^1$.  
By Remark \ref{rem15}, Bott matrices $A_{i}$ ($i=36,37$)
 are not equivalent to  $A_{j}$ ($j=38,39,40$) 
because the $(\ZZ_2)^2$-action corresponding to  $A_{j}$ ($j=38,39,40$) 
cannot be reduced to a $\ZZ_2$-action. 
On the other hand, by Proposition \ref{T5}, Bott matrices $A_{i}$ ($i=36,37$) (resp. $A_{j}$ ($j=38,39,40$))
 are not equivalent to each other.


The Bott matrices $A_{i}$ ($i=41,42,43,44$)  below are created 
from the first  Bott matrix in c) of Theorem \ref{dim-3} while 
$A_{45}$ is created from the second Bott matrix in c). 
\begin{align*}
&{
A_{41}=\left(\begin{array}{cc|ccc}
 1& 0    &0  & 0 & 0 \\
 0& 1    &1  & 1 & 0 \\
\hline
 0& 0    &1  & 0 & 0 \\
 0& 0    &0  & 1 & 1 \\
 0& 0    &0  & 0 & 1
\end{array}\right)} \, 
{
A_{42}=\left(\begin{array}{cc|ccc}
 1& 0    &0  & 1 & 0 \\
 0& 1    &1  & 0 & 0 \\
\hline
 0& 0    &1  & 0 & 0 \\
 0& 0    &0  & 1 & 1 \\
 0& 0    &0  & 0 & 1
\end{array}\right)}  \\
&{
A_{43}=\left(\begin{array}{cc|ccc}
 1& 0    &0  & 1 & 0 \\
 0& 1    &1  & 0 & 1 \\
\hline
 0& 0    &1  & 0 & 0 \\
 0& 0    &0  & 1 & 1 \\
 0& 0    &0  & 0 & 1
\end{array}\right)} \, 
{
A_{44}=\left(\begin{array}{cc|ccc}
 1& 0    &0  & 0 & 1 \\
 0& 1    &1  & 1 & 0 \\
\hline
 0& 0    &1  & 0 & 0 \\
 0& 0    &0  & 1 & 1 \\
 0& 0    &0  & 0 & 1
\end{array}\right)}\\
&{
A_{45}=\left(\begin{array}{cc|ccc}
 1& 0    &0  & 1 & 0 \\
 0& 1    &1  & 0 & 0 \\
\hline
 0& 0    &1  & 0 & 1 \\
 0& 0    &0  & 1 & 1 \\
 0& 0    &0  & 0 & 1
\end{array}\right)}.
\end{align*}
The fixed point sets of the $(\ZZ_2)^2$-actions 
on $M(B)$  corresponding to
 $A_{42}$, $A_{43}$, $A_{44}$ and $A_{45}$ are as follows:
(1) $3$ components $T^2$, $4$ components $S^1$, $4$ points,
(2) $T^2$, $8$ components $S^1$, $4$ points,
(3) $2$ components $T^2$, $4$ components $S^1$, $8$ points,
(4) $2$ components $T^2$, $6$ components $S^1$, $4$ points. 
By Remark \ref{rem15}, $A_{41}$ is not equivalent to $A_{i}$  ($i=42,43,44,45$) 
because the $(\ZZ_2)^2$-action corresponding to  $A_{i}$ ($i=42,43,44,45$) 
cannot be reduced to a $\ZZ_2$-action. Then
by Proposition \ref{T5}, Bott matrices $A_{i}$  ($i=42,43,44,45$)  are not equivalent to each other.


The following Bott matrices are created from $I_3$.
\begin{align*}
&{
A_{46}=\left(\begin{array}{cc|ccc}
 1& 0    &0  & 0 & 0 \\
 0& 1    &1  & 1 & 1 \\
\hline
 \mbox {\Large{$0$}}& \mbox {\Large{$0$}}   &  & \mbox {\large{$I_3$}} &  
\end{array}\right)}\,
{
A_{47}=\left(\begin{array}{cc|ccc}
 1& 0    &0  & 1 & 1 \\
 0& 1    &1  & 0 & 0 \\
\hline
 \mbox {\Large{$0$}}& \mbox {\Large{$0$}}   &  & \mbox {\large{$I_3$}} &  
\end{array}\right)} \\
&{
A_{48}=\left(\begin{array}{cc|ccc}
 1& 0    &1  & 0 & 1 \\
 0& 1    &1  & 1 & 0 \\
\hline
 \mbox {\Large{$0$}}& \mbox {\Large{$0$}}   &  & \mbox {\large{$I_3$}} &  
\end{array}\right)}.
\end{align*}
The fixed point sets of the $(\ZZ_2)^2$-actions  on $M(B)$  corresponding to
$A_{47}$ and $A_{48}$ are as follows:
(1) $2$ components $T^2$, $4$ components $S^1$, $8$ points,
(2) $12$ components $S^1$. 
By Remark \ref{rem15},  $A_{46}$  is not equivalent to $A_{i}$ ($i=47,48$), 
and by Proposition \ref{T5}, $A_{47}$ is not equivalent to $A_{48}$.

Since each Bott matrix B of a) to d) is not equivalent to each other, the resulting Bott matrix 
$A$ is not equivalent. Totally, 19 Bott matrices $A_{i}$ ($i=30,\dots,48$)
  are not equivalent to each other. 

When we take the second Bott matrix $B'$ from a) of Theorem \ref{dim-3},
the resulting Bott matrix $A'$  gives an action $((\ZZ_2)^s,M(B'))$ $(s=1,2)$.
We can check that $((\ZZ_2)^s,M(B'))$ is equivariantly diffeomorphic to 
one of the actions $((\ZZ_2)^s,M(B))$ corresponding to $A_i$ ($i=30,\dots,35$) by the ad hoc argument. 
(Compare Section \ref{EXAMPEL}
for the argument to find an equivariant diffeomorphism.)
Once there exists such an equivariant diffeomorphism, $A'$ is equivalent to one of 
$A_i$'s ($i=30,\dots,35$) by Theorem \ref{T2}.
Similarly, if $A'$ is another Bott matrix created from the first Bott matrix in a)
of Theorem \ref{dim-3}, we can check that the corresponding 
$((\ZZ_2)^s,M(B))$ ($s=1,2$) is equivariantly diffeomorphic to one of the actions $((\ZZ_2)^s,M(B))$ 
corresponding to $A_i$ ($i=30,\dots,35$). 
(Note that the total number of Bott matrices created from the first Bott matrix in a) is 76.)   
This argument also works for the case b), c) and d). As a consequence the Bott matrix $A'$ created  from
Bott matrices in b), c) and d) is equivalent to one of $A_i$'s ($i=36,\dots,48$). 
In summary, we obtain the following.
\begin{lemma}\label{lemS1}
A Bott matrix created from any one of Bott matrices of Theorem \ref{dim-3} 
is equivalent to 
one of the Bott matrices $A_i$ $(i=30,\dots,48)$ above.
\end{lemma}
\begin{pro}\label{proT2}
There are $19$ diffeomorphism classes of the case $T^2$-actions with $3$-dimensional
quotients.
\end{pro}

\subsection{$T^3$-actions with $2$-dimensional quotients}
\par\

The Bott matrices of $M(A)$ admitting $T^3$-actions have the following
form
\begin{align*}
{
\left(\begin{array}{c|c}
 \mbox {\large{$I_3$}} & \mbox {\large{$\ast $}} \\
\hline
\mbox {\large{0}} & \mbox {\large{B}}\\
\end{array}\right)}.
\end{align*}
In this case a Bott matrix $B$ is either  $A_1$ or $A_2$ in 
Theorem \ref{dim-2}.  
The Bott matrices $A_{49}$ and $A_{50}$ (resp. $A_{51}$ and $A_{52}$) below are created from 
$A_{1}$ (resp. $A_{2}$).
\begin{align*}
{
A_{49}=\left(\begin{array}{ccc|cc}
 1& 0    &0  & 0 & 0 \\
 0& 1    &0  & 0 & 0 \\
 0& 0    &1  & 1 & 1 \\
\hline
 0& 0    &0  & 1 & 0 \\
 0& 0    &0  & 0 & 1
\end{array}\right)} \, 
{
A_{50}=\left(\begin{array}{ccc|cc}
 1& 0    &0  & 0 & 0 \\
 0& 1    &0  & 1 & 0 \\
 0& 0    &1  & 0 & 1 \\
\hline
 0& 0    &0  & 1 & 0 \\
 0& 0    &0  & 0 & 1
\end{array}\right)}\\
{
A_{51}=\left(\begin{array}{ccc|cc}
 1& 0    &0  & 0 & 0 \\
 0& 1    &0  & 0 & 0 \\
 0& 0    &1  & 1 & 0 \\
\hline
 0& 0    &0  & 1 & 1 \\
 0& 0    &0  & 0 & 1
\end{array}\right)} \, 
{
A_{52}=\left(\begin{array}{ccc|cc}
 1& 0    &0  & 0 & 0 \\
 0& 1    &0  & 1 & 0 \\
 0& 0    &1  & 0 & 1 \\
\hline
 0& 0    &0  & 1 & 1 \\
 0& 0    &0  & 0 & 1
\end{array}\right)}.
\end{align*}
Since $A_1$ and $A_2$ in Theorem \ref{dim-2} are not equivalent, none of
$A_{49}$ and $A_{50}$ is equivalent to $A_{51}$ or $A_{52}$. Then 
$A_{49}$ (resp. $A_{51}$) is not equivalent to $A_{50}$ (resp. $A_{52}$), because 
$(\ZZ_2)^2$-action on $M(B)$ corresponding to  $A_{50}$ (or $A_{52}$) cannot be reduced to a
$\ZZ_2$-action. 
If $A'$ is another Bott matrix created from $A_1$ in Theorem \ref {dim-2}, we can check that the corresponding 
$((\ZZ_2)^s,M(A_1))$ ($s=1,2$) is equivariantly diffeomorphic to one of the actions $((\ZZ_2)^s,M(A_1))$ 
corresponding to $A_{49}$ and $A_{50}$ by the ad hoc argument. 
Once there exists such an equivariant diffeomorphism, $A'$ is equivalent to 
$A_{49}$ or $A_{50}$  by Theorem \ref{T2}.
This  argument works also for the case $A_2$ in Theorem \ref {dim-2}. 
As a consequence another Bott matrix $A'$ created  from
$A_2$ is equivalent to $A_{51}$ or $A_{52}$. 
Thus we obtain the following.
\begin{lemma}\label{lemT3}
A Bott matrix created from any one of Bott matrices in Theorem \ref{dim-2} is equivalent to 
one of the Bott matrices $A_{i}$ $(i=49,50,51,52)$ above.
\end{lemma}
\begin{pro}\label{proT3}
There are $4$ diffeomorphism classes of the case $T^3$-actions with $2$-dimensional
quotients.
\end{pro}

\subsection{$T^4$-actions with one-dimensional quotients}\par\

The Bott matrices of $M(A)$ admitting $T^4$-actions have the following
form
\begin{align*}
{
\left(\begin{array}{cccc|c}
&  \mbox {\large{$I_4$}} &  &  &\mbox {\large{$\ast $}} \\
\hline
 0 & 0  & 0 & 0 & 1\\
\end{array}\right)}.
\end{align*}
In this case $M(B)=M(1)=S^1$. 
It is easy to check by using moves {\bf II} and {\bf III},
it consists of just one diffeomorphism class,
where the corresponding  Bott matrix is
\begin{align*}{
A_{53}=\left(\begin{array}{cccc|c}
 & \mbox {\large{$I_4$}}    &   &  & \mbox {\Large{$0$}} \\
 &     &  &  & 1 \\
\hline
 0& 0    &0  & 0 & 1
\end{array}\right)}.
\end{align*}

Obviously the corresponding Bott matrix of size 5 of a real
Bott manifold admitting $T^5$-action is the identity matrix of rank 5.
Combined with Proposition \ref{proS1}, \ref{proT2}, \ref{proT3} and the case of  $T^4$-actions above
 we get the following theorem.
\begin{theorem}\label{dim-5}
The diffeomorphism classes of $5$-dimensional real Bott manifolds
consist of $54$.
\end{theorem}

\section{Classification of $n$-dimensional
Real Bott Manifolds}\label{four}

In this section we shall prove some results regarding
 the classification of
 certain types of $n$-dimensional real Bott manifolds.

\begin{theorem}\label{Alg1}
The number of diffeomorphism classes of $n$-dimensional
real Bott manifolds $(n\geq$ $4)$ which admit the maximal $T^{n-2}$-actions
$($\ie $s=1,2$ $)$ is $4$: $$M(A)=T^{(n-2)}\mathop{\times}_{(\ZZ_2)^s} M(B).$$
\end{theorem}
\begin{proof}
Since $M(B)$ is a 2-dimensional real Bott manifold,
the real Bott manifolds 
$M(A)$ created from $M(B)$ correspond to the following Bott matrices
\begin{align}\label{M1}{
\left (\begin{array}{c|c}
\mbox{\large$I_{n-2}$}  &          \mbox{\Large$*$} \\
\hline
\mbox{\large$0$}  & \mbox{\large$B_1$}  \\
\end{array}\right)},
\end{align}
\begin{align}\label{M2}{
\left (\begin{array}{c|c}
\mbox{\large$I_{n-2}$}  &          \mbox{\Large$*$} \\
\hline
\mbox{\large$0$}  & \mbox{\large$B_2$}  \\
\end{array}\right)}
\end{align}
where $B_1=I_2$, 
$B_2={
\left (\begin{array}{cc}
1& 1\\
0 & 1
\end{array}\right)}$.

Let us consider  (\ref{M1}).
If the entries in each row of  \mbox{\Large$*$} are the same then
by moves {\bf II} or {\bf III},  (\ref{M1}) is equivalent to
\begin{align}\label{newM1a}{
\left (\begin{array}{cc|cc}
\mbox{\large$I_{n-2}$}         &        &    \mbox{\Large$0$} & \mbox{\Large$0$}   \\
         &        & 1  & 1\\
\hline
         &        &  1  & 0 \\
\mbox{\Large$0$}  &        &  0  & 1 \\
\end{array}\right)}.
\end{align}
Otherwise by moves {\bf II}, {\bf III} or the equivariant diffeomorphism 
$\varphi \colon ((\ZZ_2)^2,$ $M(B_1))\to((\ZZ_2)^2,M(B_1))$ defined by 
$\varphi [z_{n-1},z_n]=[z_{n-1},z_n]$,
(\ref{M1}) is equivalent to
\begin{align}\label{newM1b}
{
\left (\begin{array}{cc|cc}
         &        &    \mbox{\Large$0$} & \mbox{\Large$0$}   \\
\mbox{\large$I_{n-2}$}  &        & 1  & 0\\
  &        & 0  & 1\\
\hline
         &        &  1  & 0 \\
\mbox{\Large$0$}  &        &  0  & 1 \\
\end{array}\right)}.
\end{align}
However (\ref{newM1a}) is not equivalent to (\ref{newM1b}) because
the $(\ZZ_2)^2$-action on $M(B_1)$ corresponding to  (\ref{newM1b})
cannot be reduced to a $\ZZ_2$-action on it.

Let us consider  (\ref{M2}). If the entries in each row of \mbox{\Large$*$}
are the same  then (\ref{M2}) is equivalent to
\begin{align}\label{newM2a-1}{
\left (\begin{array}{cc|cc}
\mbox{\large$I_{n-2}$}        &        & \mbox{\Large$0$} & \mbox{\Large$0$}      \\
         &        & 1  & 1\\
\hline
         &        &  1  & 1 \\
\mbox{\Large$0$}  &        &  0  & 1 \\
\end{array}\right)}
\end{align}
or if the entries in the second column of  \mbox{\Large$*$} are all zero  then
 (\ref{M2}) is equivalent to
\begin{align}\label{newM2a-2}{
\left (\begin{array}{cc|cc}
\mbox{\large$I_{n-2}$}        &        & \mbox{\Large$0$} & \mbox{\Large$0$}      \\
         &        & 1  & 0\\
\hline
         &        &  1  & 1 \\
\mbox{\Large$0$}  &        &  0  & 1 \\
\end{array}\right)}
\end{align}
by  moves {\bf II} or {\bf III}.
However (\ref{newM2a-1}) and
(\ref{newM2a-2}) are equivalent by the equivariant diffeomorphism
$\varphi\colon(\ZZ_2,M(B_2))\to (\ZZ_2,M(B_2))$ defined by
$\varphi([z_{n-1},z_n])$ $=[{\bf i}z_{n-1},z_n]$.
Otherwise (\ref{M2}) is equivalent to
\begin{align}\label{newM2b}{
\left (\begin{array}{cc|cc}
         &        &    \mbox{\Large$0$} & \mbox{\Large$0$}   \\
\mbox{\large$I_{n-2}$}  &        & 1  & 0\\
         &        & 0  & 1\\
\hline
         &        &  1  & 1 \\
\mbox{\Large$0$}  &        &  0  & 1 \\
\end{array}\right)}
\end{align}
by moves {\bf II}, {\bf III} or the equivariant diffeomorphism 
$\varphi \colon ((\ZZ_2)^2,$ $M(B_2))\to((\ZZ_2)^2,$ $M(B_2))$ defined by 
$\varphi [z_{n-1},z_n]=[z_{n-1},z_n]$.
Moreover  (\ref{newM2a-2}) and (\ref{newM2b}) are not equivalent
because the $(\ZZ_2)^2$-action on $M(B_2)$ corresponding to  (\ref{newM2b})
cannot be reduced to a $\ZZ_2$-action on it.

Therefore Bott matrices (\ref{M1}) and (\ref{M2}) give 4 distinct
diffeomorphism classes of real Bott manifolds $M(A)$.
\end{proof}

\begin{corollary}\label{cor1}
For any $n\geq 4$, the number of diffeomorphism classes of $n$-dimensio-nal
real Bott manifolds which admit the maximal $T^{k}$-actions
$(k=n-2,n-1,n)$ is $6$.
\end{corollary}

\begin{proof}
It is easy to check that real Bott manifolds
$M(A_1)=T^{n-1}\mathop{\times}_{\ZZ_2} S^1$ and
$M(A_2)=T^n$ have only one diffeomorphism class respectively,
because the corresponding Bott matrices are
$A_1 ={
\left (\begin{array}{c|c}
\mbox{\large$I_{n-1}$}      & \mbox{\Large$\ast $}  \\
         &        1      \\
\hline
\mbox{\Large$0$}         &  1 \\
\end{array}\right)}$ \nz and $A_2=I_n$ respectively.
Moreover,  $A_1 $ reduces to
$A'_1 ={
\left (\begin{array}{c|c}
\mbox{\large$I_{n-1}$}      & \mbox{\Large$0$}  \\
         &        1   \\
\hline
\mbox{\Large$0$}          &  1 \\
\end{array}\right)}$ \nz by move {\bf III}.
\end{proof}

\begin{corollary}
If $M(A)=S^1\mathop{\times}_{\ZZ_2} M(B)$
where $M(B)=T^k\mathop{\times}_{\ZZ_2} S^1$,
then for any $k\geq 1$ there is only
one diffeomorphism class.
\end{corollary}

\begin{proof}
Since $M(B)=T^k\mathop{\times}_{\ZZ_2} S^1$, as shown in the proof
of Corollary \ref{cor1}, 
\begin{align*}{
B=\left (\begin{array}{cccc}
   & \mbox{\large$I_k$} &    & \mbox{\Large$\ast  $}     \\
   &       &   & 1 \\
 0         & \dots    & 0     & 1
\end{array}\right)}.
\end{align*}
The Bott matrices $A$ created from $B$ are
\begin{align*}{
\left (\begin{array}{c|cccc}
 1        & 1 & \dots & 1 & 1 \\
\hline
\mbox{\Large$0 $}         &       & \mbox{\large$B$} &    &     \\
\end{array}\right) \text{\nz and}
\left (\begin{array}{c|cccc}
 1        & 1 & \dots & 1 & 0 \\
\hline
\mbox{\Large$0 $}         &       & \mbox{\large$B$} &    &      \\
\end{array}\right)}
\end{align*}
which are equivalent by
the equivariant diffeomorphism
$\varphi\colon(\ZZ_2,M(B))\to(\ZZ_2,$ $M(B))$ defined by
$\varphi([z_2,\dots,z_{k+1},z_{k+2}])=[z_2,\dots,{\bf i}z_{k+1},z_{k+2}]$.
\end{proof}

\begin{corollary}\label{Torus}
For any $k\geq 2$, there are 3 diffeomorphism classes in $(k+3)$-dimensional
real Bott manifolds $M(A)=S^1\mathop{\times}_{\ZZ_2} M(B)$
where  $M(B)=T^k\mathop{\times}_{(\ZZ_2)^s} T^2$ $(s=1,2)$.
\end{corollary}

\begin{proof}
Since $M(B)=T^k\mathop{\times}_{(\ZZ_2)^s} T^2$ ($s=1,2$), there are 2 distinct
diffeomorphism classes of $M(B)$
which correspond to the following Bott matrices
\begin{align}\label{MTorus}
B_1=\left (\begin{array}{cc|cc}
\mbox{\large$I_k$}         &        &    \mbox{\Large$0$} & \mbox{\Large$0$}   \\
         &        & 1  & 1\\
\hline
         &        &  1  & 0 \\
\mbox{\Large$0$}  &        &  0  & 1 \\
\end{array}\right),\,\,
B_2=\left (\begin{array}{cc|cc}
         &        &    \mbox{\Large$0$} & \mbox{\Large$0$}   \\
\mbox{\large$I_k$}         &        & 1  & 0\\
                           &        & 0  & 1\\
\hline
         &        &  1  & 0 \\
\mbox{\Large$0$}  &        &  0  & 1 \\
\end{array}\right)
\end{align}\nz
(see (\ref{newM1a}) and (\ref{newM1b})).

The Bott matrices of size $(k+3)$ created from $(\ref{MTorus})$
with the $\ZZ_2$-actions are as follows
\begin{align*}
A_1&={
\left (\begin{array}{c|ccccc}
 1        & 1 & \dots & 1 & \ast    & \ast  \\
\hline
\mbox{\Large$0 $}          &   &      &   \mbox{\large$B_1$}  &    &  \\
\end{array}\right)},
A_2={
\left (\begin{array}{c|ccccc}
 1        & 1 & \dots & 1 & \ast    & \ast  \\
\hline
\mbox{\Large$0 $}          &   &      &   \mbox{\large$B_2$}  &    &  \\
\end{array}\right)}.
\end{align*} \nz
The following Bott matrices in (\ref{ATorus1}) (resp. (\ref{ATorus2})) derived from
$A_1$
\begin{align}\label{ATorus1}
{
\left (\begin{array}{c|ccccc}
 1        & 1 & \dots & 1 & 0   & 0 \\
\hline
\mbox{\Large$0 $}          &   &      &   \mbox{\large$B_1$}  &    &  
\end{array}\right)},
{
\left (\begin{array}{c|ccccc}
 1        & 1 & \dots & 1 & 1   & 1 \\
\hline
\mbox{\Large$0 $}          &   &      &   \mbox{\large$B_1$}  &    &  
\end{array}\right)}\\ \label{ATorus2}
{
\left (\begin{array}{c|ccccc}
 1        & 1 & \dots & 1 & 1   & 0 \\
\hline
\mbox{\Large$0 $}          &   &      &   \mbox{\large$B_1$}  &    &  
\end{array}\right)},
{
\left (\begin{array}{c|ccccc}
 1        & 1 & \dots & 1 & 0   & 1 \\
\hline
\mbox{\Large$0 $}          &   &      &   \mbox{\large$B_1$}  &    &  
\end{array}\right)}
\end{align}\nz
are equivalent by the equivariant diffeomorphism
\begin{align*}
&\varphi\colon(\ZZ_2,M(B_1))\to (\ZZ_2,M(B_1))\\
&\varphi([z_2,\dots,z_{k+1},z_{k+2},z_{k+3}])=[z_2,\dots,{\bf i}z_{k+1},z_{k+2},z_{k+3}].
\end{align*}

Moreover, Bott matrices in (\ref{ATorus1}) are not equivalent to
(\ref{ATorus2}) because the maximal fixed point sets of $(\ZZ_2,M(B_1))$
 corresponding
to  the Bott matrices in (\ref{ATorus1}) and (\ref{ATorus2})
 are $T^2$ and $S^1$ respectively.

It is easy to see that each Bott matrix derived from $A_2$ is equivalent to
\begin{align}\label{ATorus3}
\left (\begin{array}{c|ccccc}
 1        & 1 & \dots & 1 & 0   & 0 \\
\hline
\mbox{\Large$0 $}          &   &      &   \mbox{\large$B_2$}  &    &  
\end{array}\right)
\end{align} \nz
by 
$\varphi\colon(\ZZ_2,M(B_2))\to (\ZZ_2,M(B_2))$ which is defined by 
one of the following 
\begin{align*}
\varphi([z_2,\dots,z_k,z_{k+1},z_{k+2},z_{k+3}])&=[z_2,\dots,{\bf i}z_k,z_{k+1},z_{k+2},z_{k+3}],\\
\varphi([z_2,\dots,z_k,z_{k+1},z_{k+2},z_{k+3}])&=[z_2,z_3,\dots,z_k,{\bf i}z_{k+1},z_{k+2},z_{k+3}],\\
\varphi([z_2,\dots,z_k,z_{k+1},z_{k+2},z_{k+3}])&=[z_2,z_3,\dots,{\bf i}z_k,{\bf i}z_{k+1},z_{k+2},z_{k+3}].
\end{align*}
Obviously,  (\ref{ATorus3}) is not equivalent to the Bott matrices in (\ref{ATorus1}) and
(\ref{ATorus2}) because they are created from two nonequivalent Bott matrices in (\ref{MTorus}).
Therefore there are 3 equivalence classes of the Bott matrices corresponding to  $M(A)$.
\end{proof}

\begin{corollary}\label{KBottle}
For any $k\geq 2$, there are 3 diffeomorphism classes in $(k+3)$-dimensional
real Bott manifolds $M(A)=S^1\mathop{\times}_{\ZZ_2} M(B)$
where  $M(B)=T^k\mathop{\times}_{(\ZZ_2)^s} K$, $($K=Klein bottle, $s=1,2)$.
\end{corollary}

\begin{proof}
Since $M(B)=T^k\mathop{\times}_{(\ZZ_2)^s} K$ $(s=1,2)$, there are 2 distinct
diffeomorphism classes of $M(B)$
corresponding to the following Bott matrices (see (\ref{newM2a-2}) and (\ref{newM2b}))
\begin{align}\label{MKB}
B_1=\left (\begin{array}{cc|cc}
\mbox{\large$I_k$}         &        &    \mbox{\Large$0$} & \mbox{\Large$0$}   \\
         &        & 1  & 0\\
\hline
         &        &  1  & 1 \\
\mbox{\Large$0$}  &        &  0  & 1 \\
\end{array}\right),   
B_2=\left (\begin{array}{cc|cc}
         &        &    \mbox{\Large$0$} & \mbox{\Large$0$}   \\
\mbox{\large$I_k$}         &        & 1  & 0\\
         &        & 0  & 1\\
\hline
         &        &  1  & 1 \\
\mbox{\Large$0$}  &        &  0  & 1 \\
\end{array}\right).
\end{align}\nz

The Bott matrices of size $(k+3)$ created from $(\ref{MKB})$
with the $\ZZ_2$-actions are as follows
\begin{align*}
A_1&={
\left (\begin{array}{c|ccccc}
 1        & 1 & \dots & 1 & \ast    & \ast  \\
\hline
\mbox{\Large$0 $}          &   &      &   \mbox{\large$B_1$}  &    &  
\end{array}\right)},
A_2={
\left (\begin{array}{c|ccccc}
 1        & 1 & \dots & 1 & \ast    & \ast  \\
\hline
\mbox{\Large$0 $}          &   &      &   \mbox{\large$B_2$}  &    &  
\end{array}\right)}.
\end{align*} \nz

The following Bott matrices in (\ref{AKB1}) (resp. (\ref{AKB2})) derived from $A_1$
\begin{align}\label{AKB1}
{
\left (\begin{array}{c|ccccc}
 1        & 1 & \dots & 1 & 0   & 0 \\
\hline
\mbox{\Large$0 $}          &   &      &   \mbox{\large$B_1$}  &    &  
\end{array}\right)},
{
\left (\begin{array}{c|ccccc}
 1        & 1 & \dots & 1 & 1   & 0 \\
\hline
\mbox{\Large$0 $}          &   &      &   \mbox{\large$B_1$}  &    &  
\end{array}\right)}\\ \label{AKB2}
{
\left (\begin{array}{c|ccccc}
 1        & 1 & \dots & 1 & 0   & 1 \\
\hline
\mbox{\Large$0 $}          &   &      &   \mbox{\large$B_1$}  &    &  
\end{array}\right)},
{
\left (\begin{array}{c|ccccc}
 1        & 1 & \dots & 1 & 1   & 1 \\
\hline
\mbox{\Large$0 $}          &   &      &   \mbox{\large$B_1$}  &    &  
\end{array}\right)}
\end{align}\nz
are equivalent by the equivariant diffeomorphism
\begin{align*}
&\varphi\colon(\ZZ_2,M(B_1))\to (\ZZ_2,M(B_1))\\
&\varphi([z_2,\dots,z_{k+1},z_{k+2},z_{k+3}])=[z_2,\dots,{\bf i}z_{k+1},z_{k+2},z_{k+3}].
\end{align*}

On the other hand, Bott matrices in (\ref{AKB1}) are not equivalent to
 (\ref{AKB2}) because the maximal fixed point sets
of $(\ZZ_2,M(B))$ corresponding
to  the Bott matrices in (\ref{AKB1}) and (\ref{AKB2})
 are $T^2$ and $S^1$ respectively.

It is easy to see that each Bott matrix derived from $A_2$ is equivalent to
\begin{align}\label{AKB3}
{
\left (\begin{array}{c|ccccc}
 1        & 1 & \dots & 1 & 0   & 0 \\
\hline
\mbox{\Large$0 $}          &   &      &   \mbox{\large$B_2$}  &    &  
\end{array}\right)}
\end{align} \nz
by $\varphi\colon(\ZZ_2,M(B_2))\to (\ZZ_2,M(B_2))$ which is defined by 
one of the following 
\begin{align*}
\varphi([z_2,\dots,z_k,z_{k+1},z_{k+2},z_{k+3}])&=[z_2,\dots,{\bf i}z_k,z_{k+1},z_{k+2},z_{k+3}],\\
\varphi([z_2,\dots,z_k,z_{k+1},z_{k+2},z_{k+3}])&=[z_2,z_3,\dots,z_k,{\bf i}z_{k+1},z_{k+2},z_{k+3}],\\
\varphi([z_2,\dots,z_k,z_{k+1},z_{k+2},z_{k+3}])&=[z_2,z_3,\dots,{\bf i}z_k,{\bf i}z_{k+1},z_{k+2},z_{k+3}].
\end{align*}
Obviously, the Bott matrix  (\ref{AKB3}) is not equivalent to the Bott matrices in (\ref{AKB1}) and
 (\ref{AKB2}) because they are created from two nonequivalent Bott matrices in (\ref{MKB}). 
Therefore there are 3 equivalence classes of the  Bott matrices corresponding to  $M(A)$.
\end{proof}

\begin{proposition}\label{prop1}
If $M(A)=T^k\mathop{\times}_{\ZZ_2} T^{n-k}$, then for any $n\geq 2$ and $k\geq 1$ there is only
one diffeomorphism class.
\end{proposition}

\begin{proof}
Since $M(A)$ admits the maximal $T^k$-action and  $A$ is created
from $I_{n-k}$, there is only one the 
Bott matrix $A$, namely
\begin{align}\label{A_one}
A={
\left (\begin{array}{c|ccc}
\mbox{\large$I_{k}$} & & \mbox{\Large$0$} & \\
         & 1 &  \dots & 1  \\
\hline
\mbox{\Large$0 $}          &   &   \mbox{\large$I_{n-k}$}    &
\end{array}\right)}.
\end{align}
\end{proof}

Now if we create Bott matrices from (\ref{A_one}) (for $k=1$) with $\ZZ_2$-actions then we will get a classification
of  the corresponding real Bott manifolds as follows.

\begin{theorem}\label{greatest}
For any $k\geq 2$, there are $[\frac{k}{2}]+1$
diffeomorphism classes in $(k+2)$-dimensional
real Bott manifolds $M(A_i)=S^1\mathop{\times}_{\ZZ_2} M(B)$
$(i=1,\dots,2^k)$,
where  $M(B)=S^1\mathop{\times}_{\ZZ_2} T^k$. 
Here $[x]$ is the Gauss
integer.
\end{theorem}

\begin{proof}
Since $M(B)=S^1\mathop{\times}_{\ZZ_2} T^k$,
\begin{equation}\label{M3}
B={
\left (\begin{array}{c|ccc}
   1      & 1 & \dots   & 1  \\
\hline 
  \mbox{\large$0$} &       &  \mbox{\large$I_k$}      &
\end{array}\right)}.
\end{equation}
The Bott matrices $A_i$ of size $(k+2)$ created from (\ref{M3}) with
$\ZZ_2$-actions are as follows
\begin{equation}\label{M4}
A_i={
\left (\begin{array}{c|ccc}
  1       & 1 & \mbox{\Large$\ast $}    &  \\
\hline
         &  1      &  \dots & 1   \\
  \mbox{\Large$0$} &       &  \mbox{\large$I_k$}      &
\end{array}\right)=
\left (\begin{array}{c|ccc}
  1       & 1 & \mbox{\Large$\ast $}    &  \\
\hline
\\
  \mbox{\Large$0$} &       &  \mbox{\large$B$}      &
\end{array}\right)}, \,\, (i=1,\dots,2^{k}).
\end{equation}
We apply the different $\ZZ_2$-actions on $M(B)$
such that the  Bott matrices $A_i$ are as follows
\begin{align}\label{MK}
\begin{split}
A_1&={ \left (\begin{array}{c|cccc}
 1        & \{1\}_{2} & \{0\}_{3} & \dots    & \{0\}_{2+k} \\
\hline
  \mbox{\Large$0$} &       &  \mbox{\large$B$}      &
\end{array}\right)},\\
A_2&={\left (\begin{array}{c|ccccc}
 1        & \{1\}_{2} & \{1\}_{3} & \{0\}_{4} & \dots    & \{0\}_{2+k} \\
\hline
  \mbox{\Large$0$} &    &   &  \mbox{\large$B$}      &
\end{array}\right)},\\
&\vdots\\
A_{[\frac{k}{2}]+1}&={\left (\begin{array}{c|cccccc}
 1        & \{1\}_{2} & \dots & \{1\}_{1+([\frac{k}{2}]+1)}& \{0\}_{2+([\frac{k}{2}]+1)} & \dots    & \{0\}_{2+k} \\
\hline
  \mbox{\Large$0$} &    &   &  \mbox{\large$B$}      &
\end{array}\right)}.
\end{split}
\end{align}
 It is easy to check that the maximal fixed point sets of
$(\ZZ_2,M(B))$ corresponding to $A_i$ ($i=1,2,\dots,[\frac{k}{2}]+1$) are
$T^k$, $T^{k-1}$, \dots, $T^{k-[\frac{k}{2}]}$ respectively. Hence 
they are not equivalent to each other. 
Here $\{y\}_i$ means $y$ in the $i$-th spot.

On the other hand, for ${[\frac{k}{2}]+1}< l\leq (k+1)$, Bott matrix
\begin{align}\label{ML} 
\left (\begin{array}{c|cccccc}
 1        & \{1\}_{2} & \dots & \{1\}_{1+l}& \{0\}_{2+l} & \dots    & \{0\}_{2+k} \\
\hline
  \mbox{\Large$0$} &    &   &  \mbox{\large$B$}      &
\end{array}\right)
\end{align}
is equivalent to one of the Bott matrices in (\ref{MK}). To show this, consider
the $g_1$-action corresponding to (\ref{ML}):
\begin{align*}
g_1([z_2,\dots,z_{1+l},z_{l+2},\dots,z_{k+2}])&=[\bar z_2,\dots,\bar z_{l+1},\overbrace{z_{l+2},\dots,z_{k+2}}^{(k+1)-l}]\\
                            &=[g_2(\bar z_2,\dots,\bar z_{l+1},z_{l+2},\dots,z_{k+2})]\\
                    &=[-\bar z_2,z_3,\dots,z_{l+1},\bar z_{l+2},\dots,\bar z_{k+2}].
\end{align*}
Since $(k+1)-l<k-[\frac{k}{2}]\leq [\frac{k}{2}]+1$, there is an 
  equivariant diffeomorphism $\varphi \colon(\ZZ_2,M(B))\to (\ZZ_2,M(B))$ defined by
$\varphi([z_2,\dots,z_{l+1},z_{l+2},\dots,z_{k+2}])=
[{\bf i}z_2,$ $z_{l+2},\dots,z_{k+2},z_{3},\dots,z_{l+1}]$
such that $\varphi g_1=h_1\varphi$ for some $h_1$-action corresponding to one of
the Bott matrices in (\ref{MK}).

The other Bott matrices $A_i$ $(i\not=1,\dots,[\frac{k}{2}]+1)$ may have the form
\begin{equation}\label{M4new}
A'={
\left (\begin{array}{c|cccc}
  1       & 1 & \hat 1 & \dots    & \hat 1  \\
\hline
  \mbox{\Large$0$} &   &    &  \mbox{\large$B$}      &
\end{array}\right)},
\end{equation}
where $\hat 1\in\{0,1\}$.
If the number of entries $\hat 1$ for $\hat 1=1$ is less than or equal to $[\frac{k}{2}]$ then
$A'$ is equivalent to one of the Bott matrices in $(\ref{MK})$, otherwise
$A'$ is equivalent to  $(\ref{ML})$. We shall prove it in the following way.
Suppose that the number of entries $\hat 1$ for $\hat 1=1$ is $t$. Applying move {\bf I}  on 
$A'$ such that the entries $\hat 1$ for $\hat 1=1$ are placed in series, 
we get a new Bott matrix 
\begin{equation*}\label{M4newnew}
A''={
\left (\begin{array}{c|ccccccc}
  1       & \{1\}_{2} & \{1\}_{3} & \dots    & \{1\}_{2+t} & \{0\}_{3+t} & \dots & \{0\}_{2+k} \\
\hline
  \mbox{\Large$0$} &   &   & &  \mbox{\large$B$}      &
\end{array}\right)},
\end{equation*}
which is still equivalent to $A'$.
Obviously, $A''=A_i$ for some $i=1,\dots,[\frac{k}{2}]+1$ if $0\leq t\leq [\frac{k}{2}]$,
or $A''$ is the same as  (\ref{ML}) if $t>[\frac{k}{2}]$.
Hence $A'$ is equivalent to one of the Bott matrices
in (\ref{MK}). This completes the proof of theorem.
\end{proof}


From now on, we use the notation $(\ZZ_2,M(B_j))_i$ which means that the $\ZZ_2$-action on $M(B_j)$
corresponds to a Bott matrix $A_{ij}$.

\begin{lemma}\label{lem1}
Let $M(A_{i1})=S^1\times_{\ZZ_2}M(B_1)$ $(i=1,\dots,n-2)$ be $n$-dimensional real Bott manifolds creating
from an $(n-1)$-dimensional real Bott manifold $M(B_1)$. Such real Bott manifolds $M(A_{i1})$ corresponding to
 $A_{i1}$ in \eqref{star1}
are not diffeomorphic to each other.
\begin{align}\label{star1}
\begin{split}
A_{11}&={\left (\begin{array}{c|cccc}
 1        & \{1\}_2 & \{0\}_3 & \dots    & \{0\}_{n} \\
\hline
  \mbox{\Large$0$} &       &  \mbox{\large$B_{1}$}      &
\end{array}\right)},\\
A_{21}&={ \left (\begin{array}{c|ccccc}
 1        & \{1\}_2 & \{1\}_3 & \{0\}_4 & \dots    & \{0\}_{n} \\
\hline
  \mbox{\Large$0$} &    &   &  \mbox{\large$B_{1}$}      &
\end{array}\right)},\\
&\vdots\\
A_{(n-k-1)1}&={\left (\begin{array}{c|ccccccc}
 1        & \{1\}_{2} & \{1\}_{3} & \dots & \{1\}_{1+(n-k-1)}& \{0\}_{1+(n-k)} & \dots & \{0\}_{n}\\
\hline
\mbox{\Large$0$}         &       &        &  \mbox{\large$B_{1}$}
\end{array}\right)},\\
&\vdots\\
A_{(n-2)1}&={ \left (\begin{array}{c|ccccc}
 1        & \{1\}_{2} & \{1\}_{3} & \dots & \{1\}_{1+(n-2)}& \{0\}_{n)}\\
\hline
\mbox{\Large$0$}         &       & & \mbox{\large$B_{1}$}      &
\end{array}\right)},
\end{split}
\end{align}
where
\begin{align}\label{matrixB1}
B_{1}={
\left (\begin{array}{ccccccc}
 1 & 1 & 1& \dots & \dots& \dots& 1 \\
   & 1 & 1& \dots & \dots& \dots& 1  \\
   &   & \ddots & & & & \vdots\\
   &   &        & 1 & 1 & \dots &  1\\
   &   &   &\\
   & \mbox{\Large$0$} & &  &  & \mbox{\large $I_k$}   &  \\
  &
\end{array}\right),}
\end{align}
$k\geq 2$, and $n-k\geq 3$.
\end{lemma}

\begin{proof}
Recall that,
if $M(A_{m1})$ is diffeomorphic to $M(A_{q1})$ (i.e.,  $A_{m1}$ is equivalent to $A_{q1}$, $m\not=q$),
by Theorem \ref{T2}, there is
an equivariant diffeomorphism
\begin{align*}
&(\Phi ,\varphi )\colon(\ZZ_2=<\al>,M(B_1))_m\to (\ZZ_2=<\be>,M(B_1))_{q}, \, \text{such that}\\
&\varphi(\al [z_2,\dots,z_n])=\Phi (\al)\varphi[z_2,\dots,z_n]=\beta\varphi[z_2,\dots,z_n].
\end{align*}
Let $\bar \varphi \colon\RR^{n-1}\to \RR^{n-1}$ be the lift of $\varphi$.
According to the form of $B_1$, the affine element $\bar \varphi$ has the form
\begin{align}\label{af_el_2}
\bar \varphi={ \left(\left (\begin{array}{c}
\\
{\bf a}\\
\\
\hline
\mbox{\bf b}
\end{array}\right),
\left (\begin{array}{ccc|c}
1 &  & \mbox{\Large$0$} &\\
 & \ddots &  & \mbox{\Large$0$}\\
\mbox{\Large$0$}  &  & 1 \\
\hline
& \mbox{\Large$0$}  &  & \mbox{\nz $D$}
\end{array}\right)\right)}
\end{align}
\nz  
where $D$ is a nonsingular submatrix of rank $k$, $^t{\bf a}=(a_{2},\dots,a_{n-k})$
and $^t{\bf b}=(b_{n-k+1},\dots,b_n)$ (see \eqref{finalform}).
Since $M(B_1)=T^{n-1}/(\ZZ_2)^{n-1}$, $\bar \varphi$ induces an affine map $\ti\varphi$ of $T^{n-1}$.

Put 
$X=
{
\begin{pmatrix}
x_{n-k+1}\\ \vdots \\ x_{n}
\end{pmatrix}} $. Since $\ti\varphi p=p\bar\varphi$, 
\begin{align}\label{aff_map}
\begin{split}
\ti\varphi(z_2,\dots,z_{n-k},z_{n-k+1},\dots,z_n)
   &=(\ell_2z_2,\dots,\ell_{n-k}z_{n-k},p\,^t({\bf b}+DX))\\
   & =(\ell_2z_2,\dots,\ell_{n-k}z_{n-k},c_{n-k+1}w_{n-k+1},\dots,c_{n}w_{n})
\end{split}
\end{align}
where
$\ell_p=exp(2\pi{\bf i}a_p)$ $(p=2,\dots,n-k)$,
$c_s=exp(2\pi{\bf i}b_s)$ $(s=n-k+1,\dots,n)$,
$(w_{n-k+1},\dots,w_{n})=p\, ^t(DX)$.

On the other hand, since $M(B_1)=T^{n-1}/(\ZZ_2)^{n-1}$,
the action $\langle\al\rangle$ lifts to an action
on $T^{n-1}$
such that we have the commutative diagram
\[\begin{CD}
(\ZZ_2)^{n-1} @. (\ZZ_2)^{n-1} \\
@VVV              @VVV\\
(\al,T^{n-1})   @>\ti \varphi>> (g\be,T^{n-1})\\
@VPrVV    @VPrVV\\
(\al,M(B_1))_{m}   @>\varphi>> (\be,M(B_1))_{q}
\end{CD}\]
for some $g\in(\ZZ_2)^{n-1}=\langle g_2,\dots,g_{n}\rangle$.
This means that 
\begin{align*}
Pr(\ti\varphi(\al (z_2,\dots,z_n)))&=\varphi(Pr(\al (z_2,\dots,z_n)))=\varphi(\al(Pr(z_2,\dots,z_n)))\\
                                 &=\Phi (\al)\varphi(Pr(z_2,\dots,z_n))=\be\varphi(Pr(z_2,\dots,z_n))\\
                                 &=\be Pr(\ti\varphi(z_2,\dots,z_n))=Pr(\be\ti\varphi(z_2,\dots,z_n)),
\end{align*}
(i.e.,
\begin{align}\label{g_beta}
\ti \varphi(\al(z_2,\dots,z_n))=g\be\ti \varphi(z_2,\dots,z_n).)
\end{align}
Note that 
$g_i$ $(i=2,\dots,n)$ corresponds to the $i$-th row of $A_{m1}$ and $A_{q1}$. 
This implies that $\ti\varphi$ maps the fixed point set of
$(\al,T^{n-1})$ to that of $(g\be,T^{n-1})$ diffeomorphically.
From the commutative diagram, we also have, for $g\in(\ZZ_2)^{n-1}$,
\begin{align*}
Pr(\ti\varphi(g(z_2,\dots,z_n)))&=\varphi(Pr(g(z_2,\dots,z_n)))
=\varphi(Pr(z_2,\dots,z_n))\\
&=Pr(\ti\varphi(z_2,\dots,z_n)).
\end{align*}
Hence there is an element $h\in(\ZZ_2)^{n-1}$
such that
\begin{align}\label{g_h}
\begin{split}
\ti \varphi(g(z_2,\dots,z_n))&=h\ti \varphi(z_2,\dots,z_n)\\
\ti \varphi g \ti \varphi^{-1}&=h.
\end{split}
\end{align}
\\

Recall from \eqref{star1}, that
\begin{align*}
A_{i1}&={\left (\begin{array}{c|ccccccc}
 1        & \{1\}_{2} & \{1\}_{3} & \dots & \{1\}_{1+i}& \{0\}_{2+i} & \dots & \{0\}_{n}\\
\hline
\mbox{\Large$0$}         &       &        &  \mbox{\large$B_{1}$}
\end{array}\right)},
\end{align*}
for $i=1,\dots,n-2$.
\\

To show that Bott matrices $A_{i1}$ $(i=1,\dots,n-2)$ are not equivalent
to each other, we shall prove the following cases.\\

\noindent{\bf a).} Bott matrices $A_{i1}$ $(i=1,\dots,n-k-1)$ are not equivalent
to each other.

Suppose that $A_{l1}$ is equivalent to $A_{p1}$ where $1\leq l< p\leq n-k-1$
$(l=1,\dots,n-k-2$; $p=2,\dots,n-k-1)$.
By the definition, the $\al$-action and $\beta$-action on $T^{n-1}$
corresponding  to $A_{l1}$ and $A_{p1}$
are as follows
\begin{align}\label{al_be_action1}
\begin{split}
\al(z_2,\dots,z_n)&=(\bar z_2,\dots,\bar z_{l+1},z_{l+2},\dots,z_{n}),\\
\be(z_2,\dots,z_n)&=(\bar z_2,\dots,\bar z_{l+1},\bar z_{l+2},\dots,\bar z_{p+1},z_{p+2},\dots,z_{n}).
\end{split}
\end{align}
Then from \eqref{g_beta}, we have
\begin{align}\label{equiv1}
\begin{split}
&\ti\varphi(\al(z_2,\dots,z_n))=g\be\ti\varphi(z_2,\dots,z_n),\\
&(\ell_2\bar z_2,\dots,\ell_{l+1}\bar z_{l+1},\ell_{l+2}z_{l+2},\dots,\ell_{n-k}z_{n-k},
      c_{n-k+1}w_{n-k+1},\dots,c_{n}w_{n})\\
&=g(\overline{\ell_2 z_2},\dots,\overline{\ell_{l+1}z_{l+1}},\dots,\overline{\ell_{p+1}z_{p+1}}
           ,\ell_{p+2}z_{p+2},\dots,\ell_{n-k}z_{n-k},\\
      &\hspace{.5cm}c_{n-k+1}w_{n-k+1},\dots,c_{n}w_{n})
\end{split}
\end{align}
for some $g\in(\ZZ_2)^{n-1}=\langle g_2,\dots,g_n\rangle$.
Since $$(c_{n-k+1}w_{n-k+1},\dots,c_{n}w_{n})=g(c_{n-k+1}w_{n-k+1},\dots,c_{n}w_{n}),$$
$g$ is a composition of an even number of  generators $\{g_2,\dots,g_{n-k}\}$.

On the other hand, since $g_t(z_t)=-z_t$ $(t=2,\dots,n-k)$
and
\begin{align}\label{eq}
(\ell_{l+2}z_{l+2},\dots,\ell_{n-k}z_{n-k})
=g(\overline{\ell_{l+2}z_{l+2}},\dots,\overline{\ell_{p+1}z_{p+1}}
           ,\ell_{p+2}z_{p+2},\dots,\ell_{n-k}z_{n-k}),
\end{align}
$g\not\in\langle g_{l+2},\dots,g_{n-k}\rangle $.
So, the last possibility is that $g\in\langle g_2,\dots,$ $g_{l+1}\rangle $.
If this is the case, and since $g$ is a composition of an even number of
 generators $\{g_2,\dots,g_{l+1}\}$, it contradicts \eqref{eq}.\\
 
\noindent{\bf b).} Bott matrices $A_{i1}$ ($i=n-k,\dots,n-2$) are not equivalent
to each other.

Suppose that $A_{(n-k+l)1}$ is equivalent to $A_{(n-k+p)1}$ where $n-k\leq n-k+l< n-k+p\leq n-2$
$(l=0,\dots,k-3$; $p=1,\dots,k-2)$.
By the definition, the $\al$-action and $\beta$-action on $T^{n-1}$
corresponding to $A_{(n-k+l)1}$ and $A_{(n-k+p)1}$ 
are as follows
\begin{align}\label{al_be_action1b}
\begin{split}
\al(z_2,\dots,z_n)&=
  (\bar z_2,\dots,\bar z_{n-k},\bar z_{n-k+1},\dots,\bar z_{n-k+l+1},z_{n-k+l+2},\dots,z_{n}),\\
\be(z_2,\dots,z_n)&=
  (\bar z_2,\dots,\bar z_{n-k},\bar z_{n-k+1},\dots,\bar z_{n-k+p+1},z_{n-k+p+2},\dots,z_{n}).
\end{split}
\end{align}
Then from \eqref{g_beta}, we have
\begin{align}\label{equiv1b}
\begin{split}
&\ti\varphi(\al(z_2,\dots,z_n))=g\be\ti\varphi(z_2,\dots,z_n),\\
&(\ell_2\bar z_2,\dots,\ell_{n-k}\bar z_{n-k},
      c_{n-k+1}u'_{n-k+1},\dots,c_{n}u'_{n})\\
&=g(\overline{\ell_2 z_2},\dots,\overline{\ell_{n-k}z_{n-k}},
      \overline{c_{n-k+1}w_{n-k+1}},\dots,\overline{c_{n-k+p+1}w_{n-k+p+1}},\\
     &\hspace{0.7cm} c_{n-k+p+2}w_{n-k+p+2},\dots,c_{n}w_{n})
\end{split}
\end{align}
for some $g\in(\ZZ_2)^{n-1}=\langle g_2,\dots,g_n\rangle$,
where
\begin{align*}
(u'_{n-k+1},\dots,u'_{n})
=p(
(-x_{n-k+1},\dots,-x_{n-k+l+1},x_{n-k+l+2},\dots,x_{n})
\,^tD).
\end{align*}

Now we consider the following cases for $g$.

\noindent{\bf b1).} If $g=g_tg'$ with $g_t\in\{g_2,\dots,g_{n-k-1}\}$, $g'\in\langle g_{t+1},\dots,g_n\rangle $
then
\begin{align*}
\ell_{t+1}\bar z_{t+1}=g(\overline {\ell_{t+1}z_{t+1}})=
\begin{cases}
-\ell_{t+1}z_{t+1} &\text{if $g=g_tg_{t+1}g''$, $g''\in\langle g_{t+2},\dots,g_n\rangle $}\\
\ell_{t+1}z_{t+1}  &\text{if $g=g_tg''$}.
\end{cases}
\end{align*}
Hence we get a contradiction. That is, such $g=g_tg'$ cannot occur.

\noindent{\bf b2).} If $g=g_{n-k}\hat g$ where $\hat g\in\langle g_{n-k+1},\dots,g_n\rangle $,
then 
$$\ell_{n-k}\bar z_{n-k}=g(\overline {\ell_{n-k}z_{n-k}})=-\overline {\ell_{n-k}z_{n-k}}.$$
This implies that $\ell_{n-k}=\pm {\bf i}$. 
Therefore
\begin{align*}
\ti\varphi(z_2,\dots,z_n)=&
(\ell_2z_2,\dots,\ell_{n-k-1}z_{n-k-1},\pm {\bf i}z_{n-k},
p\,^t({\bf b}+DX)),\\
\ti\varphi^{-1}(z_2,\dots,z_n)
=&(\bar \ell_2z_2,\dots,\bar\ell_{n-k-1}z_{n-k-1},\mp{\bf i}z_{n-k},p\,^t(-D^{-1}{\bf b}+D^{-1}X)).
\end{align*}
Now, from \eqref{g_h}, we consider
\begin{align*}
&\ti\varphi g_2 \ti\varphi^{-1}(z_2,\dots,z_n)\\
&=(-z_2,\ell_3^2\bar z_3,\dots,\ell_{n-k-1}^2\bar z_{n-k-1},-\bar z_{n-k},
exp(4\pi{\bf i}b_{n-k+1})\bar z_{n-k+1},\dots,exp(4\pi{\bf i}b_{n})\bar z_{n})\\
&=g_2(z_2,\overline{\ell_3^2}z_3,\dots,\overline{\ell_{n-k-1}^2}z_{n-k-1},-z_{n-k},\\
&\hspace{0.8cm}exp(4\pi{(-\bf i)}b_{n-k+1})z_{n-k+1},\dots,exp(4\pi{(-\bf i)}b_{n})z_{n})\\
&=g_2h(z_2,\dots,z_n)
\end{align*}
where
\begin{align}\label{h1}
\begin{split}
&h(z_2,\dots,z_n)=(z_2,\overline{\ell_3^2}z_3,\dots,\overline{\ell_{n-k-1}^2}z_{n-k-1},-z_{n-k},\\
&\hspace{2.5cm}exp(4\pi{(-\bf i)}b_{n-k+1})z_{n-k+1},\dots,exp(4\pi{(-\bf i)}b_{n})z_{n}).
\end{split}
\end{align}
We shall check that $h\not\in\langle g_3,\dots,g_n\rangle $ (i.e., $g_2h\not\in (\ZZ_2)^{n-1}$).

Suppose that $h\in\langle g_3,\dots,g_n\rangle $.
Since $h(z_{n-k})=-z_{n-k}$ in \eqref{h1}, we may write $h=\hat h g_{n-k}h''$,
where
$\hat h$ is a composition of an even
number of  generators $\{g_3,\dots,g_{n-k-1}\}$,
$h''\in\langle g_{n-k+1},$ $\dots,g_n\rangle $.
However such $h$ implies that
$h(z_{n-k+1},\dots,$ $z_{n})=(\pm \bar z_{n-k+1},\dots,\pm \bar z_{n})$.
This contradicts \eqref{h1}.
Similarly for $h=g_{n-k}h''$.
Thus $g_2h\not\in(\ZZ_2)^{n-1}$.
That is, such $g=g_{n-k}\hat g$ cannot occur.

\noindent{\bf b3).} If $g=\hat{g}$ satisfies \eqref{equiv1b}, then from \eqref{al_be_action1b},
\begin{align*}
&\hat g\be(z_2,\dots,z_n)=\\
&(\bar z_2,\dots,\bar z_{n-k},
      \pm \bar z_{n-k+1},\dots,\pm \bar z_{n-k+l+1},\dots,\pm \bar z_{n-k+p+1},\pm z_{n-k+p+2},\dots, \pm z_{n}).
\end{align*}
Then we obtain that
the fixed point set of $(\al,T^{n-1})$ is
\[Fix\,\al=(\{\pm 1\}_{2},\dots,\{\pm 1\}_{n-k+l+1},z_{n-k+l+2},\dots,z_{n})\]
and
the fixed point set of $(\hat g\be,T^{n-1})$ is
\[Fix\,\hat g\be=(\{\pm 1\}_{2},\dots,\{\pm 1\}_{n-k},\{\star\}_{n-k+1},\dots,\{\star\}_{n-k+p+1},z_{n-k+p+2},\dots,z_{n})\]
with $\star \in\{\pm 1, \pm {\bf i}\}$.
Then
by \eqref{g_beta}, we have
\begin{align*}
dim(Fix\,\al)&=dim(Fix\,g\be)\\
k-l-1&= k-p-1.
\end{align*}
Hence we get a contradiction. That is, such $g=\hat g$ cannot occur.\\

\noindent{\bf c).}
Each  $A_{i1}$ $(i=1,\dots,n-k-1)$ is not equivalent
to each  $A_{(n-k+j)1}$ $(j=0,\dots,k-2)$.

Suppose that $A_{i1}$ is equivalent to $A_{(n-k+j)1}$.
By the definition, the $\al$-action and  $\beta$-action on $T^{n-1}$
corresponding to $A_{i1}$ and $A_{(n-k+j)1}$ 
are as follows
\begin{align}\label{al_be_action1c}
\begin{split}
\al(z_2,\dots,z_n)&=
  (\bar z_2,\dots,\bar z_{i+1},z_{i+2},\dots,z_{n-k},\dots,z_{n}),\\
\be(z_2,\dots,z_n)&=
  (\bar z_2,\dots,\bar z_{n-k},\bar z_{n-k+1},\dots,\bar z_{n-k+j+1},z_{n-k+j+2},\dots,z_{n}).
\end{split}
\end{align}
Then by \eqref{g_beta}, we have
\begin{align}\label{equiv1c}
\begin{split}
&\ti\varphi(\al(z_2,\dots,z_n))=g\be\ti\varphi(z_2,\dots,z_n),\\
&(\ell_2\bar z_2,\dots,\ell_{i+1}\bar z_{i+1},\ell_{i+2}z_{i+2},\dots,\ell_{n-k}z_{n-k},
      c_{n-k+1}w_{n-k+1},\dots,c_{n}w_{n})\\
&=g(\overline{\ell_2 z_2},\dots,\overline{\ell_{n-k}z_{n-k}},
      \overline{c_{n-k+1}w_{n-k+1}},\dots,\overline{c_{n-k+j+1}w_{n-k+j+1}},\\
     &\hspace{0.7cm} c_{n-k+j+2}w_{n-k+j+2},\dots,c_{n}w_{n}).
\end{split}
\end{align}
for some $g\in(\ZZ_2)^{n-1}=\langle g_2,\dots,g_n\rangle $.

On the other hand, since
\begin{align*}
&(c_{n-k+1}w_{n-k+1},\dots,c_{n}w_{n})=\\
&g(\overline{c_{n-k+1}w_{n-k+1}},\dots,\overline{c_{n-k+j+1}w_{n-k+j+1}},
 c_{n-k+j+2}w_{n-k+j+2},\dots,c_{n}w_{n}),
\end{align*}
there is no $g\in\langle g_2,\dots,g_n\rangle $ satisfying \eqref{equiv1c}.
This completes the proof of Lemma.
\end{proof}


\begin{lemma}\label{lem2}
Let $M(A_{j'1})=S^1\times_{\ZZ_2}M(B_1)$ $(j=1,\dots,k,(k+1))$ be $n$-dimensional real Bott manifolds
creating from an $(n-1)$-dimensional real Bott manifold $M(B_1)$. Such real Bott manifolds $M(A_{j'1})$ corresponding to
Bott matrices derived from $A_{j'1}$ $(j=1,\dots,k,(k+1))$ in \eqref{star2}
are not diffeomorphic to each other.

\begin{align}\label{star2}
\begin{split}
&A_{({1'})1}={ \left (\begin{array}{c|cccccccc}
 1        & \{1\}_{2} & \{0\}_{3}& \{\hat 1\}_{4} & \dots & \{\hat 1\}_{(n-k)}& \{0\}_{} & \dots  & \{0\}_{(n-k)+k}\\
\hline
\mbox{\Large$0$}         &       &  & &\mbox{\large$B_{1}$}      &
\end{array}\right)},\\
&A_{({2'})1}={ \left (\begin{array}{c|ccccccccc}
 1        & \{1\}_{2} & \{0\}_{3}& \{\hat 1\}_{4} & \dots & \{\hat 1\}_{(n-k)}& \{1\}_{n-k+1}& \{0\}_{} & \dots  & \{0\}_{n}\\
\hline
\mbox{\Large$0$}         &       & & &\mbox{\large$B_{1}$}      &
\end{array}\right)},\\
&\vdots\\
&A_{({k'})1}=\\
&{ \left (\begin{array}{c|ccccccccc}
 1        & \{1\}_{2} & \{0\}_{3}& \{\hat 1\}_{4} & \dots & \{\hat 1\}_{(n-k)}& \{1\}_{n-k+1}&  \dots & \{1\}_{n-1} & \{0\}_{n}\\
\hline
\mbox{\Large$0$}         &       &  & &\mbox{\large$B_{1}$}      &
\end{array}\right)},\\
&A_{(k+1)'1}=\\
&{ \left (\begin{array}{c|ccccccccc}
 1        & \{1\}_{2} & \{0\}_{3}& \{\hat 1\}_{4} & \dots & \{\hat 1\}_{(n-k)}& \{1\}_{n-k+1}&  \dots & \{1\}_{n-1} & \{1\}_{n}\\
\hline
\mbox{\Large$0$}         &       &  & &\mbox{\large$B_{1}$}      &
\end{array}\right)}
\end{split}
\end{align}
where $\hat 1$ is either $0$ or $1$,
\begin{align}\label{except}
\begin{split}
(\{\hat 1\}_4,\dots,\{\hat 1\}_{n-k})&\not=(0,\dots,0)\\
(\text{resp}.\, (\{\hat 1\}_4,\dots,\{\hat 1\}_{n-k})&\not=(\overbrace{0,\dots,0}^{l},1,\dots,1),\,\,l=0,1,\dots,n-k-3)
\end{split}
\end{align}
for Bott matrix $A_{j'1}$ $(j=1,\dots,k)$ $($resp. $A_{(k+1)'1})$,
 the Bott matrix $B_1$ is as in \eqref{matrixB1}, $k\geq 2$, and $n-k\geq 3$.
$($That is, there are $k(2^{n-k-3}-1)+(2^{n-k-3}-(n-k-2))$
nonequivalent Bott matrices derived from \eqref{star2}.$)$
\end{lemma}

\begin{proof}
For brevity, we can write Bott matrices \eqref{star2} in this way
\begin{align*}
&A_{j'1}=\\
&{ \left (\begin{array}{c|ccccccccccc}
 1        & 1 & 0 & \hat 1 &
         \dots & \{\hat 1\}_{n-k}& \{1\}_{n-k+1}& \dots & \{1\}_{n-k+(j-1)} &\{0\}_{} & \dots  & \{0\}_{n}\\
\hline
\mbox{\Large$0$}         &       & & & & & \mbox{\large$B_{1}$}      &
\end{array}\right)}
\end{align*}
for $j=1,\dots,k,(k+1)$.

To show that Bott matrices derived  from $A_{j'1}$ ($j=1,\dots,k,(k+1)$)
in \eqref{star2}
are not equivalent to each other, we shall prove
the following claims by using the argument at the beginning
 of the proof of Lemma \eqref{lem1}.\\

\noindent{\bf Claim 1.}
Bott matrices $A_{j'1}$ $(j=1,\dots,k,(k+1))$ are not equivalent to each other.

Suppose that $A_{l'1}$ is equivalent to $A_{p'1}$ where $1\leq l< p\leq k+1$
$(l=1,\dots,k$; $p=2,\dots,k+1)$.
By the definition, the $\al$-action and $\beta$-action on $T^{n-1}$
corresponding to $A_{l'1}$ and $A_{p'1}$ 
are as follows
\begin{align}\label{al_be_action}
\begin{split}
&\al(z_2,\dots,z_n)=(\bar z_2,z_3,\overset{\tz{\al}}{\widehat{z_4}},\dots,\overset{\tz{\al}}{\widehat{z_{n-k}}},
      \bar z_{n-k+1},\dots,\bar z_{n-k+l-1},z_{n-k+l},\dots,z_{n}),\\
&\be(z_2,\dots,z_n)=\\
&(\bar z_2,z_3,\overset{\tz{\be}}{\widehat{z_4}},\dots,\overset{\tz{\be}}{\widehat{z_{n-k}}},
      \bar z_{n-k+1},\dots,\bar z_{n-k+l-1},\dots,\bar z_{n-k+p-1},z_{n-k+p},\dots,z_{n}).
\end{split}
\end{align}
Here $\overset{\tz{\alpha }}{\widehat{z_j}}(\in\{z_j,\bar z_j\})$ means an $\alpha$-action on $z_j$.
Similarly for $\overset{\tz{\beta}}{\widehat{z_j}}$. Note that $\hat z_j$ is either $z_j$ or $\bar z_j$
 depending on whether
$\hat 1$ is 0 or 1 respectively.
Then by \eqref{g_beta}, we have
\begin{align}\label{equiv}
\begin{split}
&\ti\varphi(\al(z_2,\dots,z_n))=g\be\ti\varphi(z_2,\dots,z_n),\\
&(\ell_2\bar z_2,\ell_3z_3,\ell_4\overset{\tz{\al}}{\widehat{z_4}},\dots,
        \ell_{n-k}\overset{\tz{\al}}{\widehat{z_{n-k}}},
      c_{n-k+1}w'_{n-k+1},\dots,c_{n}w'_{n})\\
&=g(\overline{\ell_2z_2},\ell_3z_3,\overset{\tz{\be}}{\widehat{\ell_4z_4}},\dots,
    \overset{\tz{\be}}{\widehat{\ell_{n-k}z_{n-k}}},\\
       &\overline{c_{n-k+1}w_{n-k+1}},\dots,\overline{c_{n-k+p-1}w_{n-k+p-1}},
      c_{n-k+p}w_{n-k+p},\dots,c_{n}w_{n}),
\end{split}
\end{align}
for some $g\in(\ZZ_2)^{n-1}=\langle g_2,\dots,g_n\rangle $,
where
\begin{align*}
(w'_{n-k+1},\dots,w'_{n})
=
& p(
(-x_{n-k+1},\dots,-x_{n-k+l-1},x_{n-k+l},\dots,x_{n})
\,^tD
).
\end{align*}
Obviously, $g\in \langle g_2,g_3\rangle$ does not satisfy \eqref{equiv}, because it implies that
\begin{align}\label{r1}
\ell_3 z_3=g(\ell_3 z_3)=
\begin{cases}
\overline{\ell_3 z_3}    &\text{if $g=g_2$}\\
-\ell_3 z_3              &\text{if $g=g_3$}\\
-\overline{\ell_3 z_3}   &\text{if $g=g_2g_3$.}
\end{cases}
\end{align}

Next we consider the following cases for $g\in\langle g_4,\dots,g_n\rangle$.\\

\noindent\textit {Case 1}. Let
$g=g_tg'$ where $g_t\in\{g_4,\dots,g_{n-k}\}$ $(t=4,\dots,n-k)$,
$g'\in\langle g_{t+1},\dots,g_{n}\rangle$.

Note that since $g_t(z_t)=-z_t$, $g(z_t)=-z_t$.
If $\overset{\tz{\al}}{\widehat{z_t}}\not=\overset{\tz{\be}}{\widehat{z_t}}$
(resp. $\overset{\tz{\al}}{\widehat{z_t}}=\overset{\tz{\be}}{\widehat{z_t}}=z_t$)
then $g=g_t{g'}$ does not
satisfy \eqref{equiv}, because it implies that
$\ell_t z_t=g(\overline{\ell_t z_t})=-\overline{\ell_t z_t}$
(resp. $\ell_t z_t=g({\ell_t z_t})=-\ell_t z_t$).
If $\overset{\tz{\al}}{\widehat{z_t}}=\overset{\tz{\be}}{\widehat{z_t}}=\bar z_t$ then
$\ell_t \bar z_t=g(\overline{\ell_t z_t})=-\overline{\ell_t z_t}$. This implies that
$\ell_t=\pm {\bf i}$.
Therefore
\begin{align*}
\ti\varphi(z_2,\dots,z_n)&=
(\ell_2z_2,\dots,\ell_{t-1}z_{t-1},\pm {\bf i}z_{t},\ell_{t+1}z_{t+1},\dots,
          \ell_{n-k}z_{n-k},
p\,^t({\bf b}+DX)),\\
\ti\varphi^{-1}(z_2,\dots,z_n)
&=(\bar \ell_2z_2,\dots,\bar\ell_{t-1}z_{t-1},
 \mp {\bf i}z_{t},\bar\ell_{t+1}z_{t+1},\dots,
          \bar\ell_{n-k}z_{n-k},\\
&\hspace{0.7cm} p\,^t(-D^{-1}{\bf b}+D^{-1}X)).\\
\end{align*}

Now, from \eqref{g_h}, we consider
\begin{align*}
&\ti\varphi g_2 \ti\varphi^{-1}(z_2,\dots,z_n)\\
&=(-z_2,\ell_3^2\bar z_3,\dots,\ell_{t-1}^2\bar z_{t-1},
 -\bar z_{t},\ell_{t+1}^2\bar z_{t+1},\dots,
          \ell_{n-k}^2\bar z_{n-k},\\
&\hspace{0.6cm}exp(4\pi{\bf i}b_{n-k+1})\bar z_{n-k+1},\dots,exp(4\pi{\bf i}b_{n})\bar z_{n})\\
&=g_2(z_2,\overline{\ell_3^2}z_3,\dots,\overline{\ell_{t-1}^2}z_{t-1},
 -z_{t},\overline{\ell_{t+1}^2}z_{t+1},\dots,
          \overline{\ell_{n-k}^2}z_{n-k},\\
&\hspace{1cm}exp(4\pi{(-\bf i)}b_{n-k+1})z_{n-k+1},\dots,exp(4\pi{(-\bf i)}b_{n})z_{n})\\
&=g_2h(z_2,\dots,z_n)
\end{align*}
where
\begin{align}\label{h2}
\begin{split}
h(z_2,\dots,z_n)&=(z_2,\overline{\ell_3^2}z_3,\overline{\ell_4^2}z_4,\dots,\overline{\ell_{t-1}^2}z_{t-1},
 -z_{t},\overline{\ell_{t+1}^2}z_{t+1},\dots,
          \overline{\ell_{n-k}^2}z_{n-k},\\
&\hspace{1cm}exp(4\pi{(-\bf i)}b_{n-k+1})z_{n-k+1},\dots,exp(4\pi{(-\bf i)}b_{n})z_{n}).
\end{split}
\end{align}
We shall check that $h\not\in\langle g_3,\dots,g_n\rangle$ (i.e., $g_2h\not\in (\ZZ_2)^{n-1}$).

Suppose that $h\in\langle g_3,\dots,g_n\rangle $.
Since
\[h(z_{n-k+1},\dots,z_n)=(exp(4\pi{(-\bf i)}b_{n-k+1})z_{n-k+1},\dots,
exp(4\pi{(-\bf i)}b_{n})z_{n})\] in \eqref{h2},
we may write $h=h'_{(even)}h''$ where $h'\in\langle g_3,\dots,g_{n-k}\rangle$ and
$h''\in\langle g_{n-k+1},$ $\dots,g_n\rangle$.
Here $h'_{(even)}$ means a composition of an even number of  generators $\{g_3,\dots,$ $g_{n-k}\}$.

On the other hand, since $h(z_t)=-z_t$ in \eqref{h2}, we may write $h'_{(even)}=\hat h g_t \Check h$,
where
$\hat h$ (resp. $\Check h$) is a composition of an even (resp. odd)
number of  generators $\{g_3,\dots,g_{t-1}\}$ (resp. $\{g_{t+1},\dots,g_{n-k}\}$).
For $t=4,5,\dots,n-k-1$, such $h'_{(even)}$ implies that
\begin{align*}
h(z_{t+1})=h'_{(even)}(z_{t+1})=
\begin{cases}
\bar z_{t+1}             &\text{if $h'_{(even)}=\hat h g_{t}\ddot{g}$,\, $\ddot{g}\in\langle g_{t+2},\dots,g_{n-k}\rangle$}\\
-\bar z_{t+1}            &\text{if $h'_{(even)}=\hat h g_{t}g_{t+1}\ddot{g}$}.
\end{cases}
\end{align*}
Hence this contradicts \eqref{h2}. Similarly for $h'_{(even)}=g_t \Check h$ $(t=4,5,\dots,n-k-1)$.

Now let us consider for $t=n-k$. Since $h(z_t)=-z_t$ in \eqref{h2}, $h=\dot{h}g_{n-k}h''$ where $\dot{h}$ is
a composition of an even number of generators $\{g_3,\dots,g_{n-k-1}\}$.
This implies that
$h(z_{n-k+1},\dots,z_{n})=(\pm \bar z_{n-k+1},\dots,\pm \bar z_{n})$.
This also contradicts \eqref{h2}. Similarly for $h=g_{n-k}h''$.
Thus $g_2h\not\in(\ZZ_2)^{n-1}$.
Hence \textit {Case 1} cannot occur.
\\

\noindent\textit {Case 2.}
Let $g=g''$ where $g''\in\langle g_{n-k+1},\dots,g_n\rangle$.

If $g=g''$ satisfies \eqref{equiv}, this implies that
$(\overset{\tz{\al}}{\widehat{z_4}},\dots,\overset{\tz{\al}}{\widehat{z_{n-k}}})
=(\overset{\tz{\be}}{\widehat{z_4}},\dots,\overset{\tz{\be}}{\widehat{z_{n-k}}})$.
Then, from \eqref{al_be_action},
\begin{align*}
&g''\be(z_2,\dots,z_n)=\\
&(\bar z_2,z_3,\overset{\tz{\be}}{\widehat{z_4}},\dots,\overset{\tz{\be}}{\widehat{z_{n-k}}},
      \pm \bar z_{n-k+1},\dots,\pm \bar z_{n-k+l-1},\dots,\pm \bar z_{n-k+p-1},\pm z_{n-k+p},\dots,
       \pm z_{n}).
\end{align*}
Then we obtain that
the fixed point set of $(\al,T^{n-1})$ is
\[Fix\,\al=(V,\{\pm 1\}_{n-k+1},\dots,\{\pm 1\}_{n-k+l-1},z_{n-k+l},\dots,z_{n})\]
with $V=\{(z_2,\dots,z_{n-k})|\al(z_2,\dots,z_{n-k})=(z_2,\dots,z_{n-k})\}$
and
the fixed point set of $(g''\be,T^{n-1})$ is
\[Fix\,g''\be=(W,\{\star\}_{n-k+1},\dots,\{\star\}_{n-k+p-1},z_{n-k+p},\dots,z_{n})\]
with 
$$W=\{(z_2,\dots,z_{n-k})|g''\be(z_2,\dots,z_{n-k})=\be(z_2,\dots,z_{n-k})=(z_2,\dots,z_{n-k})\}$$ 
and
$\star \in\{\pm 1, \pm {\bf i}\}$.
Since $(\overset{\tz{\al}}{\widehat{z_4}},\dots,\overset{\tz{\al}}{\widehat{z_{n-k}}})
=(\overset{\tz{\be}}{\widehat{z_4}},\dots,\overset{\tz{\be}}{\widehat{z_{n-k}}})$,
$dimV=dimW$.
Then
by \eqref{g_beta}, we have
\begin{align*}
dim(Fix\,\al)&=dim(Fix\,g\be)\\
dim V +(k-l+1)&= dim W +(k-p+1).
\end{align*}
Hence we get a contradiction.
That is, the \textit {Case 2} cannot occur.
This completes the proof of Claim 1.
\\

\noindent{\bf Claim 2.}
Bott matrices derived from each $A_{j'1}$ $(j=1,\dots,k)$ are not equivalent
to each other, (i.e.,
there are $(2^{n-k-3}-1)$ nonequivalent Bott matrices derived from
each $A_{j'1}$ $(j=1,\dots,k)$).

Associated with the entries ($\{\hat 1\}_4,\dots,\{\hat 1\}_{n-k})$ of each $A_{j'1}$,
there are
$2^{(n-k-3)}-1$ different actions $(\ZZ_2,M(B_1))_{j'}$.
(Note that ($\{\hat 1\}_4,$ $\dots,\{\hat 1\}_{n-k})\not=$
$(0,\dots,0)$.)

We prove that every two different actions $(\ZZ_2,M(B_1))_{j'}$
derived from $A_{j'1}$ (denoted by
$(\al,M(B_1))_{j'_{\al}}$ and $(\be,M(B_1))_{j'_{\be}}$ respectively),
the corresponding Bott matrices (denoted by $A_{j'_\al 1}$ and $A_{j'_\be 1}$ respectively)
are not equivalent.

Since the $\al$-action and  $\be$-action are different,
we may assume that $\overset{\tz{\al}}{\widehat{z_i}}=z_i$, $\overset{\tz{\be}}{\widehat{z_i}}=\bar z_i$,
for some $i\in\{4,\dots,n-k\}$. Then
\begin{align}
\begin{split}
&\al(z_2,\dots,z_n)=\\
&(\bar z_2,z_3,\overset{\tz{\al}}{\widehat{z_4}},\dots,\overset{\tz{\al}}{\widehat{z_{i-1}}},{z_i},
            \overset{\tz{\al}}{\widehat{z_{i+1}}},\dots,
                   \overset{\tz{\al}}{\widehat{z_{n-k}}},
      \bar z_{n-k+1},\dots,\bar z_{n-k+j-1},z_{n-k+j},\dots,z_{n}),\\
&\be(z_2,\dots,z_n)=\\
&(\bar z_2,z_3,\overset{\tz{\be}}{\widehat{z_4}},\dots,\overset{\tz{\be}}{\widehat{z_{i-1}}},
    {\bar z_i},\overset{\tz{\be}}{\widehat{z_{i+1}}},\dots,\overset{\tz{\be}}{\widehat{z_{n-k}}},
      \bar z_{n-k+1},\dots,\bar z_{n-k+j-1},z_{n-k+j},\dots,z_{n}),\\
&\text {for}\,\, i=4,\dots,n-k, \, j=1,\dots,k.
\end{split}
\end{align}
As before, if  $A_{j'_\al 1}$ and $A_{j'_\be 1}$ are equivalent, we have
\begin{align}\label{equiv2}
\begin{split}
&\ti\varphi(\al(z_2,\dots,z_n))=g\be\ti\varphi(z_2,\dots,z_n),\\
&(\ell_2\bar z_2,\ell_3z_3,\ell_4\overset{\tz{\al}}{\widehat{z_4}},\dots,
            ,\ell_{i-1}\overset{\tz{\al}}{\widehat{z_{i-1}}},{\ell_iz_i},
            \ell_{i+1}\overset{\tz{\al}}{\widehat{z_{i+1}}},\dots,
        \ell_{n-k}\overset{\tz{\al}}{\widehat{z_{n-k}}},
      c_{n-k+1}v'_{n-k+1},\\ &\dots,c_{n}v'_{n})\\
&=g(\overline{\ell_2z_2},\ell_3z_3,\overset{\tz{\be}}{\widehat{\ell_4z_4}},\dots,
    \overset{\tz{\be}}{\widehat{\ell_{i-1}z_{i-1}}},
    {\overline {\ell_iz_i}},\overset{\tz{\be}}{\widehat{\ell_{i+1}z_{i+1}}},\dots,
    \overset{\tz{\be}}{\widehat{\ell_{n-k}z_{n-k}}},\\
       &\hspace{0.7cm}\overline{c_{n-k+1}w_{n-k+1}},\dots,\overline{c_{n-k+j-1}w_{n-k+j-1}},
      c_{n-k+j}w_{n-k+j},\dots,c_{n}w_{n}),
\end{split}
\end{align}
for some $g\in(\ZZ_2)^{n-1}=\langle g_2,\dots,g_n\rangle$,
where
\begin{align}\label{bar-varphi2}
\begin{split}
 (v'_{n-k+1},\dots,v'_{n})
=
 p(
(-x_{n-k+1},
\dots,-x_{n-k+j-1},x_{n-k+j},\dots,x_{n})\,
^tD
).
\end{split}
\end{align}

Obviously, $g\in\langle g_2,g_3\rangle $ does not satisfy \eqref{equiv2} because of \eqref{r1}.
If $g\in\langle g_{i},\dots,g_n\rangle $ then the equation \eqref{equiv2} is also not satisfied, because it implies that
\begin{align}\label{r2}
\ell_i z_i=g(\overline{\ell_i z_i})=
\begin{cases}
-\overline{\ell_i z_i}   &\text{if $g=g_ig'$ where $g'=\langle g_{i+1},\dots,g_n\rangle $}\\
\overline{\ell_i z_i}    &\text{if $g=g'$.}
\end{cases}
\end{align}
Because of \eqref{r1} and \eqref{r2}, if $i=4$ then there is no $g\in(\ZZ_2)^{n-1}$
satisfying \eqref{equiv2}.

Now let us consider for $i\in\{5,\dots,n-k\}$.
Since $\overset{\tz{\al}}{\widehat{z_i}}=z_i$ and
$\overset{\tz{\be}}{\widehat{z_i}}=\bar z_i$, for some $i\in\{5,\dots,n-k\}$,
 we may write
$g=g_t\dot{g}^{(even)}\ddot{g}$ with $g_t\in\{g_{4},\dots,g_{i-1}\}$ $(t=4,\dots,i-1)$,
$\dot{g}\in\langle g_{t+1},\dots,g_{i-1}\rangle $, $\ddot{g}\in\langle g_{i+1},\dots,g_{n}\rangle $
(if $t=i-1$ then $g=g_t\ddot{g}$).
Here $\dot{g}^{(even)}$ means a composition of an even number of
 generators of $\dot{g}$.

Note that since $g_t(z_t)=-z_t$, $g(z_t)=-z_t$.
If $\overset{\tz{\al}}{\widehat{z_t}}\not=\overset{\tz{\be}}{\widehat{z_t}}$
(resp. $\overset{\tz{\al}}{\widehat{z_t}}=\overset{\tz{\be}}{\widehat{z_t}}=z_t$)
then $g=g_t\dot{g}^{(even)}\ddot{g}$ does not
satisfy \eqref{equiv2}, because it implies that
$\ell_t z_t=g(\overline{\ell_t z_t})=-\overline{\ell_t z_t}$
(resp. $\ell_t z_t=g({\ell_t z_t})=-\ell_t z_t$).
If $\overset{\tz{\al}}{\widehat{z_t}}=\overset{\tz{\be}}{\widehat{z_t}}=\bar z_t$ then
$\ell_t \bar z_t=g(\overline{\ell_t z_t})=-\overline{\ell_t z_t}$. This implies that
$\ell_t={\pm \bf i}$.
Therefore
\begin{align*}
\ti\varphi(z_2,\dots,z_n)&=
(\ell_2z_2,\dots,\ell_{t-1}z_{t-1},{\pm \bf i}z_{t},\ell_{t+1}z_{t+1},\dots,
          \ell_{n-k}z_{n-k},p\,^t({\bf b}+DX)).
\end{align*}
Similar to the proof of Claim 1, one can check that
 $\ti\varphi g_2 \ti\varphi^{-1}(z_2,\dots,z_n)\not\in(\ZZ_2)^{n-1}$. Hence we have a contradiction.
That is, there is no $g\in(\ZZ_2)^{n-1}$ satisfying \eqref{equiv2}.
\\

Since all combinations of ($\{\hat 1\}_4,\dots,\{\hat 1\}_{n-k})$
with ($\{\hat 1\}_4,\dots,\{\hat 1\}_{n-k})\not=(0,\dots,$ $0)$ are different for each $A_{j'1}$,
there are $(2^{n-k-3}-1)$ nonequivalent Bott matrices derived from
each $A_{j'1}$ $(j=1,\dots,k)$.
\\

\noindent{\bf Claim 3.}
Bott matrices derived from $A_{(k+1)'1}$ are not equivalent
to each other, (i.e.,
there are $2^{n-k-3}-(n-k-2)$ nonequivalent Bott matrices derived from
$A_{(k+1)'1})$.\\
Since there are $2^{n-k-3}-(n-k-2)$ different combination of ($\{\hat 1\}_4,\dots,\{\hat 1\}_{n-k})$ with 
 $$(\{\hat 1\}_4,\dots,\{\hat 1\}_{n-k})\not =
(\overbrace {0,\dots,0}^{l},1,\dots,1),$$
(that is, there are $2^{n-k-3}-(n-k-2)$ different actions $(\ZZ_2,M(B_1))_{(k+1)'}$),
by using argument in the proof of Claim 2 above,
there are $2^{n-k-3}-(n-k-2)$ nonequivalent Bott matrices derived from
$A_{(k+1)'1}$.

According to Claim 1, 2, 3, we obtain that there are $k(2^{n-k-3}-1)+(2^{n-k-3}-(n-k-2))$ 
nonequivalent Bott matrices derived from \eqref{star2}.
\end{proof}

\begin{remark}\label{rem_lem}
Consider Bott matrices $A_{i1}$ $(i=1,\dots,n-2)$ in \eqref{star1} and
 $A_{j'1}$ $(j=1,\dots,k,k+1)$ in \eqref{star2}. 
\begin{itemize}
\item[(i)]
Associated with the entries $(\{\hat 1\}_4,\dots,\{\hat 1\}_{n-k})$ in 
$A_{j'1}$ $(j=1,\dots,k)$, 
if  $(\{\hat 1\}_4,\dots,\{\hat 1\}_{n-k})=(0,\dots,0)$ then
$A_{1'1}=A_{11}$ and $A_{j'1}\sim A_{(n-j)1}$ $(j=2,\dots,k)$ by the equivariant
diffeomorphism
\[(\Phi ,\varphi )\colon(\ZZ_2,M(B_1))_{j'}\to (\ZZ_2,M(B_1))_{n-j}\]
defined by
\begin{align*}
&[z_2,\dots,z_{n-k},z_{n-k+1},\dots,z_{n-k+(j-1)},z_{n-k+j},\dots,z_n]
\stackrel{\varphi}{\longmapsto  }\\
&[{\bf i}z_2,\dots,z_{n-k},z_{n-k+j},\dots,z_n,z_{n-k+1},\dots,z_{n-k+(j-1)}].
\end{align*}
\item[(ii)] 
Associated with the entries $(\{\hat 1\}_4,\dots,\{\hat 1\}_{n-k})$ in 
$A_{(k+1)'1}$, 
if \\
$(\{\hat 1\}_4,\dots,\{\hat 1\}_{n-k})=(\overbrace {0,\dots,0}^{l},1,\dots,1)$ $
(l=0,1,\dots,n-k-3),$
 then
$A_{(k+1)'1}$ $\sim A_{(l+2)1}$ $(l=0,1,\dots,n-k-3)$
by the equivariant diffeomorphism
\[(\Phi ,\varphi)\colon(\ZZ_2,M(B_1))_{(k+1)'}\to (\ZZ_2,M(B_1))_{(l+2)}\]
defined by $$\varphi([z_2,\dots,z_n])=[{\bf i}z_2,\dots,z_n].$$
\end{itemize}
\end{remark}


\begin{lemma}\label{lem3}
Each Bott matrix in \eqref{star1} is not equivalent to each Bott matrix derived from \eqref{star2}.
\end{lemma}

\begin{proof}
Using the argument at the beginning
 of the proof of Lemma \eqref{lem1} we shall prove the following Claims.
\\

\noindent{\bf Claim i).}
 $A_{l1}$ $(l=1,\dots,n-k-1)$ is not equivalent to  $A_{1'1}$.

For this, we prove the following cases.\\

\noindent\textit{Case 1}. $A_{11}$ is not equivalent to
$A_{1'1}$.\\
By the definition, the  actions $(\ZZ_2,M(B_1))_1$ and
$(\ZZ_2,M(B_1))_{1'}$ are
\begin{align}\label{11}
\al[z_2,\dots,z_n]=[\bar z_2,z_3,\dots,z_n]
\end{align}
and
\begin{align}\label{1'1}
\be[z_2,\dots,z_n]=[\bar z_2,z_3,\hat z_4,\dots,\hat z_{n-k},z_{n-k+1},\dots,z_n]
\end{align}
respectively. If $A_{11}$ is equivalent to
$A_{1'1}$, by \eqref{g_beta}, there is $g\in(\ZZ_2)^{n-1}$ such that
\begin{align}\label{11-1'1}
\begin{split}
&\ti\varphi(\al(z_2,\dots,z_n))=g\be \ti\varphi(z_2,\dots,z_n),\\
&(\ell_2 \bar z_2,\ell_3z_3,\dots,\ell_{n-k}z_{n-k},c_{n-k+1}w_{n-k+1},\dots,c_{n}w_n)\\
&=g(\overline{\ell_2z_2},\ell_3z_3,\widehat{\ell_4z_4},\dots,\widehat{\ell_{n-k}z_{n-k}},
c_{n-k+1}w_{n-k+1},\dots,c_{n}w_n).
\end{split}
\end{align}

As before, $g\in\langle g_2,g_3\rangle$ does not satisfy \eqref{11-1'1} because of \eqref{r1}.
Since $$(c_{n-k+1}w_{n-k+1},\dots,c_{n}w_n)=g(c_{n-k+1}w_{n-k+1},\dots,c_{n}w_n),$$
$g$ is a composition of an even number of  generators $\{g_4,\dots,g_{n-k}\}$
 and $g\in\langle g_{n-k+1},$ $\dots,g_n\rangle$ does not satisfy \eqref{11-1'1}.
Then we can take $g=g_tg'$ with $t=4,\dots,n-k-1$ and
$g'\in\langle g_{t+1},\dots,g_{n-k}\rangle$. This implies that
\begin{align*}
\ell_tz_t=g(\widehat{\ell_tz_t})=
\begin{cases}
-\ell_tz_t                & \text {if $\widehat{\ell_tz_t}=\ell_tz_t$} \\
-\overline{\ell_tz_t}     & \text {if $\widehat{\ell_tz_t}=\overline{\ell_tz_t}$}.
\end{cases}
\end{align*}
Therefore there is no $g\in(\ZZ_2)^{n-1}=\langle g_2,\dots,g_n\rangle$ satisfying \eqref{11-1'1}.\\

\noindent\textit{Case 2}. $A_{j1}$ ($j=2,\dots,n-k-1$) is not equivalent to
$A_{1'1}$.\\
By the definition, the  actions $(\ZZ_2,M(B_1))_j$ and
$(\ZZ_2,M(B_1))_{1'}$ are
\begin{align}\label{j1}
\begin{split}
\al[z_2,\dots,z_n]&=[\bar z_2,\bar z_3,\dots,\bar z_{j+1},z_{j+2},\dots,z_n]\\
                &=[g_2g_{j+1}(\bar z_2,\bar z_3,\dots,\bar z_{j+1},z_{j+2},\dots,z_n)]\\
            &=[-\bar z_2,z_3,\dots,z_{j},-z_{j+1},z_{j+2},\dots,z_n]
\end{split}
\end{align}
and as in \eqref{1'1}
respectively.
If $A_{j1}$ is equivalent to
$A_{1'1}$, there is $g\in(\ZZ_2)^{n-1}$ such that
\begin{align}\label{j1-1'1}
\begin{split}
&\ti\varphi(\al(z_2,\dots,z_n))=g\be \ti\varphi(z_2,\dots,z_n),\\
&(-\ell_2 \bar z_2,\ell_3z_3,\dots,\ell_{j}z_j,-\ell_{j+1}z_{j+1},\ell_{j+2}z_{j+2},
         \dots,\ell_{n-k}z_{n-k},c_{n-k+1}w_{n-k+1},\\
& \dots,c_{n}w_n)\\
&=g(\overline{\ell_2z_2},\ell_3z_3,\widehat{\ell_4z_4},\dots,\widehat{\ell_{n-k}z_{n-k}},
c_{n-k+1}w_{n-k+1},\dots,c_{n}w_n).
\end{split}
\end{align}

Similar to the argument of the proof of \textit{Case 1} above,
$g\in\langle g_2,\dots,g_j\rangle$ $(j=2,\dots,n-k-1)$ does not satisfy \eqref{j1-1'1}.
Since
\begin{align}\label{tb1}
(c_{n-k+1}w_{n-k+1},\dots,c_{n}w_n)=g(c_{n-k+1}w_{n-k+1},\dots,c_{n}w_n)
\end{align}
in \eqref{j1-1'1},
$g\in\langle g_{n-k+1},\dots,g_n\rangle$ does not satisfy \eqref{j1-1'1}.
Because of
\begin{align*}
(-\ell_{j+1}z_{j+1})=
\begin{cases}
g(\ell_{j+1}z_{j+1})           & \text{if $j=2$}\\
g(\widehat{\ell_{j+1}z_{j+1}}) & \text{if $j=3,\dots,n-k-1$}
\end{cases}
\end{align*}
and \eqref{tb1},
then $g$ is a composition of an even number of  generators  $\{g_{j+1},$ $ \dots,$ $g_{n-k}\}$
which can be written by $g=g_{j+1}\dot{g}$ with $j=2,\dots,n-k-2$, 
$\dot{g}\in\langle g_{j+2},$ $\dots,g_{n-k}\rangle$.
However this implies that
\begin{align*}
\ell_{t}z_{t}=g(\widehat{\ell_{t}z_{t}})=
-\widehat{\ell_{t}z_{t}}    
\end{align*}
for some $t$ $(j+2\le t\le n-k)$.
For $j=n-k-1$, $g=g_{n-k}$ does not satisfy \eqref{j1-1'1}, because
by assumption \eqref{except},
$(\hat z_4,\dots,\hat z_n)\not=(z_4,\dots,z_n)$ in \eqref{1'1}.
Hence there is no $g\in(\ZZ_2)^{n-1}=\langle g_2,\dots,g_n\rangle$ satisfying \eqref{j1-1'1}.
This completes the proof of Claim i).
\\

\noindent{\bf Claim ii).}  $A_{l1}$ $(l=1,\dots,n-k-1)$ is not equivalent to
 $A_{j'1}$ $(j=2,\dots,k)$.

For this, we prove the following cases.\\

\noindent\textit{Case 1}. $A_{11}$ is not equivalent to
$A_{j'1}$ ($j=2,\dots,k$).\\
By the definition, the  actions $(\ZZ_2,M(B_1))_1$ and
$(\ZZ_2,M(B_1))_{j'}$ are as in \eqref{11}
and
\begin{align}\label{j'1}
\be[z_2,\dots,z_n]=[\bar z_2,z_3,\hat z_4,\dots,\hat z_{n-k},\bar z_{n-k+1},\dots,\bar z_{n-k+j-1},z_{n-k+j},\dots,z_n]
\end{align}
respectively.
If $A_{11}$ is equivalent to
$A_{j'1}$, there is $g\in(\ZZ_2)^{n-1}$ such that
\begin{align}\label{11-j'1}
\begin{split}
&\ti\varphi(\al(z_2,\dots,z_n))=g\be \ti\varphi(z_2,\dots,z_n),\\
&(\ell_2 \bar z_2,\ell_3z_3,\dots,\ell_{n-k}z_{n-k},c_{n-k+1}w_{n-k+1},\dots,c_{n}w_n)\\
&=g(\overline{\ell_2z_2},\ell_3z_3,\widehat{\ell_4z_4},\dots,\widehat{\ell_{n-k}z_{n-k}},
\overline{c_{n-k+1}w_{n-k+1}},\dots,\overline{c_{n-k+j-1}w_{n-k+j-1}},\\
&\hspace{.8cm}c_{n-k+j}w_{n-k+j},\dots,c_{n}w_n).
\end{split}
\end{align}
Since
\begin{align}\label{k-part}
\begin{split}
&(c_{n-k+1}w_{n-k+1},\dots,c_{n}w_n)=\\
&g(\overline{c_{n-k+1}w_{n-k+1}},\dots,\overline{c_{n-k+j-1}w_{n-k+j-1}},
c_{n-k+j}w_{n-k+j},\dots,c_{n}w_n),
\end{split}
\end{align}
it is clear that
there is no $g\in(\ZZ_2)^{n-1}=\langle g_2,\dots,g_n\rangle $ satisfying \eqref{11-j'1}.\\

\noindent\textit{Case 2}. $A_{i1}$ ($i=2,\dots,n-k-1$) is not equivalent to
$A_{j'1}$ ($j=2,\dots,k$).\\
By the definition, the  actions $(\ZZ_2,M(B_1))_i$ and
$(\ZZ_2,M(B_1))_{j'}$ are as in \eqref{j1}
and \eqref{j'1}
respectively.
If $A_{i1}$ is equivalent to
$A_{j'1}$ then there is $g\in(\ZZ_2)^{n-1}$ such that
\begin{align}\label{i1-j'1}
\begin{split}
&\ti\varphi(\al(z_2,\dots,z_n))=g\be \ti\varphi(z_2,\dots,z_n),\\
&(-\ell_2 \bar z_2,\ell_3z_3,\dots,\ell_{i}z_i,-\ell_{i+1}z_{i+1},\ell_{i+2}z_{i+2},
         \dots,\ell_{n-k}z_{n-k},c_{n-k+1}w_{n-k+1},\dots,\\ & c_{n}w_n)\\
&=g(\overline{\ell_2z_2},\ell_3z_3,\widehat{\ell_4z_4},\dots,\widehat{\ell_{n-k}z_{n-k}},
\overline{c_{n-k+1}w_{n-k+1}},\dots,\overline{c_{n-k+j-1}w_{n-k+j-1}},\\
&\hspace{.8cm}c_{n-k+j}w_{n-k+j},\dots,c_{n}w_n).
\end{split}
\end{align}
Because of \eqref{k-part},
there is no $g\in(\ZZ_2)^{n-1}=\langle g_2,\dots,g_n\rangle$ satisfying \eqref{i1-j'1}.
This completes the proof of Claim ii).
\\

\noindent{\bf Claim iii).}  $A_{l1}$ $(l=1,\dots,n-k-1)$ is not equivalent to
 $A_{(k+1)'1}$.

For this, we prove the following cases.\\

\noindent\textit{Case 1}. $A_{11}$ is not equivalent to
$A_{(k+1)'1}$.\\
By the definition, the  actions $(\ZZ_2,M(B_1))_1$ and
$(\ZZ_2,M(B_1))_{(k+1)'}$ are as in \eqref{11}
and
\begin{align}\label{1''1}
\be[z_2,\dots,z_n]=[\bar z_2,z_3,\hat z_4,\dots,\hat z_{n-k},\bar z_{n-k+1},\dots,\bar z_n]
\end{align}
respectively. If $A_{11}$ is equivalent to
$A_{(k+1)'1}$ then there is $g\in(\ZZ_2)^{n-1}$
such that
\begin{align}\label{11-1''1}
\begin{split}
&\ti\varphi(\al(z_2,\dots,z_n))=g\be \ti\varphi(z_2,\dots,z_n),\\
&(\ell_2 \bar z_2,\ell_3z_3,\dots,\ell_{n-k}z_{n-k},c_{n-k+1}w_{n-k+1},\dots,c_{n}w_n)\\
&=g(\overline{\ell_2z_2},\ell_3z_3,\widehat{\ell_4z_4},\dots,\widehat{\ell_{n-k}z_{n-k}},
\overline{c_{n-k+1}w_{n-k+1}},\dots,\overline{c_{n}w_n}).
\end{split}
\end{align}

As before, $g\in\langle g_2,g_3\rangle$ does not satisfy \eqref{11-1''1} because of \eqref{r1}.
Since
$$(c_{n-k+1}w_{n-k+1},\dots,c_{n}w_n)=g(\overline{c_{n-k+1}w_{n-k+1}},\dots,\overline{c_{n}w_n)},$$
$g$ is a composition of an odd number of  generators $\{g_4,\dots,g_{n-k}\}$
and $g\in\langle g_{n-k+1},$ $\dots,g_{n}\rangle$ does not satisfy \eqref{11-1''1}.
Then we can take $g=g_tg'$, with $t=4,\dots,n-k-1$ and
$g'\in\langle g_{t+1},\dots,g_{n-k}\rangle$. This implies that
\begin{align*}
\ell_tz_t=g(\widehat{\ell_tz_t})=
\begin{cases}
-\ell_tz_t                & \text {if $\widehat{\ell_tz_t}=\ell_tz_t$} \\
-\overline{\ell_tz_t}     & \text {if $\widehat{\ell_tz_t}=\overline{\ell_tz_t}$}.
\end{cases}
\end{align*}
For $t=n-k$, it is clear that $g=g_{n-k}$ does not also satisfy \eqref{11-1''1}.
Hence there is no $g\in(\ZZ_2)^{n-1}=\langle g_2,\dots,g_n\rangle$ satisfying \eqref{11-1''1}.
\\

\noindent\textit{Case 2}. $A_{j1}$ ($j=2,\dots,n-k-2$) is not equivalent to
$A_{(k+1)'1}$.\\
By the definition, the  actions $(\ZZ_2,M(B_1))_j$ and
$(\ZZ_2,M(B_1))_{(k+1)'}$ are as in \eqref{j1}
and  \eqref{1''1}
respectively.
If $A_{j1}$ is equivalent to
$A_{(k+1)'1}$ then there is $g\in(\ZZ_2)^{n-1}$
such that
\begin{align}\label{j1-1''1}
\begin{split}
&\ti\varphi(\al(z_2,\dots,z_n))=g\be \ti\varphi(z_2,\dots,z_n),\\
&(-\ell_2 \bar z_2,\ell_3z_3,\dots,\ell_{j}z_j,-\ell_{j+1}z_{j+1},\ell_{j+2}z_{j+2},
         \dots,\ell_{n-k}z_{n-k},c_{n-k+1}w_{n-k+1},\dots,\\ &c_{n}w_n)\\
&=g(\overline{\ell_2z_2},\ell_3z_3,\widehat{\ell_4z_4},\dots,\widehat{\ell_{n-k}z_{n-k}},
\overline{c_{n-k+1}w_{n-k+1}},\dots,\overline{c_{n}w_n}).
\end{split}
\end{align}
Note that by assumption (see \eqref{except}),
\begin{align}\label{ass}
\begin{split}
(z_3,\hat z_4,\dots,\hat z_{j+1},\hat z_{j+2},\dots,\hat z_{n-k})&\not=
(z_3,z_4,\dots,z_{j+1},\bar z_{j+2},\dots,\bar z_{n-k}),\\
(\hat z_4,\dots,\hat z_{n-k})&\not=
(z_4,\dots,z_{n-k})
\end{split}
\end{align}
in \eqref{1''1} for $j=2,\dots,n-k-2$.

Similar to the argument of the proof of \textit{Case 1} above,
$g\in\langle g_2,\dots,g_j\rangle$ $(j=2,\dots,n-k-1)$ does not satisfy \eqref{j1-1''1}.
Since
\begin{align}\label{tb2}
(c_{n-k+1}w_{n-k+1},\dots,c_{n}w_n)=g(\overline{c_{n-k+1}w_{n-k+1}},\dots,\overline{c_{n}w_n})
\end{align}
in \eqref{j1-1''1},
$g\in\langle g_{n-k+1},\dots,g_{n}\rangle$ does not satisfy \eqref{j1-1''1}.
Because of
\begin{align*}
(-\ell_{j+1}z_{j+1})=
\begin{cases}
g(\ell_{j+1}z_{j+1})           & \text{if $j=2$}\\
g(\widehat{\ell_{j+1}z_{j+1}}) & \text{if $j=3,\dots,n-k-2$}
\end{cases}
\end{align*}
and \eqref{tb2},
then $g$ is a composition of an odd number of  generators  $\{g_{j+1},\dots,$  $g_{n-k}\}$
which can be written as $g=g_{j+1}\dot{g}$  $(j=2,\dots,n-k-2)$, $\dot{g}\in\langle g_{j+2},\dots,g_{n-k}\rangle$.
However this implies that
\begin{align*}
\ell_{t}z_{t}=g(\widehat{\ell_{t}z_{t}})=
-\widehat{\ell_{t}z_{t}}    
\end{align*}
for some $t$ $(j+2\le t\le n-k)$.
Hence such $g$ does not satisfy \eqref{j1-1''1}. In particular, $g=g_{j+1}$ $(j=2,\dots,n-k-1)$
does not also
satisfy \eqref{j1-1''1} because of \eqref{ass}.
Thus there is no $g\in(\ZZ_2)^{n-1}=\langle g_2,\dots,g_n\rangle$ satisfying \eqref{j1-1''1}.
\\

\noindent\textit{Case 3}. $A_{(n-k-1)1}$ is not equivalent to
$A_{(k+1)'1}$.

By the definition, the  actions $(\ZZ_2,M(B_1))_{(n-k-1)}$ and
$(\ZZ_2,$ $M(B_1))_{(k+1)'}$ are
\begin{align}\label{(n-k-1)1}
\begin{split}
\al[z_2,\dots,z_n]&=[\bar z_2,\bar z_3,\dots,\bar z_{n-k},z_{n-k+1},\dots,z_n]\\
                &=[g_2(\bar z_2,\bar z_3,\dots,\bar z_{n-k},z_{n-k+1},\dots,z_n)]\\
            &=[-\bar z_2,z_3,\dots,z_{n-k},\bar z_{n-k+1},\dots,\bar z_n)]
\end{split}
\end{align}
and as in \eqref{1''1}
respectively.
If $A_{(n-k-1)1}$ is equivalent to
$A_{(k+1)'1}$ then there is $g\in(\ZZ_2)^{n-1}$
such that
\begin{align}\label{(n-k-1)1-1''1}
\begin{split}
&\ti\varphi(\al(z_2,\dots,z_n))=g\be \ti\varphi(z_2,\dots,z_n),\\
&(-\ell_2 \bar z_2,\ell_3z_3,\dots,\ell_{n-k}z_{n-k},c_{n-k+1}\bar w_{n-k+1},\dots,c_{n}\bar w_n)\\
&=g(\overline{\ell_2z_2},\ell_3z_3,\widehat{\ell_4z_4},\dots,\widehat{\ell_{n-k}z_{n-k}},
\overline{c_{n-k+1}w_{n-k+1}},\dots,\overline{c_{n}w_n}).
\end{split}
\end{align}
Note that by assumption \eqref{except}, $(\hat z_4,\dots,\hat z_{n-k})\not=
(z_4,\dots,z_{n-k})$ in \eqref{1''1}.
Therefore $g\in\langle g_{n-k+1},\dots,g_n\rangle$ does not satisfy \eqref{(n-k-1)1-1''1}.

Since $$(c_{n-k+1}\bar w_{n-k+1},\dots,c_{n}\bar w_n)=g(\overline{c_{n-k+1}w_{n-k+1}},\dots,\overline{c_{n}w_n}),$$
 $g$ is a composition of an even number of  generators $\{g_4,\dots,g_{n-k}\}$.
Similar to the argument in \textit{Case 1} (Claim i)), such $g$ does not satisfy \eqref{(n-k-1)1-1''1}.
Thus there is no $g\in(\ZZ_2)^{n-1}=\langle g_2,\dots,g_n\rangle$ satisfying \eqref{(n-k-1)1-1''1}.
This completes the proof of Claim iii).
\\

\noindent{\bf Claim iv).}  $A_{l1}$ $(l=n-k,\dots,n-2)$ is not equivalent to 
 $A_{j'1}$ ($j=1,\dots,k,k+1$).

For this, we prove the following cases.\\

\noindent\textit{Case 1}. $A_{l1}$ $(l=n-k,\dots,n-2)$ is not equivalent to
$A_{1'1}$.\\
By the definition, the  actions $(\ZZ_2,M(B_1))_l$ and
$(\ZZ_2,M(B_1))_{1'}$ are
\begin{align}\label{l1}
\begin{split}
\al[z_2,\dots,z_n]&=[\bar z_2,\bar z_3,\bar z_4,\dots,\bar z_{n-k},
   \bar z_{n-k+1},\dots,\bar z_{1+l},z_{2+l},\dots,z_{n} ] \\
&=[g_2(\bar z_2,\bar z_3,\bar z_4,\dots,\bar z_{n-k},
   \bar z_{n-k+1},\dots,\bar z_{1+l},z_{2+l},\dots,z_{n}) ] \\
&=[-\bar z_2,z_3,z_4,\dots,z_{n-k},
   z_{n-k+1},\dots,z_{1+l},\bar z_{2+l},\dots,\bar z_{n} ] \\
\end{split}
\end{align}
and as in \eqref{1'1}
respectively.
If $A_{l1}$ is equivalent to
$A_{1'1}$, 
there is $g\in(\ZZ_2)^{n-1}$
such that
\begin{align}\label{l1-1'1}
\begin{split}
&\ti\varphi(\al(z_2,\dots,z_n))=g\be \ti\varphi(z_2,\dots,z_n),\\
&(-\ell_2 \bar z_2,\ell_3z_3,\ell_4z_4,\dots,\ell_{n-k}z_{n-k},c_{n-k+1}y'_{n-k+1},\dots,c_{n}y'_n)\\
&=g(\overline{\ell_2z_2},\ell_3z_3,\widehat{\ell_4z_4},\dots,\widehat{\ell_{n-k}z_{n-k}},
c_{n-k+1}w_{n-k+1},\dots,c_{n}w_n)
\end{split}
\end{align}
where
\begin{align*}
&(y'_{n-k+1},
\dots,
y'_{n})
=
p(
(x_{n-k+1},
\dots,
x_{1+l},
-x_{2+l},
\dots,
-x_{n})
\,^tD
).
\end{align*}
Note that by assumption \eqref{except}, $$(\hat z_4,\dots,\hat z_{n-k})\not=(z_4,\dots,z_{n-k})$$
in \eqref{1'1}.
Therefore $g\in\langle g_{n-k+1},\dots,g_n\rangle$ does not occur.
As before $g\in\langle g_2,g_3\rangle$ does not also  satisfy \eqref{l1-1'1}.

Now we consider $g\in\langle g_4,\dots,g_{n}\rangle$. We may write
$g=g_tg'g''$ with $t=4,\dots,n-k-1$, $g'\in\langle g_{t+1},\dots,g_{n-k}\rangle$,
$g''\in\langle g_{n-k+1},\dots,g_{n}\rangle$ (if $t=n-k$ then $g=g_{n-k}g''$).
This implies that
\begin{align*}
\ell_tz_t=g(\widehat{\ell_tz_t})=
\begin{cases}
-\ell_tz_t                & \text {if $\widehat{\ell_tz_t}=\ell_tz_t$} \\
-\overline{\ell_tz_t}     & \text {if $\widehat{\ell_tz_t}=\overline{\ell_tz_t}$}.
\end{cases}
\end{align*}
Therefore
there is no $g\in(\ZZ_2)^{n-1}=\langle g_2,\dots,g_n\rangle$ satisfying \eqref{l1-1'1}.
\\

\noindent\textit{Case 2}. $A_{l1}$ ($l=n-k,\dots,n-2$) is not equivalent to
$A_{j'1}$ ($j=2,\dots,k$).\\
By the definition, the  actions $(\ZZ_2,M(B_1))_l$ and
$(\ZZ_2,M(B_1))_{j'}$ are as in \eqref{l1}
and \eqref{j'1}
respectively.
If $A_{l1}$ is equivalent to
$A_{j'1}$ then there is $g\in(\ZZ_2)^{n-1}$
such that
\begin{align}\label{l1-j'1}
\begin{split}
&\ti\varphi(\al(z_2,\dots,z_n))=g\be \ti\varphi(z_2,\dots,z_n),\\
&(-\ell_2 \bar z_2,\ell_3z_3,\ell_4z_4,\dots,\ell_{n-k}z_{n-k},c_{n-k+1}y'_{n-k+1},\dots,c_{n}y'_n)\\
&=g(\overline{\ell_2z_2},\ell_3z_3,\widehat{\ell_4z_4},\dots,\widehat{\ell_{n-k}z_{n-k}},
\overline{c_{n-k+1}w_{n-k+1}},\dots,\overline{c_{n-k+j-1}w_{n-k+j-1}},\\
&\hspace{.8cm}c_{n-k+j}w_{n-k+j},\dots,c_{n}w_n).
\end{split}
\end{align}
Note that $$(\hat z_4,\dots,\hat z_{n-k})\not=(z_4,\dots,z_{n-k})$$ in \eqref{j'1}.
Similar to the previous argument,
there is no $g\in(\ZZ_2)^{n-1}=\langle g_2,\dots,g_n\rangle$ satisfying \eqref{l1-j'1}.
\\

\noindent\textit{Case 3}. $A_{l1}$ ($l=n-k,\dots,n-2$) is not equivalent to
$A_{(k+1)'1}$.\\
By the definition, the  actions $(\ZZ_2,M(B_1))_{l}$ and
$(\ZZ_2,M(B_1))_{(k+1)'}$ are as in \eqref{l1}
and \eqref{1''1}
respectively.
If $A_{l1}$ is equivalent to
$A_{(k+1)'1}$ then there is $g\in(\ZZ_2)^{n-1}$
such that
\begin{align}\label{l1-1''1}
\begin{split}
&\ti\varphi(\al(z_2,\dots,z_n))=g\be \ti\varphi(z_2,\dots,z_n),\\
&(-\ell_2 \bar z_2,\ell_3z_3,\ell_4z_4,\dots,\ell_{n-k}z_{n-k},c_{n-k+1}y'_{n-k+1},\dots,c_{n}y'_n)\\
&=g(\overline{\ell_2z_2},\ell_3z_3,\widehat{\ell_4z_4},\dots,\widehat{\ell_{n-k}z_{n-k}},
\overline{c_{n-k+1}w_{n-k+1}},\dots,\overline{c_{n}w_n}).
\end{split}
\end{align}
\\

Note that by assumption
$$(\hat z_4,\dots,\hat z_{n-k})\not=(z_4,\dots,z_{n-k})$$  in \eqref{1''1}.
Similar to the argument in \textit{Case 1} (Claim iv)) above,
there is no $g\in(\ZZ_2)^{n-1}=\langle g_2,\dots,g_n\rangle$ satisfying \eqref{l1-1''1}.
This completes the proof of Claim iv).
\end{proof}

\begin{theorem}\label{type1..1k}
Let $M(A)=S^1\times_{\ZZ_2}M(B)$ be an $n$-dimensional real Bott
manifold. Suppose that $B$ is either one of the list in
\eqref{bottB}. Then $M(B)$ are diffeomorphic each other and the
number of diffeomorphism classes of such real Bott manifolds $M(A)$
above is $(k+1)2^{n-k-3}$ $(k\geq 2$, $n-k\geq 3)$.

\begin{align}\label{bottB}
\begin{split}
B_{1}&={
\left (\begin{array}{ccccccc}
 1 & 1 & 1& \dots & \dots& \dots& 1 \\
   & 1 & 1& \dots & \dots& \dots& 1  \\
   &   & \ddots & & & & \vdots\\
   &   &        & 1 & 1 & \dots &  1\\
   &   &   &\\
   & \mbox{\large$0$} & &  &  & \mbox{\large $I_k$}   &  \\
  &
\end{array}\right)},
B_{2}={
\left (\begin{array}{ccccccc}
 1 & 1 & 0& \dots & \dots& \dots& 0 \\
   & 1 & 1& \dots & \dots& \dots& 1  \\
   &   & \ddots & & & & \vdots\\
   &   &        & 1 & 1 & \dots &  1\\
   &   &   &\\
   & \mbox{\large $0$} & &  &  & \mbox{\large $I_k$}   &  \\
  &
\end{array}\right)},\\
&\vdots\\
&B_{n-k-1}={
\left (\begin{array}{cccccccc}
 1 & 1 & 0& \dots & \dots& \dots& \dots& 0 \\
   & 1 & 1& 0 & \dots& \dots& \dots & 0  \\
   &   & \ddots & & & & & \vdots\\
   &   &        & 1 & 1 & 0 & \dots &  0\\
   &   &        &   & 1 & 1 & \dots &  1\\
   &   &   &\\
   & \mbox{\large$0$} & &  & &  & \mbox{\large $I_k$}   &  \\
  &
\end{array}\right)},\\
&B_{n-k}={
\left (\begin{array}{cccccccccc}
 1 & 1 & 0& \dots & \dots& \dots& \dots & \dots& \dots& 0 \\
   & 1 & 1& 0 & \dots& \dots& \dots & \dots& \dots & 0  \\
   &   & \ddots & & & & & & &\vdots\\
   &   &        & 1 & 1 & 0 & 0 &0 & \dots &  0\\
   &   &        &  &1 & 1 & 1 & 0 & \dots &  0\\
   &   &        &  & & 1 & 1 & 0 & \dots &  0\\
   &   &        &  & &   & 1 & 1 & \dots &  1\\
   &   &   &\\
   & \mbox{\large$0$} & & & & &   & &  \mbox{\large $I_k$}   &  \\
  &
\end{array}\right)},\\
&\vdots\\
&B_{n-k+(n-k-4)}={
\left (\begin{array}{cccccccc}
 1 & 1 & \dots & \dots & 1 & 0& \dots& 0 \\
   & 1 & \dots& \dots & 1 & 0& \dots & 0  \\
   &   & \ddots & &\vdots &\vdots & & \vdots\\
   &   &        & 1 & 1 & 0 & \dots &  0\\
   &   &        &   & 1 & 1 & \dots &  1\\
   &   &   &\\
   & \mbox{\large$0$} & &  & &  & \mbox{\large $I_k$}   &  \\
  &
\end{array}\right)}.
\end{split}
\end{align}
\end{theorem}

\begin{proof}
Note that each $B_j$ is of size $n-1$ and  $M(B_j)=S^1\times_{\ZZ_2}M(B'_j)$ for some 
$(n-2)$-dimensional real Bott manifold $M(B'_j)$. 
$M(B_{j-1})$ is diffeomorphic to $M(B_j)$ 
($j=2,3,\dots,n-k-1$ ) by the equivariant diffeomorphism 
$\varphi \colon(\ZZ_2,M(B'_{j-1}))$ $\to (\ZZ_2,M(B'_{j}))$
which is defined by
$\varphi([z_2,\dots,z_{n-1}])=[z_2,\dots,{\bf i}z_j,\dots,z_{n-1}]$. 
$M(B_a)$ is diffeomorphic to $M(B_b)$ 
($a=n-k+(n-k-3)-(j-1)$; $b=n-k+(n-k-3)-(j-2)$; $j=n-k-1,n-k-2,\dots,4,3$ )
 by the equivariant diffeomorphism 
$\varphi :(\ZZ_2,M(B'_{a}))\rightarrow (\ZZ_2,M(B'_{b}))$
which is defined by
$\varphi([z_2,\dots,z_{n-1}])=[z_2,\dots,{\bf i}z_{j-1},\dots,z_{n-1}]$.

By the hypothesis, there are $2^{n-2}$ possible
$\ZZ_2$-actions on each $M(B_j)$, $j=1,\dots,2(n-k)-4$.
We shall prove that among $2^{n-2}\times (2(n-k)-4)$ real Bott manifolds
created from $M(B_j)$ $(j=1,\dots,2(n-k)-4)$,
there are only
$(k+1)2^{n-k+3}$ diffeomorphism classes.

First of all we show that there are $(k+1)2^{n-k+3}$ diffeomorphism classes
of real Bott manifolds $M(A_{i1})$ created from $M(B_1)$ by $\ZZ_2$-actions.

If a Bott matrix $A'$  created from $B_1$  is different from the Bott matrices in \eqref{star1} and \eqref{star2}, 
it is easy to  check that the corresponding 
$(\ZZ_2,M(B_1))$ is equivariantly diffeomorphic to one of the actions $(\ZZ_2,M(B_1))$ 
corresponding to the Bott matrices in \eqref{star1} or \eqref{star2} by the ad hoc argument. 
(Compare Section \ref{EXAMPEL}
for the argument to find an equivariant diffeomorphism.)
Once there exists such an equivariant diffeomorphism, $A'$ is equivalent to one of 
$A_{i1}$'s ($i=1,\dots,n-2$) in \eqref{star1} or $A_{j'1}$'s ($j=1,\dots,k,(k+1)$) in \eqref{star2}
  by Theorem \ref{T2}. 
Then, because of Lemma \ref{lem1}, \ref{lem2}, \ref{lem3} and Remark \ref{rem_lem},
there are $(n-2)+ (2^{n-k-3}-1)k+2^{n-k-3}-(n-k-2)=(k+1)2^{n-k+3}$ distinct diffeomorphism classes
of $M(A_{i1})$\ $(i=1,\dots,2^{n-2})$ created from $M(B_1)$.

Next, we show that  $M(A_{ij})$ $(i=1,\dots,2^{n-2})$ created from $M(B_j)$
$(j=2,\dots,n-k+(n-k-4))$
is diffeomorphic to one of the real Bott manifolds corresponding to
Bott matrices in \eqref{star1} or \eqref{star2}.

For brevity we can consider 
\begin{align}\label{B1}
A_{\ell 1}=
{
\left (\begin{array}{c|ccccccc}
 1& 1 & 0 & \{\hat 1\}_{4_{(1)}} & \dots & \dots& \dots& \{\hat 1\}_{n_{(1)}} \\
\hline
\mbox{\Large$0$} &  &   &  & \mbox{\large$B_1$} & & &
\end{array}\right)}
\end{align}
($\ell=1,\dots,n-2,1',\dots,k',(k+1)'$)
 representing each $A_{j'1}$ $(j=1,\dots,k,(k+1))$ in \eqref{star2},
  and each $A_{i1}$ $(i=1,\dots,n-2)$ in \eqref{star1}, where $A_{i1}$ $(i=2,\dots,n-2)$
  is equivalent to \eqref{B1}
by the equivariant diffeomorphism $\varphi \colon (\ZZ_2,M(B_{1}))\to (\ZZ_2,M(B_{1}))$
which is defined by
$\varphi([z_2,\dots,z_{n}])=[{\bf i}z_2,\dots,z_{n}]$.
Note that $\{\hat 1\}_{l_{(j)}}$ means $\hat 1(\in\{0,1\})$ in the $l$-th spot where
the corresponding Bott matrix is created from $B_j$.

Now we define an equivariant diffeomorphism
$\varphi \colon(\ZZ_2,M(B_{1}))\to (\ZZ_2,$ $M(B_{2}))$ by
$\varphi([z_2,\dots,z_{n}])=[z_2,{\bf i}z_3,z_4,\dots,z_{n}]$, then
(\ref{B1}) is equivalent to 
\begin{align}\label{B2}
A_{\ell 2}=
{
\left (\begin{array}{c|ccccccc}
 1& 1 & 0 & \{\hat 1\}_{4_{(2)}} & \dots & \dots& \dots& \{\hat 1\}_{n_{(2)}} \\
\hline
\mbox{\Large$0$} &  &   &  & \mbox{\large$B_2$} & & &
\end{array}\right)}.
\end{align}
Next we define an equivariant diffeomorphism $\varphi \colon(\ZZ_2,M(B_{2}))\to (\ZZ_2,$  $M(B_{3}))$
by $\varphi([z_2,\dots,z_{n}])=[z_2,z_3,{\bf i}z_4,z_5,\dots,z_{n}]$ so that
Bott matrix \eqref{B2} is equivalent to
\begin{align*}
A_{\ell 3}=
{
\left (\begin{array}{c|cccccccc}
 1& 1 & 0 & \{\hat 1\}_{4_{(3)}} & \dots & \dots& \dots& \dots &\{\hat 1\}_{n_{(3)}} \\
\hline
\mbox{\Large$0$} &  &   &  & \mbox{\large$B_3$} & & &
\end{array}\right)}.
\end{align*}

In general, let us consider 
\begin{align}\label{Bj-1}
\begin{split}
&A_{\ell (j-1)}=\\
&{
\left (\begin{array}{c|cccccccccc}
 1& 1 & 0 & \{\hat 1\}_{4_{(j-1)}} & \dots & \dots& \dots& \dots& \dots& \dots &\{\hat 1\}_{n_{(j-1)}} \\
\hline
 & 1 & 1 & 0& \dots & \dots& \dots& \dots& \dots& \dots& 0 \\
 &   & 1 & 1 & 0& \dots & \dots& \dots& \dots& \dots& 0 \\
 &   &   & \ddots &   & &  & & & & \vdots\\
 &   &   & & 1 & 1& 0 & \dots& \dots& \dots& 0  \\
 &   &   & &   & 1& 1     & \dots& \dots&  \dots&1  \\
\mbox{\Large$0$} &  &   & &   &  & 1     & 1 & \dots&  \dots&1  \\
 &  &   & \mbox{\Large$0$} &   & & \ddots & & & & \vdots\\
 &  &   & &   & &        & 1 & 1 & \dots &  1\\
 &  &   & &   & &   &\\
 &  &   & &  & & & &  & \mbox{\large$I_k$}   &  \\
 &  &
\end{array}\right)
\begin{array}{c}
g_1\\
g_2\\
g_3\\
\vdots\\
g_{j-1}\\
g_{j}\\
g_{j+1}\\
\vdots\\
g_{n-k}\\
\\
\\
\\
\end{array}}
\end{split}
\end{align}
$(j=2,\dots,n-k-1)$.
Defining an equivariant diffeomorphism
$\varphi \colon(\ZZ_2,M(B_{j-1}))\to (\ZZ_2,M(B_{j}))$ by
$\varphi([z_2,$ $\dots,z_{n}])=[z_2,\dots,z_{j},{\bf i}z_{j+1},z_{j+2},\dots,z_{n}]$, 
we obtain that \eqref{Bj-1} is equivalent to 

\begin{align*}
&A_{\ell(j)}=
{
\left (\begin{array}{c|cccccccccc}
 1& 1 & 0 & \{\hat 1\}_{4_{(j)}} & \dots & \dots& \dots& \dots& \dots& \dots &\{\hat 1\}_{n_{(j)}} \\
\hline
 & 1 & 1 & 0& \dots & \dots& \dots& \dots& \dots& \dots& 0 \\
 &   & 1 & 1 & 0& \dots & \dots& \dots& \dots& \dots& 0 \\
 &   &   & \ddots &   & &  & & & & \vdots\\
 &   &   & & 1 & 1& 0 & \dots& \dots& \dots& 0  \\
 &   &   & &   & 1& 1     & 0 & \dots&  \dots&0  \\
\mbox{\Large$0$} &  &   & &   &  & 1     & 1 & \dots&  \dots&1  \\
 &  &   & \mbox{\Large$0$} &   & & \ddots & & & & \vdots\\
 &  &   & &   & &        & 1 & 1 & \dots &  1\\
 &  &   & &   & &   &\\
 &  &   & &  & & & &  & \mbox{\large$I_k$}   &  \\
 &  &
\end{array}\right)
\begin{array}{c}
g_1\\
g_2\\
g_3\\
\vdots\\
g_{j-1}\\
g_{j}\\
g_{j+1}\\
\vdots\\
g_{n-k}\\
\\
\\
\\
\end{array}}.
\end{align*}
Therefore the previous Bott matrix is equivalent to
\begin{align}\label{i(n-k-1)}
\begin{split}
&A_{\ell(n-k-1)}=\\
&{
\left (\begin{array}{c|cccccccccccc}
 1& 1 & 0 & \{\hat 1\}_{4_{(n-k-1)}} & \dots & \dots& \dots& \dots & \dots & \dots &\{\hat 1\}_{n_{(n-k-1)}} \\
\hline
\mbox{\Large$0$} &  &   &  & \mbox{\large$B_{n-k-1}$} & & &\\
\end{array}\right)}.
\end{split}
\end{align}

Next,  \eqref{i(n-k-1)}  is equivalent to the following one by
the equivariant diffeomorphism
$\varphi \colon(\ZZ_2,M(B_{n-k-1}))$ $\to (\ZZ_2,M(B_{n-k}))$ defined by
$\varphi([z_2,\dots,z_{n}])=[z_2,\dots,z_{n-k-2},$ ${\bf i}z_{n-k-1},z_{n-k},\dots,z_{n}]$
\begin{align*}
A_{\ell(n-k)}=
{
\left (\begin{array}{c|cccccccccccc}
 1& 1 & 0 & \{\hat 1\}_{4_{(n-k)}} & \dots & \dots& \dots& \dots & \dots & \dots &\{\hat 1\}_{n_{(n-k)}} \\
\hline
\mbox{\Large$0$} &  &   &  & \mbox{\large$B_{n-k}$} & & &\\
\end{array}\right)}.
\end{align*}

In general, let us consider the following Bott matrix
\begin{align}\label{Bj-1b}
\begin{split}
&A_{\ell a}=\\
&{
\left (\begin{array}{c|ccccccccccc}
 1& 1 & 0 & \{\hat 1\}_{4_{(a)}} & \dots & \dots& \dots& \dots & \dots & \dots & \dots &\{\hat 1\}_{n_{(a)}} \\
\hline
 & 1 & 1 & 0& \dots & \dots& \dots& \dots& \dots & \dots & \dots & 0 \\
 &   & 1 & 1 & 0& \dots & \dots& \dots& \dots& \dots  & \dots & 0 \\
 &   &   & \ddots &   & &  & & & &  &\vdots\\
 &   &   &        & 1 & 1& 0 & \dots& 0 & 0& \dots& 0  \\
 &   &   &        &   & 1& 1     & \dots& 1 & 0 & \dots&0  \\
\mbox{\Large$0$} & &  &  &   &  & 1   &\dots  & 1 & 0 &  \dots& 0  \\
 &  &    &       & & & \ddots & & & & & \vdots\\
 &  &    &\mbox{\Large$0$} & & &        & 1 & 1 & 0 & \dots &  0\\
 &  &    &       & & &        & & 1 & 1 & \dots &  1\\
 &  &   & &   &\\
 &  &   & & & & & & & & \mbox{\large$I_k$}   &  \\
 &  &
\end{array}\right)
\begin{array}{c}
g_1\\
g_2\\
g_3\\
\vdots\\
g_{j-1}\\
g_{j}\\
g_{j+1}\\
\vdots\\
g_{n-k-1}\\
g_{n-k}\\
\\
\\
\\
\end{array}}
\end{split}
\end{align}
where $a=n-k+(n-k-3)-(j-1)$ $(j=n-k-1,n-k-2,\dots,4,3)$. Defining an equivariant diffeomorphism
$\varphi \colon(\ZZ_2,M(B_{a}))$ $\to (\ZZ_2,M(B_{b}))$ by
$\varphi([z_2,\dots,z_{n}])=[z_2,\dots,z_{j-1},{\bf i}z_{j},$ $z_{j+1},\dots,z_{n}]$, 
we obtain that \eqref{Bj-1b} is equivalent to 
\begin{align*}
&A_{\ell b}=\\
&{
\left (\begin{array}{c|ccccccccccc}
 1& 1 & 0 & \{\hat 1\}_{4_{(b)}} & \dots & \dots& \dots& \dots & \dots & \dots & \dots &\{\hat 1\}_{n_{(b)}} \\
\hline
 & 1 & 1 & 0& \dots & \dots& \dots& \dots& \dots & \dots & \dots & 0 \\
 &   & 1 & 1 & 0& \dots & \dots& \dots& \dots& \dots  & \dots & 0 \\
 &   &   & \ddots &   & &  & & & &  &\vdots\\
 &   &   &        & 1 & 1& 0 & \dots& 0 & 0& \dots& 0  \\
 &   &   &        &   & 1& 1     & \dots& 1 & 0 & \dots&0  \\
\mbox{\Large$0$} & &  &  &   &  & 1   &\dots  & 1 & 0 &  \dots& 0  \\
 &  &    &       & & & \ddots & & & & & \vdots\\
 &  &    &\mbox{\Large$0$} & & &        & 1 & 1 & 0 & \dots &  0\\
 &  &    &       & & &        & & 1 & 1 & \dots &  1\\
 &  &   & &   &\\
 &  &   & & & & & & & & \mbox{\large$I_k$}   &  \\
 &  &
\end{array}\right)
\begin{array}{c}
g_1\\
g_2\\
g_3\\
\vdots\\
g_{j-2}\\
g_{j-1}\\
g_{j}\\
\vdots\\
g_{n-k-1}\\
g_{n-k}\\
\\
\\
\\
\end{array}}
\end{align*}
where $b=n-k+(n-k-3)-(j-2)$.
Therefore the previous Bott matrix is equivalent to
\begin{align*}
A_{\ell c}=
{
\left (\begin{array}{c|cccccccc}
 1&1 & 0 & \{\hat 1\}_{4_{(c)}} & \dots & \{\hat 1\}_{n-k_{(c)}} & \{\hat 1\}_{n-k+1_{(c)}} & \dots& \{\hat 1\}_{n_{(c)}} \\
\hline
\mbox{\Large$0$} &  &   &  & \mbox{\large$B_{2(n-k)-4}$} & & &
\end{array}\right)}
\end{align*}
where $c=n-k+(n-k-3)-1$.

Finally each Bott matrix
\begin{align*}
{
\left (\begin{array}{c|ccccc}
1 & 1 & 1 & \{\hat 1\}_{4_{(j)}} & \dots & \{\hat 1\}_{n_{(j)}}\\
\hline
\mbox{\Large$0$}& & & \mbox{\large$B_j$}\\
\end{array}\right)}
\end{align*}
for $j=2,\dots,2(n-k)-4$
is equivalent to one of Bott matrices in \eqref{star1} or \eqref{star2}
by the equivariant diffeomorphism $\varphi([z_2,\dots,z_n])=[{\bf i}z_2,z_3,\dots,z_n]$. 
This completes the proof of the theorem.
\end{proof}

\noindent{\bf Acknowledgment}. I would like to thank Professor
Mikiya Masuda  for his useful suggestions.

\end{document}